\documentclass[11pt]{article}
	\pdfoutput=1
    \usepackage{url} 
    \usepackage{verbatim}
    \usepackage[titletoc]{appendix}    	
    \usepackage{caption3}
    \usepackage{subfigure}
    \usepackage{float}
         \usepackage{bm}
	\usepackage{amsmath}
	\usepackage{amsthm}
	\usepackage{amsfonts}
	\usepackage{multirow}
    \usepackage{nicefrac}
    \usepackage{longtable}
    \usepackage{color}
    \usepackage{graphicx, amssymb,graphics}
  \usepackage{fancyhdr}
  
\usepackage{cite}
\allowdisplaybreaks[4]
\usepackage{hyperref}

\usepackage{indentfirst} 
\newtheorem{theorem}{Theorem}[section]
\newtheorem{lemma}[theorem]{Lemma}
\newtheorem{remark}[theorem]{Remark}
\newcommand{\xqedhere}[2]{%
\rlap{\hbox to#1{\hfil\llap{\ensuremath{#2}}}}}

\makeatletter
\@addtoreset{equation}{section}
\makeatother

\newcommand{\R}{{\if mm {\rm I}\mkern -3mu{\rm R}\else \leavevmode
\hbox{I}\kern -.17em\hbox{R} \fi}}

\theoremstyle{plain}

\newtheorem{prop}{Proposition}[section]

\newcommand{\eps}{\epsilon}

\newcommand{\bX}{{\mathbf X}}

\newcommand{\Grad}[1]{\nabla #1}

		\title{ A new flow dynamic approach  for Wasserstein gradient flows }

\date{}
	
\begin{document}

\author{
		Qing Cheng\thanks{Key Laboratory of Intelligent Computing and Application (Tongji University), Ministry of Education, Department of Mathematics; Tongji University, Shanghai, China 200092 ({\tt qingcheng@tongji.edu.cn})}
		\and
		Qianqian Liu\thanks{School of Mathematical Sciences; Fudan University, Shanghai, China 200433 ({\tt qianqianliu21@m.fudan.edu.cn})} 
		\and 
		Wenbin Chen\thanks{
		Shanghai Key Laboratory for Contemporary Applied Mathematics, School of Mathematical Sciences; Fudan University, Shanghai, China 200433 ({\tt wbchen@fudan.edu.cn})} 
		\and
		Jie Shen\thanks{Eastern Institute of Technology, Ningbo, Zhejiang, China 315200 ({\tt jshen@eitech.edu.cn})} 
}
\footnotetext[1]{Corresponding author: Qianqian Liu.}
	\maketitle
	\begin{abstract}{
	We  develop in this paper a new regularized flow dynamic approach to construct efficient numerical schemes for Wasserstein gradient flows in Lagrangian coordinates.  Instead of approximating the Wasserstein distance  which needs to solve constrained minimization problems, we reformulate the problem using  the Benamou-Brenier's  flow dynamic approach, leading to  algorithms which  only need to solve unconstrained minimization problem in $L^2$ distance.  
	 Our schemes  automatically inherit  some essential properties of  Wasserstein gradient systems such as positivity-preserving, mass conservative and energy dissipation.  We present ample numerical simulations of Porous-Medium equations, Keller-Segel equations and Aggregation equations to validate the accuracy and stability of the proposed schemes.  Compared to numerical schemes in Eulerian coordinates,  our new schemes can capture sharp interfaces for  various Wasserstein gradient flows using relatively smaller number of unknowns. 
	
	  \medskip
	\noindent{\bfseries Keywords}: Wasserstein gradient flow;  Energetic variational approach; Lagrangian coordinates; Structure preservation
	
	\medskip
	\noindent{\bfseries Mathematics Subject Classification:}  65M06, 65M12, 35K65, 35A15
		}
	\end{abstract}

	\section{Introduction}
	Equations that are gradient flows in the Wasserstein metric  arise in many physical and biological applications, including  Porous-Medium equations \cite{vazquez2007porous,aronson2006porous},  Poisson-Nernst-Planck equations \cite{eisenberg2010energy,eisenberg1996computing} and Keller-Segel equations \cite{horstmann20031970,winkler2010aggregation,keller1970initiation}.
In this paper, we  consider numerical approximation of Wasserstein gradient flows which	 take the following form \cite{ambrosio2005gradient}:
	\begin{align}\label{eq:wd}
	\partial_t\rho=-\nabla\cdot(\rho {\bm v}), \qquad {\bm v}=-\nabla\frac{\delta E}{\delta \rho},
	\end{align}
	with the initial value $\rho({\bm x},0)=\rho_0({\bm x})\ge 0$, where $\rho({\bm x},t)$ is the particle density on the domain $\Omega\subset\mathbb{R}^d\ (d\ge 1)$, and ${\bm v}$ is the velocity field of the transport equation, 
and the energy functional is given by \cite{carrillo2003kinetic}
\begin{align}
	E(\rho)=\int_{\Omega}F(\rho)\mathrm{d}{\bm x}=\int_{\Omega}U(\rho({\bm x}))+V({\bm x})\rho({\bm x})\mathrm{d}{\bm x}+\frac{1}{2}\int_{\Omega\times\Omega}W({\bm x}-{\bm y})\rho({\bm x})\rho({\bm y})\mathrm{d}{\bm x}\mathrm{d}{\bm y},
\end{align}
where $F(\rho)$ is the energy density, $U(\rho)$ is the internal energy which can be taken as  $U_m(s)=s\log s$ for $m=1$ and $U_m(s)=\frac{s^m}{m-1}$ for $m>1$, $V({\bm x})$ is a drift potential and $W({\bm x},{\bm y})=W({\bm y},{\bm x})$ is an interaction potential, see also \cite{benamou2016augmented,carrillo2021lagrangian}. 
Solutions of the Wasserstein gradient flows possesses three essential properties: positivity-preserving; mass conservation; and the energy  dissipation in the sense that 
\begin{align}
\frac{\mathrm{d}}{\mathrm{d}t}E(\rho)(t)=-\int_{\Omega}\rho |{\bm v}|^2\mathrm{d}{\bm x}.
\end{align}

In recent years, considerable attention have been devoted to  constructing structure-preserving schemes  of various Wasserstein gradient flows \eqref{eq:wd}, \ e.g, Porous-Medium equation \cite{liu2023envara,duan2019pme,liu2020lagrangian,zhang2009numerical,ngo2017study}, Fokker-Planck equation \cite{duan2021structure,jordan1998variational,liu2004numerical,liu2016entropy,pareschi2018structure}, Keller-Segel equation \cite{wang2022fully,chen2022error,shen2020unconditionally,liu2018positivity,filbet2006finite}, drift diffusion and aggregation equation \cite{carrillo2022primal}.  Specially, numerical methods based on the approximation for the Wasserstein metric have attracted much more attention, see  \cite{carrillo2015finite,carrillo2022primal,carrillo2018lagrangian,li2020fisher,benamou2016augmented}. 
For example,  the celebrated   Jordan-Kinderlehrer-Otto (JKO) scheme \cite{jordan1998variational} is proposed to solve Fokker-Planck equation in \cite{jordan1998variational}. The scheme preserves essential properties of the Wasserstein gradient flows, but a key  difficulty of its numerical implementation is to approximate the Wasserstein metric efficiently.  Benamou and Brenier discovered   \cite{benamou2000computational} that the Wasserstein metric can be  rewritten  into the  Brenier formulae \cite{benamou2000computational}. 

\begin{figure}[htbp!]
\centering
\includegraphics[width=0.8\textwidth,clip==]{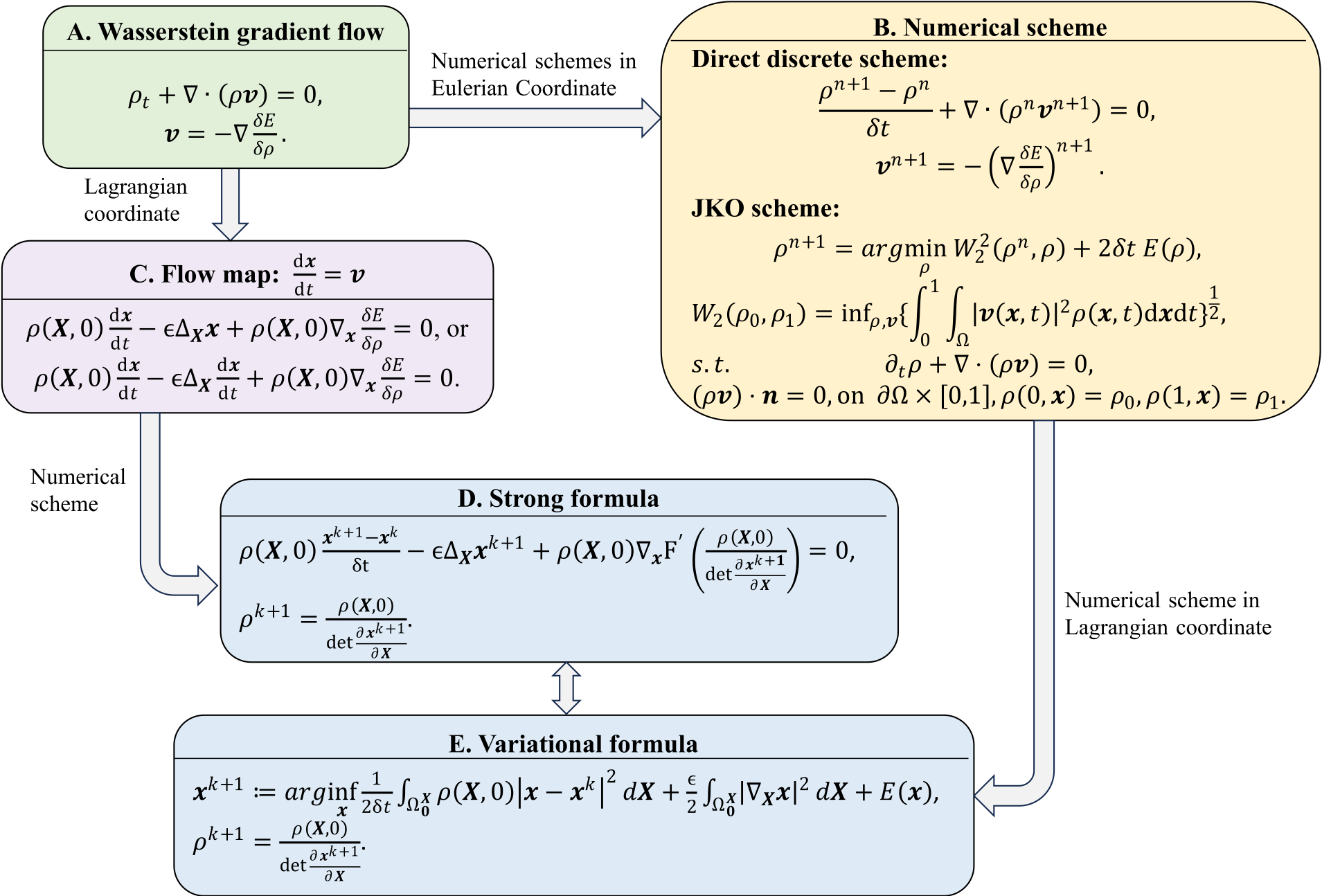}
	\caption{Relationship between the strong formula and the variational formula for Wasserstein gradient flows.}\label{fig:flow}
\end{figure}

 In the present work, we will adapt Benamou-Brenier's dynamic formulation for the Wasserstein metric \cite{benamou2000computational}. In particular, given two measures $\rho_0$ and $\rho_1$, their Wasserstein distance can be obtained by solving
\begin{align}\label{JKO:ori:1}
\begin{cases}
	&W_2(\rho_0,\rho_1)=\inf_{\rho,\bm{v}}\{\int_{0}^1\int_{\Omega}|{\bm v}({\bm x},t)|^2\rho({\bm x},t)\mathrm{d}{\bm x}\mathrm{d}t\}^{\frac{1}{2}},\\
&s.t. \qquad \partial_t\rho+\nabla\cdot(\rho \bm{v})=0,\\
&(\rho v)\cdot \bm{n}=0 \quad \text{on}\quad \partial\Omega\times[0,1],\quad \rho(0,{\bm x})=\rho_0,\ \rho(1,{\bm x})=\rho_1,
\end{cases}
\end{align}
where ${\bm n}$ is the outer unit normal on the boundary of the domain $\Omega$. Denote ${\bm m}=\rho{\bm v}$, then the original JKO scheme, proposed by Jordan, Kinderlehrer and Otto \cite{jordan1998variational},  can be reformulated into the following equivalent form \cite{li2020fisher,carrillo2022primal}: given $\rho^k$, solve $\rho^{k+1}=\rho(1,{\bm x})$ as
\begin{align}\label{JKO:ori}
\begin{cases}
	&(\rho,{\bm m})=\arg\inf_{\rho,\bm{v}}\frac{1}{2\delta t}\int_{0}^{1}\int_{\Omega}G(\rho,{\bm m})\mathrm{d}{\bm x}\mathrm{d}t+E(\rho(1,{\bm x})),\\
	&s.t. \qquad \partial_t\rho+\nabla\cdot(\rho \bm{v})=0,\\
	&(\rho{\bm v})\cdot \bm{n}=0 \quad \text{on}\quad \partial\Omega\times[0,1],\quad \rho(0,{\bm x})=\rho^k,
	\end{cases}
\end{align}
where $G(\rho,{\bm m})$ is defined by
\begin{align*}
G(\rho,{\bm m})=\begin{cases}
\frac{{\bm m}^2}{\rho},\qquad &\text{if } \rho>0,\\
0,\qquad &\text{if } (\rho,{\bm m})=(0,{\bm 0}),\\
+\infty,\qquad & \text{otherwise}.
\end{cases}
\end{align*}
Various numerical methods \cite{liu2023dynamic,kinderlehrer2017wasserstein,cances2020variational,carrillo2018lagrangian,carrillo2022primal,bonet2022efficient} are proposed based on the dynamic formulation \eqref{JKO:ori}. But they  usually require solving,  at each time step, a constrained minimization problem which can be difficult and costly. We adopt in this paper the flow dynamic approach  to develop a new class of numerical  schemes in Lagrangian coordinates  which only need to solve an unconstrained minimization problem at  each time step. The relationship between the strong formula and variational formula is shown in Figure~\ref{fig:flow}. The strong formula is obtained directly by discretizing the flow map equation, while the variational formula is derived by combining the JKO scheme and the flow map. It is noteworthy that the two forms are equivalent. This flow dynamic approach also enjoys the following advantages:
\begin{itemize}
\item It allows us to construct positivity-preserving schemes for Wasserstein gradient systems with positive solutions, and the schemes can also conserve mass   in both semi-discrete and fully-discrete cases.

\item Since the schemes are derived from  an energetic variational approach, they are energy dissipative  for Wasserstein gradient flows.
\item  The flow dynamic approach, behaving like a moving mesh method,  can automatically capture the trajectory of movement in Lagrangian coordinates, and as a consequence, requires fewer spatial points than an Eulerian approach to capture the interface movements or solutions with large gradients.

\end{itemize}
We present several numerical results to validate the proposed approach. In particular, 
For Porous-Medium equations, our numerical schemes can  accurately capture the sharp interface, and  obtain the correct finite propagation  speed and waiting time; and for the  Keller-Segel equations, they allow us to simulate the phenomenon of blow-up.

The rest of the paper is organized as follows. In Section \ref{sec:semi}, the semi-discrete numerical scheme is proposed for the Wasserstein gradient flow \eqref{eq:wd}. In Section \ref{sec:fully discrete}, fully discrete numerical schemes for the Wasserstein gradient flow \eqref{eq:wd} in 1D and 2D are constructed and analyzed, respectively. 
 Numerical experiments in 1D and 2D are carried out in Section \ref{sec:num} to validate the theoretical results.

\section{Regularized flow dynamic approach}\label{sec:semi}
As our objective is to develop numerical methods for Wasserstein gradient flows based on flow dynamic approach, we will first   introduce this approach \cite{cheng2020new,liu2020lagrangian,wang2022some} and then apply it to Wasserstein gradient flows in the semi-discrete case.  

Given an initial position or a reference configuration ${\bm X}$, and a velocity field ${\bm v}$, define the flow map ${\bm x}({\bm X},t)$ \cite{cheng2020new,duan2021structure}  as follows:
\begin{align}
&	\frac{\mathrm{d}{\bm x}({\bm X},t)}{\mathrm{d} t}={\bm v}({\bm x}({\bm X},t),t),\label{flow}\\
&	{\bm x}({\bm X},0)={\bm X},\label{flow1}
\end{align}
where ${\bm x}$ is the Eulerian coordinate and ${\bm X}$ is the Lagrangian coordinate, $\frac{\partial {\bm x}}{\partial {\bm X}}$ represents the deformation associated with the flow map.
We assume that  $\bm v$ is the velocity such that
\begin{equation}\label{trans}
\partial_t\rho + \nabla\cdot(\rho {\bm v}) =0 . 
\end{equation}

\begin{figure}[htbp!]
\centering
\includegraphics[width=0.45\textwidth,clip==]{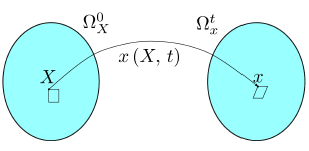}
\caption{A schematic illustration of a flow map $\bm x(\bm X,t)$ at a fixed time $t$: $\bm x(\bm X,t)$ maps $\Omega_0^{\bm X}$ to $\Omega_t^{\bm x}$. $\bm X$ is the Lagrangian coordinate while $\bm x$ is  the Eulerian coordinate, and  $F(\bm X,t)=\frac{\partial \bm x(\bm X,t)}{\partial \bm X}$ represents the deformation associated with the flow map. }\label{flow_map}
\end{figure}

Then the transport equation \eqref{trans} and the flow map \eqref{flow}-\eqref{flow1} determine the following kinematic relationship between the Eulerian and Lagrangian coordinates:
\begin{equation}\label{equal}
\frac{\rho({\bm X},0)}{\mathrm{det}\frac{\partial {\bm x}}{\partial {\bm X}}}= \rho(\bm x,t). 
\end{equation}

According to the Energetic Variational Approach \cite{eisenberg2010energy,liu2019energetic,onsager1931reciprocal,onsager1931reciprocal2,rayleigh1873some,liu2020variational}, we have
\begin{equation}\label{force}
F_{\text{conservative}}=- \frac{\delta E}{\delta {\bm x}} = \rho {\bm v} = F_{\text{dissipative}}.
\end{equation}
Combining \eqref{force} with Wasserstein gradient system \eqref{eq:wd}, we derive the following equivalent formulation:
\begin{equation}\label{force:2}
\rho {\bm v}=\rho\frac{\mathrm{d} {\bm x}}{\mathrm{d}t}=- \frac{\delta E}{\delta {\bm x}} = - \rho \Grad_{\bm x} \frac{\delta E}{\delta \rho},
\end{equation}
which can be written in Lagrangian coordinates as
\begin{equation}\label{balance}
 \rho({\bm X},0)\frac{\mathrm{d} {\bm x}}{\mathrm{d}t}   + \rho({\bm X},0)\Grad_{\bm x}\frac{\delta E}{\delta {\bm \rho}}  = 0.
\end{equation}

Notice that equation \eqref{balance} is a  fully nonlinear equation with respect to ${\bm x}$ in Lagrangian coordinate. 
By adding an extra regularizing term into \eqref{balance}, we obtain
 \begin{equation}\label{vis:force}
 \rho({\bm X},0)\frac{\mathrm{d} {\bm x}}{\mathrm{d}t}-\eps \Delta_{\bX} {\bm x}   +\rho(\bX,0)\Grad_{\bm x}\frac{\delta E}{\delta{\bm \rho}}  =0,
\end{equation}
where $\eps\ll 1$ is a regularization parameter. Taking the inner product of \eqref{vis:force} with ${\bm x}_t$, we obtain the following energy dissipative law: 
\begin{equation}\label{en:dis:x}
\frac{\mathrm{d}}{\mathrm{d}t} \tilde{E}(t) = -\int_{\Omega_0^{\bm X}}  \rho({\bm X},0)|\frac{\mathrm{d} {\bm x}}{\mathrm{d}t}|^2\mathrm{d}{\bm X},
\end{equation}
where $\tilde E$ is defined by $\tilde E(t) = E(t) +  \frac{\eps}{2}\int_{\Omega_0^{\bm X}}|\Grad_{\bX} {\bm x}|^2\mathrm{d}{\bm X} $.
\begin{remark}
Another regularized approach is
	 \begin{equation}\label{vis:force2}
		\rho({\bm X},0)\frac{\mathrm{d} {\bm x}}{\mathrm{d}t}-\eps \Delta_{\bX} \frac{\mathrm{d} {\bm x}}{\mathrm{d}t}   +\rho(\bX,0)\Grad_{\bm x}\frac{\delta E}{\delta{\bm \rho}}  =0,
	\end{equation}
	which satisfies the following energy dissipative law:
	\begin{equation*}
\frac{\mathrm{d}}{\mathrm{d}t}  E(t) =-\int_{\Omega_0^{\bm X}}  \rho({\bm X},0)|\frac{\mathrm{d} {\bm x}}{\mathrm{d}t}|^2\mathrm{d}{\bm X}  -\int_{\Omega_0^{\bm X}}\epsilon |\Grad_X \frac{\mathrm{d} {\bm x}}{\mathrm{d}t}|^2 \mathrm{d}\bX \le 0.
\end{equation*}
This  regularized approach can be used to construct second-order schemes for  \eqref{balance}.
\end{remark}

\subsection{Numerical schemes based on the FDM approach}
We observed from \eqref{en:dis:x} that the flow map ${\bm x}$ also satisfies a gradient flow structure in Lagrangian coordinate. A benefit of the gradient flow structure  is that  one can construct numerical schemes to inherit the energy dissipation of the  gradient flows.  In contrast to the classic JKO scheme \cite{jordan1998variational}, we shall construct numerical schemes based on FDM approach which do not involve the calculation of the Wasserstein metric, but only of a $L^2$ distance which is much   easier to obtain. 

\subsubsection{First-order scheme}
Given $\delta t>0$, $t^{k}=k\delta t$ for $k=0$, $1$, $\cdots$, $\frac{T}{\delta t}$, let ${\bm x}^k$ denote the numerical approximation   to ${\bm x}(\cdot, t^k)$.  
 Then, a first-order scheme  in a variational form to the solution of \eqref{vis:force} at $t^{k+1}$ can be obtained by  
\begin{align}\label{JKO:new}
\begin{cases}
&{\bm x}^{k+1}:=\arg\inf_{{\bm x}}\frac{1}{2\delta t}\int_{\Omega_0^{\bm X}}\rho({\bm X},0) |{\bm x}-{\bm x}^{k}|^2\mathrm{d}{\bm X}+\frac{\epsilon}{2}\int_{\Omega_0^{\bm X}} |\Grad_{\bX} {\bm x}|^2 \mathrm{d}\bX +E(\bm x),\\
&\rho^{k+1} = \frac{\rho({\bm X},0)}{\mathrm{det}\frac{\partial {\bm x^{k+1}}}{\partial {\bm X}}} . 
	\end{cases}
\end{align}


Motivated from the minimization problem \eqref{JKO:new}, 
a first-order semi-discrete  scheme for \eqref{eq:wd}  based on \eqref{vis:force} with $\epsilon=\delta t$ and \eqref{JKO:new} is:
\begin{align}
		&\rho({\bm X},0)\frac{{\bm x}^{k+1}-{\bm x}^k}{\delta t}-\delta t\Delta_{\bm X} {\bm x}^{k+1}+\rho({\bm X},0)\nabla_{\bm x}F'\left(\frac{\rho({\bm X},0)}{\mathrm{det}\frac{\partial {\bm x}^{k+1}}{\partial {\bm X}}}\right)=0,\label{eq:jko_1}\\
		&\rho^{k+1}=\frac{\rho({\bm X},0)}{\mathrm{det}\frac{\partial {\bm x}^{k+1}}{\partial {\bm X}}}, \label{eq:jko_11}
	\end{align}
	with the Dirichlet boundary condition ${\bm x}^{k+1}|_{\partial\Omega}={\bm X}|_{\partial\Omega}$, and $\Omega_0^{\bm x}=\Omega_0^{\bm X}$. 


We will show that  the numerical scheme \eqref{eq:jko_1}-\eqref{eq:jko_11} in Lagrangian coordinate achieves a first-order accuracy to \eqref{eq:wd} in Eulerian coordinate: 
\begin{prop}\label{prop1}
 The solution $\rho^{k+1}$ to numerical scheme \eqref{eq:jko_1}-\eqref{eq:jko_11} is  a first-order approximation to the exact solution of Wasserstein gradient flow \eqref{eq:wd} at $t^{k+1}$.
\end{prop}
\begin{proof}
Firstly, we shall calculate the variational derivative of the energy $E$:
\begin{align}\label{var:de}
\delta E=&\int_{\Omega_0^{\bm X}}-F'\left(\frac{\rho({\bm X},0)}{\mathrm{det}\frac{\partial {\bm x}}{\partial {\bm X}}}\right)\frac{\rho({\bm X},0)}{\left(\mathrm{det}\frac{\partial {\bm x}}{\partial {\bm X}}\right)^2}\mathrm{det}\frac{\partial {\bm x}}{\partial {\bm X}}\frac{\partial \delta {\bm x}}{\partial {\bm X}}+F\left(\frac{\rho({\bm X},0)}{\mathrm{det}\frac{\partial {\bm x}}{\partial{\bm X}}}\right)\frac{\partial \delta {\bm x}}{\partial {\bm X}}\mathrm{d}{\bm X}\nonumber\\
	=&\int_{\Omega_0^{\bm X}}\partial_{\bm X}\left(F'\left(\frac{\rho({\bm X},0)}{\mathrm{det}\frac{\partial {\bm x}}{\partial {\bm X}}}\right)\frac{\rho({\bm X},0)}{\left(\mathrm{det}\frac{\partial {\bm x}}{\partial {\bm X}}\right)^2}\mathrm{det}\frac{\partial {\bm x}}{\partial X}-F\left(\frac{\rho({\bm X},0)}{\mathrm{det}\frac{\partial {\bm x}}{\partial {\bm X}}}\right)\right)\delta {\bm x}\mathrm{d}{\bm X}\nonumber\\
	=&\int_{\Omega_0^{\bm X}}\partial_{\bm X}\left(F'\left(\frac{\rho({\bm X},0)}{\mathrm{det}\frac{\partial {\bm x}}{\partial {\bm X}}}\right)\frac{\rho({\bm X},0)}{\mathrm{det}\frac{\partial {\bm x}}{\partial {\bm X}}}-F\left(\frac{\rho({\bm X},0)}{\mathrm{det}\frac{\partial {\bm x}}{\partial {\bm X}}}\right)\right)\delta {\bm x}\mathrm{d}{\bm X}\nonumber\\
	=&\int_{\Omega_0^{\bm X}}\partial_{\bm x}\left(F'\left(\rho\right)\rho-F\left(\rho\right)\right)\delta {\bm x}\mathrm{d}{\bm x}=\int_{\Omega_0^{\bm X}}\rho\partial_{\bm x}F'(\rho)\delta {\bm x}\mathrm{d}{\bm x}.
\end{align}
We can then derive from \eqref{var:de}  the variational derivatives of $E$ in  Eulerian coordinates and Lagrangian coordinates, respectively:
\begin{align}
&	\frac{\delta E}{\delta {\bm x}}=\rho\partial_{\bm x}F'(\rho),\\
&	
\frac{\delta E}{\delta {\bm x}}=\partial_{\bm X}\left(F'\left(\frac{\rho({\bm X},0)}{\mathrm{det}\frac{\partial {\bm x}}{\partial {\bm X}}}\right)\frac{\rho({\bm X},0)}{\mathrm{det}\frac{\partial {\bm x}}{\partial {\bm X}}}-F\left(\frac{\rho({\bm X},0)}{\mathrm{det}\frac{\partial {\bm x}}{\partial {\bm X}}}\right)\right)=\rho({\bm X},0)\frac{\partial_{\bm X}F'\left(\frac{\rho({\bm X},0)}{\mathrm{det}\frac{\partial {\bm x}}{\partial {\bm X}}}\right)}{\mathrm{det}\frac{\partial {\bm x}}{\partial {\bm X}}}.\label{eq:delta E}
\end{align}

	From \eqref{eq:delta E}, we observe that the last term in \eqref{eq:jko_1} is exactly the variational derivative of $E(\frac{\rho({\bm X},0)}{\mathrm{det}\frac{\partial {\bm x}}{\partial {\bm X}}  })$ with respective to ${\bm x}$ in Lagrangian coordinates. Then the equation \eqref{eq:jko_1} can be obtained by taking  the variational derivative of \eqref{JKO:new} with respect to ${\bm x}$, which also show that  the minimizer $({\bm x}^{k+1},\rho^{k+1})$ of the variational problem \eqref{JKO:new} is the solution of \eqref{eq:jko_1}-\eqref{eq:jko_11}. 
	Assuming that $\rho(\bX,0) \neq 0$ and rewriting the equation \eqref{eq:jko_1} into the following equivalent form:
	\begin{align}\label{dis:e}
		\frac{{\bm x}^{k+1}-{\bm x}^k}{\delta t}=\frac{\delta t}{\rho({\bm X},0)}\Delta_{\bm X} {\bm x}^{k+1}-\nabla_{\bm x}F'\left(\frac{\rho({\bm X},0)}{\mathrm{det}\frac{\partial {\bm x}^{k+1}}{\partial {\bm X}}}\right).
	\end{align}
Notice  the definition of flow map ${\bm v}({\bm X},t^{k+1}) =\frac{\mathrm{d}{\bm x}({\bm X},t^{k+1})}{\mathrm{d}t}$, we can easily derive from \eqref{dis:e} that
	\begin{align*}
		{\bm v}^{k+1}=-\nabla_{\bm x}F'(\rho({\bm x}^{k+1}))+O(\delta t)=-\nabla_{\bm x}(\frac{\delta E}{\delta \rho })^{k+1}+O(\delta t).
	\end{align*}
	Then we conclude that the numerical scheme \eqref{eq:jko_1} achieves a first-order accuracy in time for the Wasserstein gradient flow \eqref{eq:wd}.
\end{proof}
\begin{remark}
	It is worth mentioning that there are several choices to add the regularization term in scheme \eqref{JKO:new}, such as $\int_{\Omega_0^{\bm X}}\delta t\rho({\bm X},0)|\nabla_{\bm X} \bm{x}|^2\mathrm{d}{\bm X}$ and  $\int_{\Omega_0^{\bm X}}|\nabla_{\bm X}( \bm{x}-{\bm x}^k)|^2\mathrm{d}{\bm X}$, which can also be proved to be a first-order approximation to the Wasserstein gradient flow \eqref{eq:wd}. 
\end{remark}

\begin{remark}
	We can solve the  nonlinear  scheme \eqref{eq:jko_1}--\eqref{eq:jko_11} by using the damped Newton's iteration. Denote the linear and nonlinear operators by 
	\begin{align*}
		\mathcal{L}({\bm x}^{k+1}):=\frac{\rho({\bm X},0)}{\delta t}{\bm x}^{k+1}-\delta t\Delta_{\bm X}{\bm x}^{k+1},\qquad \mathcal{N}({\bm x}^{k+1}):=\rho({\bm X},0)\nabla_{\bm x}F'\left(\frac{\rho({\bm X},0)}{\mathrm{det}\frac{\partial {\bm x}^{k+1}}{\partial {\bm X}}}\right),
	\end{align*}
	and $b({\bm x}^k):=\frac{\rho({\bm X},0)}{\delta t}{\bm x}^{k}$,  then the scheme \eqref{eq:jko_1}  can be reformulated as a nonlinear system $$\mathcal{L}({\bm x}^{k+1})+\mathcal{N}({\bm x}^{k+1})=b({\bm x}^k),$$
	  which can be solved iteratively by
	finding ${\bm x}^{k+1,n+1}$, such that
	\begin{equation*}
	{\bm x}^{k+1,n+1}={\bm x}^{k+1,n} + \alpha \delta{\bm x},\quad\text{with}\quad \delta {\bm x} =\frac{b({\bm x}^{k})-\mathcal{L}({\bm x}^{k+1,n})-\mathcal{N}({\bm x}^{k+1,n})}{\mathcal{L}'({\bm x}^{k+1,n})+\mathcal{N}'({\bm x}^{k+1,n})}. 
	\end{equation*}
	where  $0 <\alpha<1$ is a damping coefficient.
\end{remark} 

\subsubsection{Second-order scheme}
In analogy to the first-order scheme \eqref{eq:jko_1}-\eqref{eq:jko_11}, we can construct a second-order scheme for \eqref{eq:wd} based on a Crank-Nicolson discretization of \eqref{vis:force2} with $\eps=\delta t^2$: 
\begin{align}\label{scheme:JKO:cn}
\begin{cases}
&\bm{x}^{k+1}=\arg\inf_{\bm{x}}\frac{1}{2\delta t}\int_{\Omega_0^{\bm X}}\rho({\bm X},0)|\bm{x}-\bm{x}^k|^2+\delta t^2(|\nabla_{\bm X} (\bm{x}- \bm{x}^k)|^2)\mathrm{d}{\bm X}\\
&\qquad\quad+\frac{1}{2}E\left(\frac{\rho({\bm X},0)}{\mathrm{det}\frac{\partial {\bm x}}{\partial {\bm X}}}\right)+\frac{1}{2}\int_{\Omega_0^{\bm X}}{\bm x}(\frac{\delta E}{\delta {\bm x}})^k\mathrm{d}{\bm X},\\
&\rho^{k+1}=\frac{\rho({\bm X},0)}{\mathrm{det}\frac{\partial {\bm x}^{k+1}}{\partial  {\bm X}}}.
\end{cases}
\end{align}

Next, we show that the above unconstrained minimization problem is equivalent to a Crank-Nicolson discretization.
\begin{theorem}
	For any $k>0$, the minimizer of the variational problem \eqref{scheme:JKO:cn} is the solution of the following second-order Crank-Nicolson scheme:  given $({\bm x}^k,\rho({\bm X},0))$, solve $({\bm x}^{k+1},\rho^{k+1})$ from
	\begin{align}
		&\rho({\bm X},0)\frac{{\bm x}^{k+1}-{\bm x}^k}{\delta t}-\delta t(\Delta_{\bm X} {\bm x}^{k+1}-\Delta_{\bm X} {\bm x}^{k})+\frac{1}{2}\rho({\bm X},0)\nabla_{\bm x}F'\left(\frac{\rho({\bm X},0)}{\mathrm{det}\frac{\partial {\bm x}^{k+1}}{\partial {\bm X}}}\right)\nonumber\\
		&\qquad\qquad\qquad\qquad\qquad\qquad\qquad\qquad+\frac{1}{2}\rho({\bm X},0)\nabla_{\bm x}F'\left(\frac{\rho({\bm X},0)}{\mathrm{det}\frac{\partial {\bm x}^{k}}{\partial {\bm X}}}\right)=0,\label{eq:jko_2}\\
				&\rho^{k+1}=\frac{\rho({\bm X},0)}{\mathrm{det}\frac{\partial {\bm x}^{k+1}}{\partial {\bm X}}}, \label{eq:jko_21}
	\end{align}
	with  the Dirichlet boundary condition ${\bm x}^{k+1}|_{\partial\Omega}={\bm X}|_{\partial\Omega}$. 
\end{theorem}
\begin{proof}
	 The equation \eqref{eq:jko_2} can be obtained by taking  the variational derivative of \eqref{scheme:JKO:cn} with respect to ${\bm x}$,  then the minimizer $({\bm x}^{k+1},\rho^{k+1})$ of the minimization  problem \eqref{scheme:JKO:cn} is the solution of \eqref{eq:jko_2} with \eqref{eq:jko_11}. The equation \eqref{eq:jko_2} can be rewritten into
	\begin{align}
			&\frac{{\bm x}^{k+1}-{\bm x}^k}{\delta t}-\frac{\delta t}{\rho({\bm X},0)}(\Delta_{\bm X} {\bm x}^{k+1}-\Delta_{\bm X} {\bm x}^{k})+\frac{1}{2}\nabla_{\bm x}F'\left(\frac{\rho({\bm X},0)}{\mathrm{det}\frac{\partial {\bm x}^{k+1}}{\partial {\bm X}}}\right)+\frac{1}{2}\nabla_{\bm x}F'\left(\frac{\rho({\bm X},0)}{\mathrm{det}\frac{\partial {\bm x}^{k}}{\partial {\bm X}}}\right)=0,\nonumber
	\end{align}
	which is equivalent to
	\begin{align*}
		{\bm v}^{k+\frac{1}{2}}=-\nabla_{\bm x}F'(\rho({\bm x}^{k+\frac{1}{2}}))+O(\delta t^2)=-\nabla_{\bm x}(\frac{\delta E}{\delta \rho})^{k+\frac{1}{2}}+O(\delta t^2),
	\end{align*}
	where the fact that the term $\frac{1}{2}\nabla_{\bm x}F'(\rho({\bm x}^{k}))+\frac{1}{2}\nabla_{\bm x}F'(\rho({\bm x}^{k+1}))$ is a second-order approximation to $\nabla_{\bm x}F'(\rho({\bm x}^{k+\frac{1}{2}}))$ has been utilized.
		Then the numerical scheme \eqref{eq:jko_2} is a second-order time discretization for the Wasserstein gradient flow \eqref{eq:wd}.
\end{proof}
Similarly as in the last subsection, we can also show that  $\rho^{k+1}$  obtained  from \eqref{scheme:JKO:cn} is a second-order approximation to the solution $\rho({\bm x},t)$ at $t=t^{k+1}$. 
\begin{prop}
	The solution $\rho^{k+1}$ to numerical scheme \eqref{eq:jko_1}-\eqref{eq:jko_11} is  a first-order approximation to the exact solution of Wasserstein gradient flow \eqref{eq:wd} at $t^{k+1}$.
\end{prop}
The proof is  analogous to that of Proposition \ref{prop1} so we omit it for the sake of brevity. 

 \begin{remark}
		Similarly as \eqref{eq:jko_1} , the nonlinear scheme \eqref{eq:jko_2}  can also be reformulated as a nonlinear system 
		$$\mathcal{L}({\bm x}^{k+1})+\frac{1}{2}\mathcal{N}({\bm x}^{k+1})=\mathcal{L}({\bm x}^k)-\frac{1}{2}\mathcal{N}({\bm x}^{k}),$$
		which can be solved by
		finding ${\bm x}^{k+1,n+1}$, such that
		\begin{equation*}
		{\bm x}^{k+1,n+1}={\bm x}^{k+1,n} + \alpha \delta{\bm x},\quad\text{with}\quad \delta {\bm x} =\frac{\mathcal{L}({\bm x}^k-{\bm x}^{k+1,n})-\frac{1}{2}\mathcal{N}({\bm x}^{k})-\frac{1}{2}\mathcal{N}({\bm x}^{k+1,n})}{\mathcal{L}'({\bm x}^{k+1,n})+\frac{1}{2}\mathcal{N}'({\bm x}^{k+1,n})}. 
		\end{equation*}
		where $0 <\alpha<1$ is a damping coefficient.
\end{remark} 
\subsection{Properties of numerical schemes based on  FDM approach}
Now   we show that our numerical schemes based on the FDM approach inherit essential properties of the Wasserstein gradient flows. 
More precisely, we have
\begin{theorem}
	For any $0\leq k\leq \frac{T}{\delta t}$, assume  initial value $\rho({\bm X},0)>0$  and consider the internal energy to be the form of $U(s)=s\log s$ in \eqref{JKO:new}, then the solution $\rho^{k+1}({\bm x})$ of numerical scheme \eqref{eq:jko_1}-\eqref{eq:jko_11} satisfies the following properties:
	\begin{itemize}
		\item The numerical solution $\rho^{k+1}>0$ is also positive.
		\item The scheme \eqref{eq:jko_1}-\eqref{eq:jko_11}   is mass conservative in the sense that
		\begin{align}
		\int_{\Omega_t^{\bm x}}\rho^{k+1}({\bm x})\mathrm{d}{\bm x}=\int_{\Omega_0^{\bm X}}\rho({\bm X},0)\mathrm{d}{\bm X}.
		\end{align} 
		\item It is unconditionally energy stable in the sense that
		\begin{align}
			E({\bm x}^{k+1})+\frac{\delta t}{2}\int_{\Omega_0^{\bm X}}|\nabla_{\bm X} {\bm x}^{k+1}|^2\mathrm{d}{\bm X}\le E({\bm x}^k)+\frac{\delta t}{2}\int_{\Omega_0^{\bm X}}|\nabla_{\bm X} {\bm x}^{k}|^2\mathrm{d}{\bm X}.
		\end{align}
	\end{itemize}
\end{theorem}
\begin{proof}
	
	Since we choose the internal energy $U(s)$ to  be the form of $s\log s$ and ${\bm x}^{k+1}$ is the minimizer of variational problem \eqref{JKO:new}, then $\frac{\rho({\bm X},0)}{\mathrm{det}\frac{\partial {\bm x}^{k+1}}{\partial {\bm X}}}$ should stay in the domain of the logarithmic function which implies $\rho^{k+1}>0$.
	
	Using the equality \eqref{eq:jko_11}, we derive
	\begin{align*}
	\int_{\Omega_t^{\bm x}}\frac{\rho({\bm X},0)}{\mathrm{det}\frac{\partial {\bm x}^{k+1}}{\partial {\bm X}}}\mathrm{d}{\bm x}=\int_{\Omega_0^{\bm X}}\frac{\rho({\bm X},0)}{\mathrm{det}\frac{\partial {\bm x}^{k+1}}{\partial {\bm X}}}\mathrm{det}\frac{\partial {\bm x}^{k+1}}{\partial {\bm X}}\mathrm{d}{\bm X}=\int_{\Omega_0^{\bm X}}\rho({\bm X},0)\mathrm{d}{\bm X},
	\end{align*}
	where the equality $\mathrm{det}\frac{\partial {\bm x}^{k+1}}{\partial {\bm X}}\mathrm{d}{\bm X}=\mathrm{d}{\bm x}$ has been used. Then the scheme \eqref{eq:jko_1}-\eqref{eq:jko_11} is mass conservative.
	
	From the minimization problem \eqref{JKO:new}, we obtain that
	\begin{align*}
		E({\bm x}^{k+1})+\frac{\delta t}{2}\int_{\Omega_0^{\bm X}}|\nabla_{\bm X} {\bm x}^{k+1}|^2\mathrm{d}{\bm X}\le E({\bm x}^k)+\frac{\delta t}{2}\int_{\Omega_0^{\bm X}}|\nabla_{\bm X} {\bm x}^{k}|^2\mathrm{d}{\bm X},
	\end{align*}
	the energy stability property is arrived.
\end{proof}

\begin{remark}
	If the regularization term in \eqref{JKO:new} is taken as $\int_{\Omega_0^{\bm X}}|\nabla_{\bm X}({\bm x}-{\bm x}^k) |^2\mathrm{d}{\bm X}$, the energy dissipation $E({\bm x}^{k+1})\le E({\bm x}^k)$ for any $0\le k\le\frac{T}{\delta t}$ can also be derived. 
\end{remark}
\begin{remark}
	Similarly to  the first-order scheme \eqref{eq:jko_1}-\eqref{eq:jko_11}, we can show that the second-order Crank-Nicolson scheme \eqref{eq:jko_2} is  mass conservative and its numerical solution $\rho^{k+1}$ is also positivity-preserving if the internal energy is chosen as a logarithm type function. But we are unable to prove that the Crank-Nicolson scheme \eqref{eq:jko_2} is energy dissipative.
\end{remark}
\section{Fully discretizations based on FDM approach}\label{sec:fully discrete}

In this section, we provide more details about the spatial discretizations and  propose full discrete schemes for the Wasserstein gradient flows \eqref{eq:wd} based on the FDM approach. To better explain our  approach, we first consider the discretization in one dimension, then we generalize our FDM approach to two-dimensional case. More specifically, we show that our fully discrete schemes are also   positivity-preserving, mass conservative and energy dissipative. 


\subsection{Numerical scheme in 1D}
In this subsection, we develop  an implicit finite difference scheme for the Wasserstein gradient flows \eqref{eq:wd}. For convenience, we define some notations, let $\Omega_0^X=[-L,L]$ and $\Omega_0^x=[-L,L]$ be the computational domain in Lagrangian and Eulerian coordinate and choose the spatial grids: $-L=X_0<X_1<\cdots<X_N=L$. We define 
\begin{align*}
	x_j(t_k)=x(X_j,t_k),\qquad 0\le j\le N,\ 0\le k\le \frac{T}{\delta t},\\
	\rho_j(t_k)=\rho(x_j,t_k),\qquad 0\le j\le N,\ 0\le k\le \frac{T}{\delta t},
\end{align*}
where $X_j=X_0+j\delta X$, $\delta X=X_{j+1}-X_j$.

For the following finite dimensional minimization problem, we define the admissible set $S_{ad}:=\{{\bm x}: \ x_{j+1}>x_{j}\ \text{for}\ j=0, 1, \cdots, N-1, \ \text{and} \ x_0=X_0,\ x_{N}=X_N\}$:
\begin{align}
\begin{cases}
&{\bm x}^{k+1}:=\arg\inf_{{\bm x}\in S_{ad}}\dfrac{1}{2\delta t}\sum\limits_{j=0}^{N-1}\rho(X_{j+\frac{1}{2}},0)|x_{j+\frac{1}{2}}-x_{j+\frac{1}{2}}^{k}|^2\delta X+ \dfrac{\epsilon}{2}\sum\limits_{j=0}^{N-1}|\dfrac{x_{j+1}-x_{j}}{\delta X}|^2\delta X\\
&\qquad\qquad\qquad\qquad\qquad+E_h({\bm x}),\label{eq:discre optim}\\
&\rho_{j+\frac{1}{2}}^{k+1}:=\dfrac{\rho(X_{j+\frac{1}{2}},0)}{\frac{x_{j+1}^{k+1}-x_{j}^{k+1}}{\delta X}},
\end{cases}
\end{align}
where ${\bm x}=(x_0,x_1,\ldots,x_{N-1},x_N)$, 
${x}_{j+\frac{1}{2}}:=\frac{1}{2}(x_{j+1}+x_{j})$, and the discrete energy is defined by
\begin{align}
	E_h({\bm x}):=\sum\limits_{j=0}^{N-1}F\left(\frac{\rho(X_{j+\frac{1}{2}},0)}{\frac{x_{j+1}-x_j}{\delta X}}\right)\frac{x_{j+1}-x_j}{\delta X}\delta X.
\end{align} 

\begin{lemma}
The minimizer of the variational problem \eqref{eq:discre optim} is the solution of the following numerical scheme with $\epsilon=\delta t$: given $({\bm x}^k,\rho(X,0))$, we solve $({\bm x}^{k+1},\rho^{k+1})$ from
\begin{align}
	&\frac{1}{2\delta t}\rho(X_{j+\frac{1}{2}},0)(x_{j+\frac{1}{2}}^{k+1}-x_{j+\frac{1}{2}}^{k})\delta X+	\frac{1}{2\delta t}\rho(X_{j-\frac{1}{2}},0)(x_{j-\frac{1}{2}}^{k+1}-x_{j-\frac{1}{2}}^{k})\delta X\nonumber\\
	&\qquad\qquad\qquad\qquad-\delta t\frac{x_{j+1}^{k+1}-2x_j^{k+1}+x_{j-1}^{k+1}}{(\delta X)^2}\delta X+\frac{\delta E_h}{\delta x_j}({\bm x}^{k+1})=0,\label{schem:1}\\
	&\rho_{j+\frac{1}{2}}^{k+1}:=\frac{\rho(X_{j+\frac{1}{2}},0)}{\frac{x_{j+1}^{k+1}-x_{j}^{k+1}}{\delta X}},\quad j=0,1\cdots, N-1,\label{schem:1-2}
\end{align}
with the initial and boundary conditions
\begin{align}
	{\bm x}^0=(X_0,X_1,\cdots,X_N)\quad\text{and}\quad x_{0}^{k+1}=X_0,\quad x_{N}^{k+1}=X_N.\label{schem:1-3}
\end{align}
\end{lemma}
The last term $\frac{\delta E_h}{\delta x_j}({\bm x}^{k+1})$ in \eqref{schem:1} is defined as follows:
\begin{align}
\frac{\delta E_h}{\delta x_j}({\bm x}^{k+1})=&F'\left(\frac{\rho(X_{j+\frac{1}{2}},0)}{\frac{x_{j+1}^{k+1}-x_j^{k+1}}{\delta X}}\right)\frac{\rho(X_{j+\frac{1}{2}},0)\delta X}{(x_{j+1}^{k+1}-x_j^{k+1})^2}\frac{x_{j+1}^{k+1}-x_j^{k+1}}{\delta X}\delta X-F\left(\frac{\rho(X_{j-\frac{1}{2}},0)}{\frac{x_{j+1}^{k+1}-x_j^{k+1}}{\delta X}}\right)\nonumber\\
&-F'\left(\frac{\rho(X_{j-\frac{1}{2}},0)}{\frac{x_{j}^{k+1}-x_{j-1}^{k+1}}{\delta X}}\right)\frac{\rho(X_{j-\frac{1}{2}},0)\delta X}{(x_{j}^{k+1}-x_{j-1}^{k+1})^2}\frac{x_{j}^{k+1}-x_{j-1}^{k+1}}{\delta X}\delta X+F\left(\frac{\rho(X_{j-\frac{1}{2}},0)}{\frac{x_{j}^{k+1}-x_{j-1}^{k+1}}{\delta X}}\right)\nonumber\\
=&F'\left(\frac{\rho(X_{j+\frac{1}{2}},0)}{\frac{x_{j+1}^{k+1}-x_j^{k+1}}{\delta X}}\right)\frac{\rho(X_{j+\frac{1}{2}},0)}{\frac{x_{j+1}^{k+1}-x_j^{k+1}}{\delta X}}-F\left(\frac{\rho(X_{j-\frac{1}{2}},0)}{\frac{x_{j+1}^{k+1}-x_j^{k+1}}{\delta X}}\right)\nonumber\\
&-F'\left(\frac{\rho(X_{j-\frac{1}{2}},0)}{\frac{x_{j}^{k+1}-x_{j-1}^{k+1}}{\delta X}}\right)\frac{\rho(X_{j-\frac{1}{2}},0)}{\frac{x_{j}^{k+1}-x_{j-1}^{k+1}}{\delta X}}+F\left(\frac{\rho(X_{j-\frac{1}{2}},0)}{\frac{x_{j}^{k+1}-x_{j-1}^{k+1}}{\delta X}}\right).\nonumber
\end{align}

Numerical scheme \eqref{schem:1}-\eqref{schem:1-2} can be applied to various Wasserstein gradient flows, for example:
\begin{itemize}
	\item For the linear Fokker-Planck equation with one-well energy potential:
	\begin{align}\label{eq:energy fp}
	E_h({\bm x}):=\sum_{j=0}^{N-1}\delta X\rho(X_{j+\frac{1}{2}},0)\left(\log\left(\frac{\rho(X_{j+\frac{1}{2}},0)}{\frac{x_{j+1}-x_{j}}{\delta X}}\right)+\frac{|x_{j+1}+x_{j}|^2}{8}\right),
	\end{align}
	then we have
	\begin{align}
\frac{\delta E_h}{\delta x_j}({\bm x}^{k+1})=&\delta X\frac{\rho(X_{j+\frac{1}{2}},0)}{x_{j+1}^{k+1}-x_{j}^{k+1}}-\delta X\frac{\rho(X_{j-\frac{1}{2}},0)}{x_{j}^{k+1}-x_{j-1}^{k+1}}\nonumber\\
	&+\frac{\delta X}{4}\left(\rho(X_{j+\frac{1}{2}},0)(x_{j+1}^{k+1}+x_{j}^{k+1})+\rho(X_{j-\frac{1}{2}},0)(x_{j}^{k+1}+x_{j-1}^{k+1})\right).
	\end{align}
	\item For the porous medium equation:
	\begin{align}\label{eq:energy pme}
	E_h({\bm x}):=\frac{1}{m-1}\sum_{j=0}^{N-1}(\delta X)^m\rho^m(X_{j+\frac{1}{2}},0)(x_{j+1}-x_{j})^{-m+1},\ m>1,
	\end{align}
	we obtain
	\begin{align}
	\frac{\delta E_h}{\delta x_j}({\bm x}^{k+1})=(\delta X)^m\left(\frac{\rho^m(X_{j+\frac{1}{2}},0)}{(x_{j+1}^{k+1}-x_{j}^{k+1})^{m}}-\frac{\rho^m(X_{j-\frac{1}{2}},0)}{(x_{j}^{k+1}-x_{j-1}^{k+1})^{m}}\right).
	\end{align}
  \item For the aggregation equation with the interaction potential $W(x)$, we consider the discrete energy as follows by treating $x$ implicit and $y$ explicit:
	\begin{align}
	E_h({\bm x}^{k+1}):=\sum_{j=0}^{N-1}\delta X\rho(X_{j+\frac{1}{2}},0)\sum_{i=0}^{N-1}\rho_{i+\frac{1}{2}}^k\int_{x_i^k}^{x_{i+1}^{k}}W(x_{j+\frac{1}{2}}^{k+1}-y)\mathrm{d}y,
	\end{align}
	it can be calculated that
	\begin{align}
	\frac{\delta E_h}{\delta x_j^{k+1}}({\bm x}^{k+1})=&\delta X\rho(X_{j+\frac{1}{2}},0)\sum_{i=0}^{N-1}\rho_{i+\frac{1}{2}}^k\int_{x_i^k}^{x_{i+1}^{k}}\frac{1}{2}W'_x(x_{j+\frac{1}{2}}^{k+1}-y)\mathrm{d}y\nonumber\\
	&+\delta X\rho(X_{j-\frac{1}{2}},0)\sum_{i=0}^{N-1}\rho_{i+\frac{1}{2}}^k\int_{x_i^k}^{x_{i+1}^{k}}\frac{1}{2}W'_x(x_{j-\frac{1}{2}}^{k+1}-y)\mathrm{d}y.
	\end{align}
	The Hessian matrix can be calculated by
	\begin{align*}
	\nabla^2E_h=\begin{bmatrix}
	\frac{\delta^2 E_h}{\delta (x_1^{k+1})^2}& \frac{\delta^2 E_h}{\delta x_1^{k+1}\delta x_2^{k+1}}&\\
	\frac{\delta^2 E_h}{\delta x_2^{k+1}\delta x_1^{k+1}}&\frac{\delta^2 E_h}{\delta (x_2^{k+1})^2}& \frac{\delta^2 E_h}{\delta x_2^{k+1}\delta x_3^{k+1}}&\\
	&\frac{\delta^2 E_h}{\delta x_3^{k+1}\delta x_2^{k+1}}&\frac{\delta^2 E_h}{\delta (x_3^{k+1})^2}& \frac{\delta^2 E_h}{\delta x_3^{k+1}\delta x_4^{k+1}}&\\   
	& &\ddots &\ddots &\ddots\\		
	& &&\frac{\delta^2 E_h}{\delta x_{N-2}^{k+1}\delta x_{N-3}^{k+1}}&\frac{\delta^2 E_h}{\delta (x_{N-2}^{k+1})^2}& \frac{\delta^2 E_h}{\delta x_{N-2}^{k+1}\delta x_{N-1}^{k+1}}\\
	& && &\frac{\delta^2 E_h}{\delta x_{N-1}^{k+1}\delta x_{N-2}^{k+1}}&\frac{\delta^2 E_h}{\delta (x_{N-1}^{k+1})^2}
	\end{bmatrix}.
	\end{align*}
	The discrete energy is strictly convex if the Hessian matrix $\nabla^2E_h$ is positive definite. 
	For instance, when choosing $W(x)=\frac{|x|^2}{2}-\ln|x|$, we can find that the discrete energy $E_h({\bm x}^{k+1})$ is convex since $W''(x)=1+\frac{1}{x^2}>0$.
	However, for the Keller-Segel equation with $U(\rho)=\rho\log\rho$ and $W(x)=\frac{1}{2\pi}\ln|x|$, we consider the discrete energy 
	\begin{align}
	E_h({\bm x}^{k+1})=&\sum_{i=0}^{N-1}\delta X\rho(X_{i+\frac{1}{2}},0)\log\left(\frac{\rho(X_{i+\frac{1}{2}},0)}{\frac{x_{i+1}^{k+1}-x_i^{k+1}}{\delta X}}\right)    -\frac{\delta X}{2\pi}\sum_{i=0}^{N-1}\rho(X_{i+\frac{1}{2}},0)\nonumber\\
	\times\sum_{j=0}^{N-1}\rho_{j+\frac{1}{2}}^k&\left((x_{i+\frac{1}{2}}^{k+1}-x_{j+1}^k)\ln(|x_{i+\frac{1}{2}}^{k+1}-x_{j+1}^k|)-(x_{i+\frac{1}{2}}^{k+1}-x_{j}^k)\ln(|x_{i+\frac{1}{2}}^{k+1}-x_{j}^k|)\right),
	\end{align}
	it is uncertain whether the discrete energy is convex. The details are shown in Appendix.

\end{itemize}

The fully discretized numerical scheme \eqref{schem:1}-\eqref{schem:1-2} in one dimensional case  holds the following structure-preserving properties. 
\begin{theorem}\label{thm:1d}
	Assume the initial value $\rho(X,0)>0$ for $X\in\Omega_0^X$, the energy density $F(s)$ satisfies $F(s)\ge 0$ for $s\ge 0$, and $\lim\limits_{s\rightarrow 0}F(\frac{1}{s})s=\infty$,  
	then the scheme \eqref{schem:1}-\eqref{schem:1-2} has a unique solution $\bm{x}^{k+1}\in S_{ad}$ when $\frac{\delta^2 E_h}{\delta{\bm x}^2}>0$, and the following properties hold: 
	\begin{itemize}
		\item the solution to numerical scheme \eqref{schem:1}-\eqref{schem:1-2} is positive,   $\rho_{j+\frac{1}{2}}^{k+1}>0$,
		\item the numerical scheme \eqref{schem:1}-\eqref{schem:1-2}  satisfies the property of mass conserving for the density $\rho^{k+1}$ in the sense that
		\begin{align}
		\sum_{j=0}^{N-1}\rho_{j+\frac{1}{2}}^{k+1}(x_{j+1}^{k+1}-x_j^{k+1})=	\sum_{j=0}^{N-1}\rho_{j+\frac{1}{2}}^{k}(x_{j+1}^{k}-x_j^{k})=	\sum_{j=0}^{N-1}\rho(X_{j+\frac{1}{2}},0)\delta X,
		\end{align}
		\item the discrete energy in scheme \eqref{schem:1}-\eqref{schem:1-2}   is dissipative in the sense that 
		\begin{align}
		\tilde E_h({\bm x}^{k+1})\le \tilde E_h({\bm x}^k).
		\end{align}
		where the energy $\tilde E_h(\bm x^{k+1})$ is defined by
		\begin{equation}
		\tilde E_h(\bm x^{k+1}) = E_h({\bm x}^{k+1})+\sum_{i=0}^{N-1}\frac{\delta t}{2}\left|\frac{x_{j+1}^{k+1}-x_{j}^{k+1}}{\delta X}\right|^2\delta X. 
		\end{equation}

	\end{itemize}
\end{theorem}
\begin{proof}
{\bf Step 1 (energy law):}

	We can easily check that  ${\bm x}^{k+1}$ is the minimizer of the following Lagrangian function which is defined by
	\begin{align}\label{eq:lag}
		L({\bm x}^{k+1})=\frac{1}{2\delta t}\sum_{j=0}^{N-1}\left(\rho(X_{j+\frac{1}{2}},0)|x_{j+\frac{1}{2}}^{k+1}-x_{j+\frac{1}{2}}^{k}|^2\delta X+\delta t^2|\frac{x_{j+1}^{k+1}-x_{j}^{k+1}}{\delta X}|^2\delta X\right)+E_h({\bm x}^{k+1}).
	\end{align}
	
	Since $x_j^{k+1}$ is the minimizer of \eqref{eq:discre optim}, it is easy to show that the discrete energy is dissipative in the sense that 
	\begin{equation*}
	\begin{split}
	& E_h({\bm x}^{k+1})+\sum_{i=0}^{N-1}\frac{\delta t}{2}|\frac{x_{j+1}^{k+1}-x_{j}^{k+1}}{\delta X}|^2\delta X\le E_h({\bm x}^k)+\sum_{i=0}^{N-1}\frac{\delta t}{2}|\frac{x_{j+1}^{k}-x_{j}^{k}}{\delta X}|^2\delta X \\& \le E_h({\bm x}^0)+\sum_{i=0}^{N-1}\frac{\delta t}{2}|\frac{x_{j+1}^{0}-x_{j}^{0}}{\delta X}|^2\delta X .
	 \end{split}
	 \end{equation*}
	
	{\bf Step 2 (positivity-preserving):}
	
Taking variational derivative with respect to the Lagrangian $\frac{\partial L}{\partial x_j^{k+1}}=0$ for $j=1$, $\cdots$, $N-1$, we derive the numerical scheme \eqref{schem:1}. Obviously,  the first two terms of \eqref{eq:lag} are strictly convex with respect to  ${\bm x}^{k+1}$. 
	Combining  the assumption that $\frac{\delta^2 E_h}{\delta{\bm x}^2}>0$,  we find there exists a unique minimizer ${\bm x}^{k+1}$ belonging to the closed convex set $\overline{S_{ad}}=\{{\bm x}: \ x_{j+1}\ge x_{j}\ \text{for}\ j=0, 1, \cdots, N-1,\ x_{j_0+1}=x_{j_0} \ \text{for some}\ 0\le j_0\le N-1,\ \text{and} \ x_0=X_0,\ x_{N}=X_N\}$.
	
	Next, we prove that  the minimizer does not lie on the boundary of  $\overline{S_{ad}}$ by contradiction. We assume there exists a minimizer ${\bm x}^{k+1}$ satisfies $x_{j_0+1}^{k+1}=x_{j_0}^{k+1}$ for some $0\le j_0\le N-1$, then we obtain $E({\bm x}^{k+1})=\infty$ since we have 
	\begin{align*}
		F\left(\frac{\rho( X_{j_0+\frac{1}{2}},0)}{\frac{x_{j_0+1}^{k+1}-x_{j_0}^{k+1}}{\delta X}}\right)\frac{x_{j_0+1}^{k+1}-x_{j_0}^{k+1}}{\delta X}=\infty, 
	\end{align*}
	under the assumption $\lim\limits_{s\rightarrow 0}F(\frac{1}{s})s=\infty$. 
	The above result implies
	\begin{align}
		L({\bm x}^{k+1})=\infty,\nonumber
	\end{align}
	which leads to a contradiction. Then the minimizer ${\bm x}^{k+1}$ must lie in the interior of the admissible set $S_{ad}$ satisfying $x_{j+1}^{k+1}> x_j^{k+1}$ for all $j=0$, $1$, $\cdots$, $N-1$.  Combining the $x_{j+1} \ge x_j$  and using the equality \eqref{schem:1-2}, the positivity of the numerical solution $\rho_{j+\frac{1}{2}}^{k+1}$ can easily be obtained.
	
	{\bf Step 3 (mass conserving):}
	
	From the equality \eqref{schem:1-2},  we derive the following result:
	\begin{align}
				\sum_{j=0}^{N-1}\rho_{j+\frac{1}{2}}^{k+1}(x_{j+1}^{k+1}-x_j^{k+1})=	\sum_{j=0}^{N-1}\frac{\rho(X_{j+\frac{1}{2}},0)}{\frac{x_{j+1}^{k+1}-x_{j}^{k+1}}{\delta X}}(x_{j+1}^{k+1}-x_j^{k+1})=	\sum_{j=0}^{N-1}\rho(X_{j+\frac{1}{2}},0)\delta X,\nonumber
	\end{align}
	which implies that  the proposed scheme is mass conserving. Finally, the proof is completed.
\end{proof}
\begin{remark}
	Specifically, for the porous medium equation and the Fokker-Planck equation, we have $\frac{\delta^2 E_h}{\delta{\bm x}^2}>0$ since the convexity of the energy, then Theorem \ref{thm:1d} will also hold without the regularization term. Actually, numerical experiments for both models are implemented without the regularization term in Section \ref{sec:num}.
\end{remark}
\begin{remark}
	To be specific, for the Fokker-Planck equation with energy \eqref{eq:energy fp} and porous medium equation with energy \eqref{eq:energy pme}, from the energy dissipation law $E_h({\bm x}^{k+1})
		\le E_h({\bm x}^0)+\frac{T\delta X}{2}$, we have
		\begin{align*}
		F\left(\frac{\rho(X_{j+\frac{1}{2}},0)}{\frac{x_{j+1}^{k+1}-x_j^{k+1}}{\delta X}}\right)\frac{x_{j+1}^{k+1}-x_j^{k+1}}{\delta X}\le E_h({\bm x}^0)+\frac{T\delta X}{2},\ 0\le j\le N-1.
		\end{align*}    
	Using the fact that energy density in both cases satisfies $F(\frac{1}{s})s\rightarrow\infty$ as $s\rightarrow0$, we derive that $\frac{x_{j+1}^{k+1}-x_j^{k+1}}{\delta X}>0$, $0\le j\le N-1$ is  uniformly away from zero.
\end{remark}
\begin{remark}
	If the regularization term is taken as $\sum_{j} |\frac{(x_{j+1}^{k+1}-x_{j+1}^k)-(x_{j}^{k+1}-x_j^k)}{\delta X}|^2\delta X$, properties in Theorem \ref{thm:1d} will also be derived, and the unconditionally discrete energy dissipation law $E_h({\bm x}^{k+1})\le E_h({\bm x}^k)$ holds.
\end{remark}
\begin{remark}
	In fact, the first term of the optimization problem \eqref{eq:discre optim} can also be substituted by
	\begin{align*}
	\frac{1}{2}\sum_{j=0}^{N-1}\rho(X_{j},0)|x_{j}^{k+1}-x_{j}^{k}|^2\delta X+\rho(X_{j+1},0)|x_{j+1}^{k+1}-x_{j+1}^{k}|^2\delta X,
	\end{align*}
	which is also an approximation to $\int_{\Omega_0^{\bm X}}\rho(X,0)|x^{k+1}-x^k|^2\mathrm{d}X$. Then the corresponding numerical scheme follows
	\begin{align}
	&\rho(X_{j},0)\frac{x_{j}^{k+1}-x_{j}^{k}}{\delta t}\delta X
	-\delta t\frac{x_{j+1}^{k+1}-2x_j^{k+1}+x_{j-1}^{k+1}}{(\delta X)^2}\delta X+\frac{\delta E_h}{\delta x_j}({\bm x}^{k+1})=0,\label{schem:11}
	\end{align}
	with the initial and boundary conditions \eqref{schem:1-3}, and $\rho^{k+1}$ defined by \eqref{schem:1-2}. The solution to scheme \eqref{schem:11} also preserves properties proved in Theorem \ref{thm:1d}.
\end{remark}
\subsection{Two dimensional case}\label{sec:two dim}
We extend our numerical approach to multidimensional case, for simplicity,  we only consider two-dimensions. Denote ${\bm x}=(x,y)$, ${\bm X}=(X,Y)$, and the Jacobian matrix $\frac{\partial {\bm x}}{\partial {\bm X}}=\frac{\partial (x,y)}{\partial (X,Y)}$. 
Set $\Omega_0^{\bm X}=[-L_x,L_x]\times[-L_y,L_y]$ with $L_x$, $L_y>0$, and $\Omega_0^{\bm x}=\Omega_0^{\bm X}$. Let $M_x$, $M_y\in\mathbb{N}$ be given, and define the grid spacing $h_x=\frac{2L_x}{M_x}$, $h_y=\frac{2L_y}{M_y}$. 
Let $X_{ij}=X_0+jh_x$, $Y_{ij}=Y_0+ih_y$ for $0\le j\le M_x$, $ 0\le i\le M_y$.
We define
\begin{align*}
&{\bm x}_{ij}(t_k)=\bm{x}(X_{ij},Y_{ij},t_k),\qquad 0\le j\le M_x,\ 0\le i\le M_y,\ 1\le k\le \frac{T}{\delta t},\\
&\rho_{ij}^0=\rho(X_{ij},Y_{ij},0)\ge 0.
\end{align*}
The explicit and implicit numerical schemes will be proposed in the following. 

\subsubsection{Implicit numerical scheme}
The following fully implicit numerical scheme with the regularization term $\epsilon\Delta_{\bm X}{\bm x}$, $\epsilon=\delta t$ can be proposed:  given $(x^k,y^k)$, solving $(x^{k+1},y^{k+1})$ from
\begin{align}
&\rho_{ij}^0\frac{x_{ij}^{k+1}-x_{ij}^k}{\delta t}-\delta t\left(\frac{x_{i,j+1}^{k+1}+x_{i,j-1}^{k+1}-2x_{ij}^{k+1}}{h_x^2}+\frac{x_{i+1,j}^{k+1}+x_{i-1,j}^{k+1}-2x_{ij}^{k+1}}{h_y^2}\right)+\frac{\delta \bar{E}_h}{\delta x}({\bm x}_{ij}^{k+1})=0,\label{im:1}\\
&\rho_{ij}^0\frac{y_{ij}^{k+1}-y_{ij}^k}{\delta t}-\delta t\left(\frac{y_{i,j+1}^{k+1}+y_{i,j-1}^{k+1}-2y_{ij}^{k+1}}{h_x^2}+\frac{y_{i+1,j}^{k+1}+y_{i-1,j}^{k+1}-2y_{ij}^{k+1}}{h_y^2}\right)+\frac{\delta \bar{E}_h}{\delta y}({\bm x}_{ij}^{k+1})=0,\label{im:2}
\end{align}
with the following initial and boundary conditions:
\begin{align}\label{eq:initial and boundary}
{\bm x}^0={\bm X},\qquad {\bm x }^{k+1}|_{\partial\Omega}={\bm X}|_{\partial\Omega}.
\end{align}
Then the density $\rho_{ij}^{k+1}$ can be obtained by
\begin{align}\label{eq:rho in 2d}
\rho_{ij}^{k+1}=\frac{\rho_{ij}^0}{F_{ij}^{k+1}}\quad\text{with}\quad 
F_{ij}^{k+1}=\left |\begin{array}{cc}
\frac{\partial x_{ij}^{k+1}}{\partial X} &\frac{\partial y_{ij}^{k+1}}{\partial X}   \\
\frac{\partial x_{ij}^{k+1}}{\partial Y} &\frac{\partial y_{ij}^{k+1}}{\partial Y}   \\
\end{array}\right|=\left |\begin{array}{cc}
\frac{ x_{i,j+1}^{k+1}-x_{i,j-1}^{k+1}}{2h_x} &\frac{ y_{i,j+1}^{k+1}-y_{i,j-1}^{k+1}}{2h_x}   \\
\frac{ x_{i+1,j}^{k+1}-x_{i-1,j}^{k+1}}{2h_y} &\frac{y_{i+1,j}^{k+1}-y_{i-1,j}^{k+1}}{2h_y}   \\
\end{array}\right|.
\end{align}
The proposed fully implicit scheme \eqref{im:1}-\eqref{im:2}  is highly nonlinear and should be solved in the admissible set $E_{ad}=\{{\bm x}:\ \mathrm{det}\frac{\partial {\bm x}}{\partial {\bm X}}|_{ij}>0 \ \text{for all}\ i,j\in\mathbb{N}, \ {\bm x}|_{\partial\Omega} ={\bm X}|_{\partial\Omega}\}$. The solution of the implicit scheme \eqref{im:1}-\eqref{im:2} is the minimizer of the minimization problem: 
\begin{align}\label{min:2d}
	{\bm x}^{k+1}:=\arg\inf_{\bm{x}\in E_{ad}}J_k({\bm x}),
\end{align}
where $J_k({\bm x})$ is defined by
\begin{align}
	J_k({\bm x})=&\frac{1}{2\delta t}\sum_{i,j}\rho_{ij}^0|{\bm x}_{ij}-{\bm x}_{ij}^{k}|^2h_xh_y+\delta t^2 (|\frac{{\bm x}_{i,j+1}-{\bm x }_{i,j}}{h_x}|^2+|\frac{{\bm x}_{i+1,j}-{\bm x}_{i,j}}{h_y}|^2)h_xh_y+\bar{E}_h({\bm x}),\nonumber
\end{align}
with $\bar{E}_h({\bm x})=\sum_{i,j}F\left(\frac{\rho_{ij}^0}{\mathrm{det}\frac{\partial {\bm x}}{\partial {\bm X}}|_{ij}}\right)\mathrm{det}\frac{\partial {\bm x}}{\partial {\bm X}}|_{ij}h_xh_y$. Following the analysis in \cite{liu2020lagrangian,carrillo2018lagrangian}, we obtain the following  result for the implicit numerical scheme \eqref{im:1}-\eqref{im:2} .
\begin{theorem}
	Assume the density of energy $F(s)$ satisfies $\lim\limits_{s\rightarrow 0}F(\frac{1}{s})s=\infty$, and $F(s)\ge 0$ for $s\ge 0$. 
	Then there exists a solution ${\bm x}^{k+1}\in E_{ad}$ to the nonlinear numerical scheme \eqref{im:1}-\eqref{im:2} , and the following energy dissipation law holds:
	\begin{align}
		&\bar{E}_h({\bm x}^{k+1})+\frac{h_xh_y}{2}\sum_{i,j}\delta t\left(\left|\frac{{\bm x}_{i,j+1}^{k+1}-{\bm x }_{i,j}^{k+1}}{h_x}\right|^2+\left|\frac{{\bm x}_{i+1,j}^{k+1}-{\bm x}_{i,j}^{k+1}}{h_y}\right|^2\right)\nonumber\\
		\le& \bar{E}_h({\bm x}^k)+\frac{h_xh_y}{2}\sum_{i,j}\delta t\left(\left|\frac{{\bm x}_{i,j+1}^{k}-{\bm x }_{i,j}^{k}}{h_x}\right|^2+\left|\frac{{\bm x}_{i+1,j}^{k}-{\bm x}_{i,j}^{k}}{h_y}\right|^2\right).
	\end{align}
\end{theorem}
\begin{proof}
	The existence of the solution to the nonlinear numerical scheme \eqref{im:1}-\eqref{im:2}  is equivalent to obtaining  the minimizer of $J_k({\bm x})$ in the admissible set $E_{ad}$. Then we turn to prove the existence of the minimizer of the minimization problem \eqref{min:2d}. If the minimizer lies on the boundary of the admissible set, i.e.\ ${\bm x}\in \partial E_{ad}$, we have $J_k({\bm x})=\infty$, which is a contradiction. Following the proof in \cite{liu2020lagrangian,carrillo2018lagrangian}, the claim of the theorem will be derived once we show that the sub-level set
	\begin{align}
		\mathcal{S}:=\Big\{{\bm x}\in E_{ad}:\ J_k({\bm x})\le \bar{E}_h({\bm x}^k)+\frac{h_xh_y}{2}\sum_{i,j}\delta t(|\frac{{\bm x}_{i,j+1}^{k}-{\bm x }_{i,j}^{k}}{h_x}|^2+|\frac{{\bm x}_{i+1,j}^{k}-{\bm x}_{i,j}^{k}}{h_y}|^2):=\gamma\Big\}\nonumber
	\end{align}
	is a non-empty compact subset of $\mathbb{R}^2$. Clearly, ${\bm x}^k\in \mathcal{S}$, so it is non-empty.
	
	{\it $\mathcal{S}$ is bounded.}  Assume the initial value $\rho_{ij}^0>0$, there exists $\lambda>0$ such that
	\begin{align}
		\frac{\lambda}{2\delta t} \sum_{i,j}|{\bm x}_{ij}-{\bm x}_{ij}^k|^2h_xh_y\le \frac{1}{2\delta t}\sum_{i,j}\rho_{ij}^0|{\bm x}_{ij}-{\bm x}_{ij}^{k}|^2h_xh_y\le \gamma.
	\end{align}
	Then $\mathcal{S}$ is bounded.
	
	{\it $\mathcal{S}$ is a closed subset of $\mathbb{R}^2$.} It suffices to show that the limit $\tilde{\bm x}$ of any sequence $\{{\bm x}^{(k)}\}_{k=1}^{\infty}\subset E_{ad}$ belongs to $E_{ad}$. For all $k$, we have
	\begin{align}
		\gamma\ge \bar{E}_h({\bm x}^{(k)})\ge h_xh_yF\left(\frac{\rho_{ij}^0}{\mathrm{det}\frac{\partial {\bm x}^{(k)}}{\partial {\bm X}}|_{ij}}\right)\mathrm{det}\frac{\partial {\bm x}^{(k)}}{\partial {\bm X}}\Big|_{ij}.
	\end{align}
	Since $F\left(\frac{\rho_{ij}^0}{\mathrm{det}\frac{\partial {\bm x}^{(k)}}{\partial {\bm X}}|_{ij}}\right)\mathrm{det}\frac{\partial {\bm x}^{(k)}}{\partial {\bm X}}|_{ij}\rightarrow \infty$ as $\mathrm{det}\frac{\partial {\bm x}^{(k)}}{\partial {\bm X}}|_{ij}\rightarrow 0$, it follows that $\mathrm{det}\frac{\partial {\bm x}^{(k)}}{\partial {\bm X}}|_{ij}>0$ is bounded away from zero, uniformly in $k$. Then we obtain $\mathrm{det}\frac{\partial {\tilde{\bm x}}}{\partial {\bm X}}|_{ij}>0$, and $\tilde{\bm x}\in E_{ad}$.
	
	If ${\bm x}^{k+1}\in \mathcal{S}$ is a minimizer of the minimization problem \eqref{min:2d}, we have
		\begin{align}
	&\bar{E}_h({\bm x}^{k+1})+\frac{h_xh_y}{2}\sum_{i,j}\delta t\left(|\frac{{\bm x}_{i,j+1}^{k+1}-{\bm x }_{i,j}^{k+1}}{h_x}|^2+|\frac{{\bm x}_{i+1,j}^{k+1}-{\bm x}_{i,j}^{k+1}}{h_y}|^2\right)\nonumber\\
	\le& \bar{E}_h({\bm x}^k)+\frac{h_xh_y}{2}\sum_{i,j}\delta t\left(|\frac{{\bm x}_{i,j+1}^{k}-{\bm x }_{i,j}^{k}}{h_x}|^2+|\frac{{\bm x}_{i+1,j}^{k}-{\bm x}_{i,j}^{k}}{h_y}|^2\right),\nonumber
	\end{align}
	which completes the proof.
\end{proof}

\begin{remark}
	As stated in \cite{liu2020lagrangian,carrillo2018lagrangian}, we do not claim the uniqueness of the solution to the nonlinear numerical scheme \eqref{im:1}-\eqref{im:2}  due to the lack of convexity of $J_k({\bm x})$ and $\bar{E}_h({\bm x})$ in $d$-dimension ($d\ge 2$).
\end{remark}

\begin{remark}
	The regularization term can also be taken as $\epsilon\sum_{i,j}(|\frac{{\bm x}_{i,j+1}-{\bm x }_{i,j}-({\bm x}_{i,j+1}^k-{\bm x }_{i,j}^k)}{h_x}|^2+|\frac{{\bm x}_{i+1,j}-{\bm x}_{i,j}-({\bm x}_{i+1,j}^k-{\bm x}_{i,j}^k)}{h_y}|^2)$, then the energy dissipation law still holds $\bar{E}_h({\bm x}^{k+1})\le \bar{E}_h({\bm x}^k)$.
\end{remark}

\subsubsection{Explicit numerical scheme}
The trajectories $(x^{k+1},y^{k+1})$ can also be solved by the following linear scheme with the regularization term $\epsilon\Delta_{\bm X}{\bm x}^{k+1}$, $\epsilon=\delta t$, for given $(x^k,y^k)$: 
\begin{align}
&\rho_{ij}^0\frac{x_{ij}^{k+1}-x_{ij}^k}{\delta t}-\delta t\left(\frac{x_{i,j+1}^{k+1}+x_{i,j-1}^{k+1}-2x_{ij}^{k+1}}{h_x^2}+\frac{x_{i+1,j}^{k+1}+x_{i-1,j}^{k+1}-2x_{ij}^{k+1}}{h_y^2}\right)+\frac{\delta \bar{E}_h}{\delta x}({\bm x}_{ij}^k)=0,\label{ex:1}\\
&\rho_{ij}^0\frac{y_{ij}^{k+1}-y_{ij}^k}{\delta t}-\delta t\left(\frac{y_{i,j+1}^{k+1}+y_{i,j-1}^{k+1}-2y_{ij}^{k+1}}{h_x^2}+\frac{y_{i+1,j}^{k+1}+y_{i-1,j}^{k+1}-2y_{ij}^{k+1}}{h_y^2}\right)+\frac{\delta \bar{E}_h}{\delta y}({\bm x}_{ij}^k)=0,\label{ex:2}
\end{align}
 with the initial and boundary conditions \eqref{eq:initial and boundary},  then density $\rho^{k+1}$ will be derived by \eqref{eq:rho in 2d}.

The last term $\frac{\delta \bar{E}_h}{\delta{\bm x}}({\bm x}^k)$ of the proposed numerical scheme \eqref{ex:1}-\eqref{ex:2} is taken to be fully explicit in the following numerical experiments. Obviously, the proposed scheme \eqref{ex:1}-\eqref{ex:2}  admits  a unique solution from its linearity. To be specific, we give some explicit formulas of $\frac{\delta \bar{E}_h}{\delta{\bm x}}({\bm x}^k)$  for various kinds of gradient flows. 

For the porous medium equation, we have the following results for $m>1$:
\begin{align}
\frac{\delta \bar{E}_h}{\delta x}({\bm x}_{ij}^k)=\frac{\left(\frac{\rho_{i,j+1}^0}{F_{i,j+1}^k}\right)^m-\left(\frac{\rho_{i,j-1}^0}{F_{i,j-1}^k}\right)^m}{2h_x},\qquad \frac{\delta \bar{E}_h}{\delta y}({\bm x}_{ij}^k)=\frac{\left(\frac{\rho_{i+1,j}^0}{F_{i+1,j}^k}\right)^m-\left(\frac{\rho_{i-1,j}^0}{F_{i-1,j}^k}\right)^m}{2h_y}.
\end{align}
For the aggregation-diffusion models, we have
\begin{align}
\frac{\delta \bar{E}_h}{\delta x}({\bm x}_{ij}^k)=\frac{\left(\frac{\rho_{i,j+1}^0}{F_{i,j+1}^k}\right)^m-\left(\frac{\rho_{i,j-1}^0}{F_{i,j-1}^k}\right)^m}{2h_x}+\nu\rho_{ij}^0\sum_{p,q}\rho_{p,q}^kW_{x}'({\bm x}_{ij}^k-{\bm z}_{pq}^k)V_{pq},\\
\frac{\delta \bar{E}_h}{\delta y}({\bm x}_{ij}^k)=\frac{\left(\frac{\rho_{i+1,j}^0}{F_{i+1,j}^k}\right)^m-\left(\frac{\rho_{i-1,j}^0}{F_{i-1,j}^k}\right)^m}{2h_y}+\nu\rho_{ij}^0\sum_{p,q}\rho_{p,q}^kW_{y}'({\bm x}_{ij}^k-{\bm z}_{pq}^k)V_{pq},
\end{align}
where $V_{pq}$ represents the area of the control volume of ${\bm x}_{pq}^k$,  $m=1$ represents the linear diffusion case where $U(\rho)=\rho\log\rho$, and $m>1$ denotes the nonlinear diffusion case where $U(\rho)=\frac{1}{m-1}\rho^m$.

Let us discuss the explicit numerical scheme \eqref{ex:1}-\eqref{ex:2} with a different regularization term $\epsilon_k\Delta_{\bm X}({\bm x}^{k+1}-{\bm x}^k)$, $\epsilon_k\ge 0$, which also has a unique solution. We assume the solution belongs to the set $E_{ad}:=\{{\bm x}:\ \mathrm{det}\frac{\partial {\bm x}}{\partial {\bm X}}|_{ij}>0 \ \text{for all}\ i,j\in\mathbb{N},\ {\bm x}|_{\partial\Omega} ={\bm X}|_{\partial\Omega}\}$, and introduce the following minimization problem in the admissible set $E_{ad}$:
\begin{align}\label{min:explicit 2d}
{\bm x}^{k+1}:=\arg\inf_{\bm{x}\in E_{ad}}L_k({\bm x}),
\end{align} 
where $L_k({\bm x})$ is defined by
\begin{align}
L_k({\bm x})=&\frac{1}{2\delta t}\sum_{i,j}\rho_{ij}^0|{\bm x}_{ij}-{\bm x}_{ij}^{k}|^2h_xh_y+\frac{\epsilon_k}{2}\|\nabla_{\bm X}({\bm x}-{\bm x}^k)\|^2
+\bar{E}_h({\bm x}^k)+(\frac{\delta \bar{E}_h}{\delta {\bm x}}({\bm x}^k),{\bm x}-{\bm x}^k),\nonumber
\end{align}
with $\bar{E}_h({\bm x})=\sum_{i,j}F\left(\frac{\rho_{ij}^0}{\mathrm{det}\frac{\partial {\bm x}}{\partial {\bm X}}|_{ij}}\right)\mathrm{det}\frac{\partial {\bm x}}{\partial {\bm X}}|_{ij}h_xh_y$. Then the solution to the explicit numerical scheme \eqref{ex:1}-\eqref{ex:2} with regularization term $ \epsilon\Delta_{\bm X}({\bm x}^{k+1}-{\bm x}^k)$ is the minimizer of the minimization problem \eqref{min:explicit 2d}.

For the sake of simplicity, we consider the Porous-Medium equation in 2D, we have
\begin{align}
	\left(\frac{\delta \bar{E}_h}{\delta x}({\bm x}^k),x-x^k\right)=-\left((\frac{\rho^0}{\mathrm{det}\frac{\partial{\bm x}^k}{\partial {\bm X}}})^{m},\mathrm{det}\frac{\partial((x-x^k),y^k)}{\partial(X,Y)}\right),\nonumber\\
	\left(\frac{\delta \bar{E}_h}{\delta y}({\bm x}^k),y-y^k\right)=-\left((\frac{\rho^0}{\mathrm{det}\frac{\partial{\bm x}^k}{\partial {\bm X}}})^m,\mathrm{det}\frac{\partial(x^k,(y-y^k))}{\partial(X,Y)}\right),\nonumber\\
	\left(\frac{\delta^2 \bar{E}_h}{\delta x^2}({\bm x}^k),(x-x^k)^2\right)=m\left(\frac{(\rho^0)^m}{(\mathrm{det}\frac{\partial{\bm x}^k}{\partial {\bm X}})^{m+1}},(\mathrm{det}\frac{\partial((x-x^k),y^k)}{\partial(X,Y)})^2\right),\nonumber\\
	\left(\frac{\delta^2 \bar{E}_h}{\delta y^2}({\bm x}^k),(y-y^k)^2\right)=m\left(\frac{(\rho^0)^m}{(\mathrm{det}\frac{\partial{\bm x}^k}{\partial {\bm X}})^{m+1}},(\mathrm{det}\frac{\partial(x^k,(y-y^k))}{\partial(X,Y)})^2\right).\nonumber
\end{align}
Assume that the initial value $\rho({\bm X},0)\ge0$ is bounded, and $\mathrm{det}\frac{\partial {\bm x}^k}{\partial {\bm X}}>0$ is uniformly bounded away from zero in $k$ satisfying $\mathrm{det}\frac{\partial{\bm x}^k}{\partial {\bm X}}\ge \delta_0$ for some $\delta_0>0$, we have $\|\frac{m(\rho^0)^m}{(\mathrm{det}\frac{\partial{\bm x}}{\partial {\bm X}})^{m+1}}\|_{\infty}\le \frac{C_0} {\delta_0^{m+1}}$, where $C_0$ is a positive constant depending on $m$ and $\rho^0$. Then the following estimates will be obtained:
\begin{align}
	\left|m\left(\frac{(\rho^0)^m}{(\mathrm{det}\frac{\partial{\bm x}^k}{\partial {\bm X}})^{m+1}},(\mathrm{det}\frac{\partial((x-x^k),y^k)}{\partial(X,Y)})^2\right)\right|\le \frac{2C_0 }{\delta_0^{m+1}}\|\nabla_{\bm X} y^k\|_{\infty}^{2}\|\nabla_{\bm X}(x-x^k)\|^2,\\
	\left|m\left(\frac{(\rho^0)^m}{(\mathrm{det}\frac{\partial{\bm x}^k}{\partial {\bm X}})^{m+1}},(\mathrm{det}\frac{\partial(x^k,(y-y^k))}{\partial(X,Y)})^2\right)\right|\le \frac{2C_0 }{\delta_0^{m+1}}\|\nabla_{\bm X} x^k\|_{\infty}^{2}\|\nabla_{\bm X}(y-y^k)\|^2.
\end{align}
Under these assumptions, the following energy dissipation law can be obtained.

\begin{theorem}\label{thm:2d}
	The solution to the explicit numerical scheme \eqref{ex:1}-\eqref{ex:2} with the regularization term $\epsilon_k\Delta_{\bm X}({\bm x}^{k+1}-{\bm x}^k)$, $\epsilon_k\ge 0$ is the minimizer of the minimization problem \eqref{min:explicit 2d}. 
	If either of the following conditions is true:
	\begin{itemize}
		\item $\epsilon_k=0$, choose suitable time step controlled by $\delta t\le \tau_{\min}=\frac{\min{\rho^0_{ij}}\delta_0^{m+1}h^2}{2C_1C_0\|\nabla_{\bm X}{\bm x}^k\|_{\infty}^{2}}$ with $h=h_x=h_y$,
		\item choose suitable regularization parameter $\epsilon_k\ge \frac{C_0 \|\nabla_{\bm X}{\bm x}^k\|_{\infty}^{2}}{\delta_0^{m+1}}$,
	\end{itemize}
  then we have	the following energy dissipation law:
	\begin{align}
	\bar{E}_h({\bm x}^{k+1})\le\bar{E}_h({\bm x}^k).
	\end{align}
\end{theorem}

\begin{proof}
	Notice that $L_k({\bm x}^{k+1})\le L_k({\bm x}^{k})=\bar{E}_h({\bm x}^k)$, and
	\begin{align}
	\bar{E}_h({\bm x}^{k+1})=&\bar{E}_h({\bm x}^{k})+(\frac{\delta\bar{E}_h}{\delta{\bm x}}({\bm x}^k), {\bm x}^{k+1}-{\bm x}^k)+\frac{1}{2}(\frac{\delta^2 \tilde{E}}{\delta {\bm x}^2}((1-t){\bm x}^k+t{\bm x}^{k+1}),({\bm x}^{k+1}-{\bm x}^k)^2)\nonumber\\
	\le&\bar{E}_h({\bm x}^{k})+(\frac{\delta\bar{E}_h}{\delta{\bm x}}({\bm x}^k), {\bm x}^{k+1}-{\bm x}^k)+\frac{C_0 }{\delta_0^{m+1}}\|\nabla_{\bm X}{\bm x}^k\|_{\infty}^{2}\|\nabla_{\bm X}({\bm x}^{k+1}-{\bm x}^k)\|^2,\nonumber
	\end{align}
	for $0<t<1$, then the energy dissipation law will be derived once we choose suitable $\epsilon_k$ and time step $\delta t$ such that the following inequality holds:
	\begin{align}
		\bar{E}_h({\bm x}^{k})+(\frac{\delta\bar{E}_h}{\delta{\bm x}}({\bm x}^k), {\bm x}^{k+1}-{\bm x}^k)+\frac{C_0 }{\delta_0^{m+1}}\|\nabla_{\bm X}{\bm x}^k\|_{\infty}^{2}\|\nabla_{\bm X}({\bm x}^{k+1}-{\bm x}^k)\|^2\le L_k({\bm x}^{k+1}),\nonumber
	\end{align}
	we only need to guarantee that
	\begin{align}
	\frac{C_0 }{\delta_0^{m+1}}\|\nabla_{\bm X}{\bm x}^k\|_{\infty}^{2}\|\nabla_{\bm X}({\bm x}^{k+1}-{\bm x}^k)\|^2
	\le \frac{1}{2\delta t}\sum_{i,j}\rho_{ij}^0|{\bm x}_{ij}^{k+1}-{\bm x}_{ij}^{k}|^2h^2+\epsilon_k\|\nabla_{\bm X} ({\bm x}^{k+1}-{\bm x}^k)\|^2.\nonumber
	\end{align}
	
	In the case where $\epsilon_k=0$, 
	using the inverse estimate, we have 
	$$\frac{C_0 }{\delta_0^{m+1}}\|\nabla_{\bm X}{\bm x}^k\|_{\infty}^{2}\|\nabla_{\bm X}({\bm x}^{k+1}-{\bm x}^k)\|^2\le\frac{C_1C_0\|\nabla_{\bm X}{\bm x}^k\|_{\infty}^{2}}{h^2 \delta_0^{m+1}}\|{\bm x}^{k+1}-{\bm x}^k\|^2,$$
	then the left hand side of above inequality will be controlled by $ \frac{1}{2\delta t}\sum_{i,j}\rho_{ij}^0|{\bm x}_{ij}^{k+1}-{\bm x}_{ij}^{k}|^2h^2$ once we choose suitable time step $\delta t\le \frac{\min{\rho^0_{ij}}\delta_0^{m+1}h^2}{2C_1C_0\|\nabla_{\bm X}{\bm x}^k\|_{\infty}^{2}}:=\tau_{\min}$ such that $\frac{C_1C_0\|\nabla_{\bm X}{\bm x}^k\|_{\infty}^{2}}{\delta_0^{m+1}h^2}\le\frac{\min{\rho^0_{ij}}}{2\delta t}$. The energy dissipation law can be obtained. 
	
	In the other case where we choose suitable regularization parameter $\epsilon_k\ge \frac{C_0 \|\nabla_{\bm X}{\bm x}^k\|_{\infty}^{2}}{\delta_0^{m+1}}$ such that $\frac{C_0 }{\delta_0^{m+1}}\|\nabla_{\bm X}{\bm x}^k\|_{\infty}^{2}\|\nabla_{\bm X}({\bm x}^{k+1}-{\bm x}^k)\|^2
	\le \epsilon_k\|\nabla_{\bm X} ({\bm x}^{k+1}-{\bm x}^k)\|^2$, the desired energy dissipation law is derived. 
\end{proof}

\begin{remark}
	In the proof of Theorem \ref{thm:2d}, we assume that $\mathrm{det}\frac{\partial{\bm x}}{\partial {\bm X}}\ge\delta_0>0$ is away from zero. From  determinant plots depicted by numerical experiments in Section \ref{sec:num in 2d}, we observe that the determinant of the deformation gradient for the Porous-Medium equation with Barenblatt solution and Aggregation equation satisfies this assumption. However, for the Keller-Segel model, where the particles aggregate at the center or at the circumference under certain conditions, the positivity of the determinant will not be maintained with time, and eventually the trajectory will become distorted, and once this distortion occurs, the numerical experiment will be stopped. Such cases need to be further investigated.
\end{remark}

\begin{remark}
	Similarly, if we consider the minimization problem with the regularization term $\epsilon_k\delta t\|\nabla_{\bm X}{\bm x}^{k+1}\|^2$, the energy dissipation law $\bar{E}_h({\bm x}^{k+1})+\epsilon_k\delta t\|\nabla_{\bm X}{\bm x}^{k+1}\|^2\le\bar{E}_h({\bm x}^k)+\epsilon_k\delta t\|\nabla_{\bm X}{\bm x}^{k}\|^2$ is derived if  we choose suitable $\epsilon_k$ and $\delta t$ such that the following inequality holds:
	\begin{align}
	\frac{C_0 }{\delta_0^{m+1}}\|\nabla_{\bm X}{\bm x}^k\|_{\infty}^{2}\|\nabla_{\bm X}({\bm x}^{k+1}-{\bm x}^k)\|^2
	\le \frac{1}{2\delta t}\sum_{i,j}\rho_{ij}^0|{\bm x}_{ij}-{\bm x}_{ij}^{k}|^2h^2+\epsilon_k\delta t\|\nabla_{\bm X}{\bm x}^{k+1}\|^2.\nonumber
	\end{align}
	In the case where we choose suitable regularization parameter $\epsilon_k$ to control the left hand side of above inequality by the regularization term, it can be roughly estimated that 
	  $$\epsilon_k\ge\frac{1}{\delta t}\frac{C_0\|\nabla_{\bm X}{\bm x}^k\|_{\infty}^{2}\|\nabla_{\bm X}({\bm x}^{k+1}-{\bm x}^k)\|^2}{ \delta_0^{m+1}\|\nabla_{\bm X}{\bm x}^{k+1}\|^2}= \frac{C_0\|\nabla_{\bm X}{\bm x}^k\|_{\infty}^{2}\|\nabla_{\bm X}\frac{{\bm x}^{k+1}-{\bm x}^k}{\delta t}\|^2}{\delta_0^{m+1}\|\nabla_{\bm X}{\bm x}^{k+1}\|^2} \delta t\ge C_2\delta t.$$	
\end{remark}

\begin{remark}
	It is worth noting that when we simulate numerical experiments using the explicit numerical scheme without regularization term, the time step $\delta t$ is supposed to be sufficient small to guarantee the stability of the numerical scheme. If we carry out numerical experiments with the regularization term $\epsilon_k\|\nabla_{\bm X}({\bm x}^{k+1}-{\bm x}^k)\|^2$, the regularization parameter should be taken appropriately through $\epsilon_k\ge C_0\|\nabla_{\bm X}{\bm x}^k\|_{\infty}^{2}/\delta_0^{m+1}$ to ensure the stability of the numerical scheme. The regularization term can also be taken as $\epsilon_k\delta t\|\nabla_{\bm X} {\bm x}^{k+1}\|^2$ with $\epsilon_k\ge C_2\delta t$, as displayed in the following numerical experiments.
\end{remark}

\section{Numerical simulations}\label{sec:num}
In this section, numerical experiments for Porous-Medium equation, Fokker-Planck equation, Keller-Segel equation and Aggregation equation will be considered in one dimension and two dimension to validate the accuracy and stability of our proposed numerical schemes based on our flow dynamic approach. 
\subsection{One dimension}
 For simplicity, we shall first show numerical experiments for  models in 1D, then we consider numerical simulations in 2D  in next subsection. 
\subsubsection{Porous medium equation}
The porous medium equation $\partial_t\rho=\Delta\rho^m$, $m>1$, can be regarded as the Wasserstein gradient flow with energy defined by
\begin{align*}
E(\rho)=\int_{\Omega}\frac{1}{m-1}\rho^m\mathrm{d}x.
\end{align*}

{\bf Convergence test}. 
Consider the following smooth initial value: 
\begin{align}
	\rho_0(x)=\cos\left(\frac{\pi x}{2}\right),\qquad x\in[-1,1],
\end{align}
with Dirichlet boundary condition $x|_{\partial\Omega}=X|_{\partial\Omega}$.
 The numerical solution is solved by using scheme \eqref{schem:1}-\eqref{schem:1-2} without regularization term. The reference solution is computed under  very fine meshes with $M=8000$, $\delta t=1/64000$. The convergence rates for density $\rho$ and trajectories ${\bm x}$ in $L^2$  and $L^\infty$ norms  are shown in Table~\ref{convergence fix}. We also depict  the evolutions of  energy and mass in Figure~\ref{fig:energy_and_mass} which show the property of energy dissipation and mass conserving with respective to time. 
\begin{table}[!htb]
	\centering
	\caption{Convergence order of trajectory $x$ and density $\rho$  with $m=2$ at $T=0.5$.}
		\begin{tabular}{cccccccccc}
			\hline
			M &$\delta t$&$L_h^{2}$ error ($x$) & order & $L^{\infty}$ error ($x$) & order &\ $L_h^{2}$ error ($\rho$) & order\\ \hline
			100 &1/100& 4.0522e-04	& & 2.3834e-04& & 3.6068e-04&\\
			200 &1/400& 9.6283e-05&2.0734 &5.6442e-05&2.0782&8.8534e-05 &2.0624\\
			400 &1/1600& 2.3484e-05&2.0356& 1.3746e-05 &2.0377&	2.2188e-05& 1.9964\\
			800 &1/6400& 5.5614e-06&2.0781&  3.2461e-06&2.0822&5.4442e-06  &2.0270\\
			\hline
		\end{tabular}\label{convergence fix}
\end{table}

\begin{figure}[!htb]
	\centering
	\subfigure[Energy dissipation]{
		\includegraphics[width=0.3\textwidth]{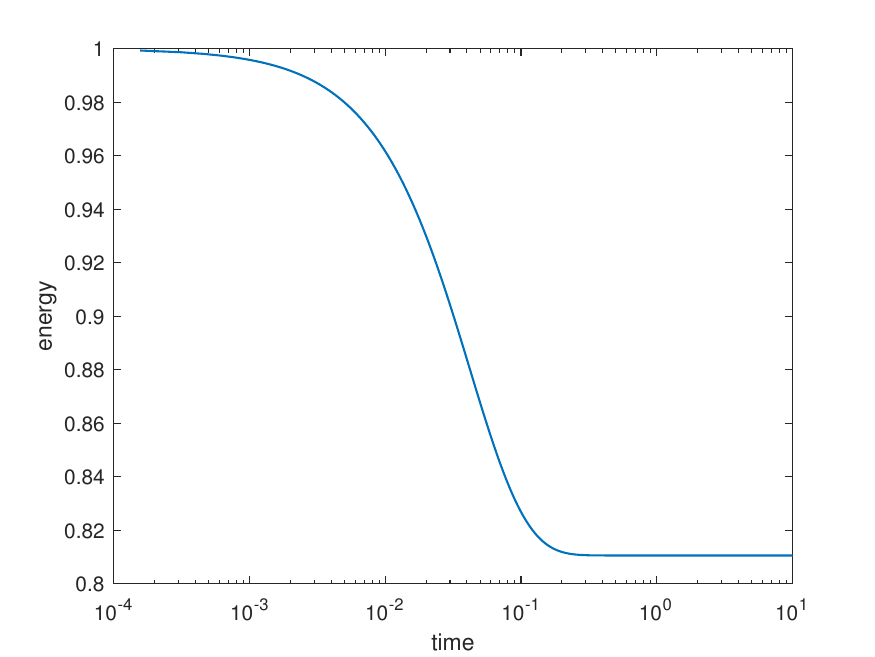}
	}
	\subfigure[Mass conservation]{
		\includegraphics[width=0.3\textwidth]{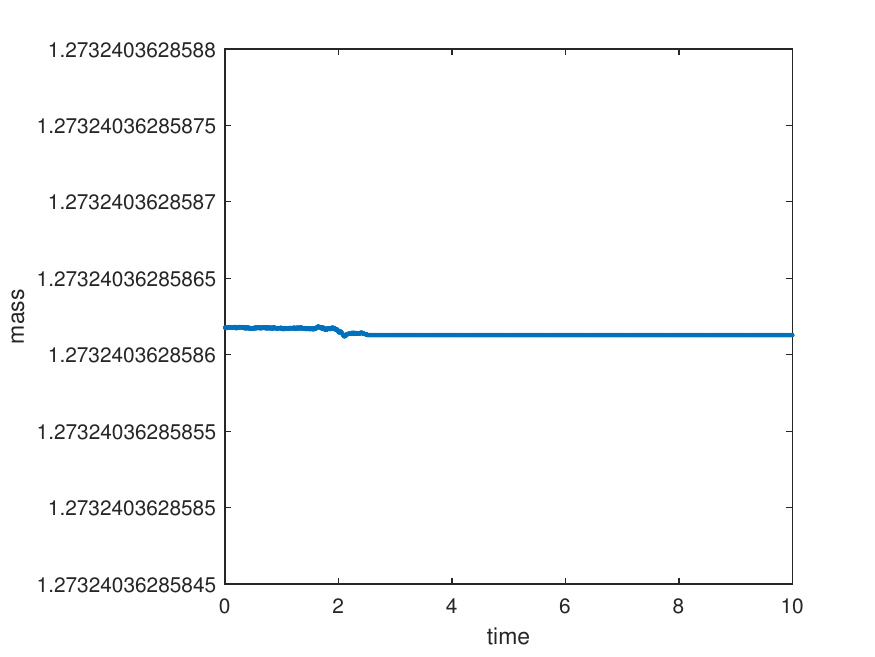}
	}
	\subfigure[Convergence order]{
	\includegraphics[width=0.3\textwidth]{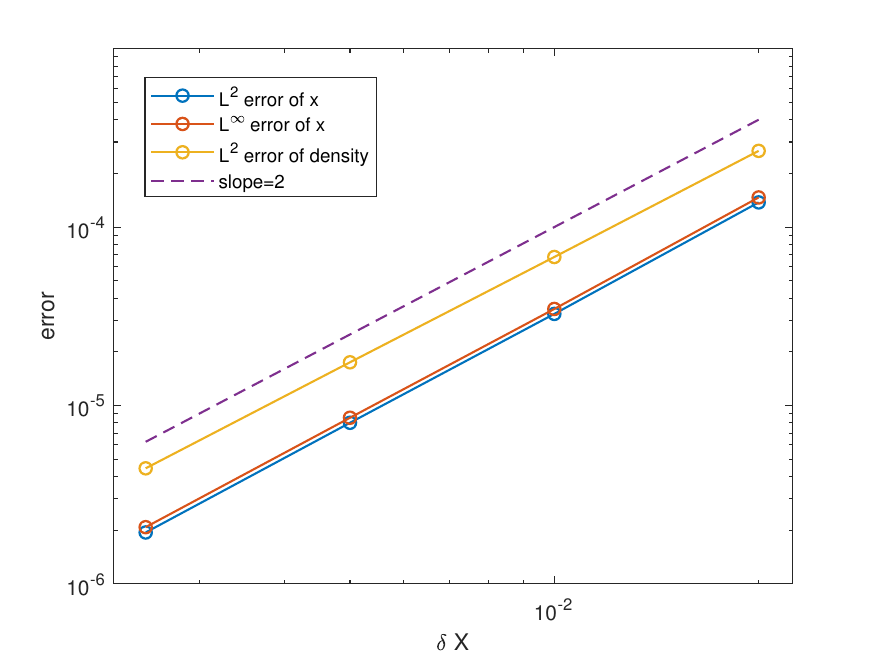}
}
	\caption{The evolutions of energy and mass with respective to time under $m=2$ with $M=800$, $\delta t=1/6400$.}\label{fig:energy_and_mass}
\end{figure}

{\bf Free boundaries.} 
Considering  the Barenblatt solution for Porous medium equation with free boundaries \cite{vazquez2007porous}:
\begin{align}
	B_m(x,t)=(t+1)^{-k}\left(1-\frac{k(m-1)}{2m}\frac{|x|^2}{(t+1)^{2k}}\right)_{+}^{1/(m-1)},
\end{align}
where $k=(m+1)^{-1}$. The support set of the solution is $[l_m(t),r_m(t)]$ with the moving interface $r_m(t)=-l_m(t):=\sqrt{\frac{2m}{k(m-1)}}(t+1)^k$. 

Using scheme \eqref{schem:1}-\eqref{schem:1-2} without regularization term to calculate the interior points, and \eqref{appendix:pme_boundary1}-\eqref{appendix:pme_boundary2} to compute the boundaries. We choose the Barenblatt solution $B_m(x,0)$ as the initial value to simulate the phenomenon of moving interface. The results are displayed in Figure~\ref{fig:barenblatt}, it can be found that the free boundaries move with a finite speed, and the numerical propagation speed is consistent with the exact solution. The proposed scheme satisfies the property of  energy dissipating, positivity-preserving and mass conserving for  the density $\rho$. 
\begin{figure}[!htb]
	\centering
	\subfigure[Density $\rho$ when $m=2$]{
		\includegraphics[width=0.45\textwidth]{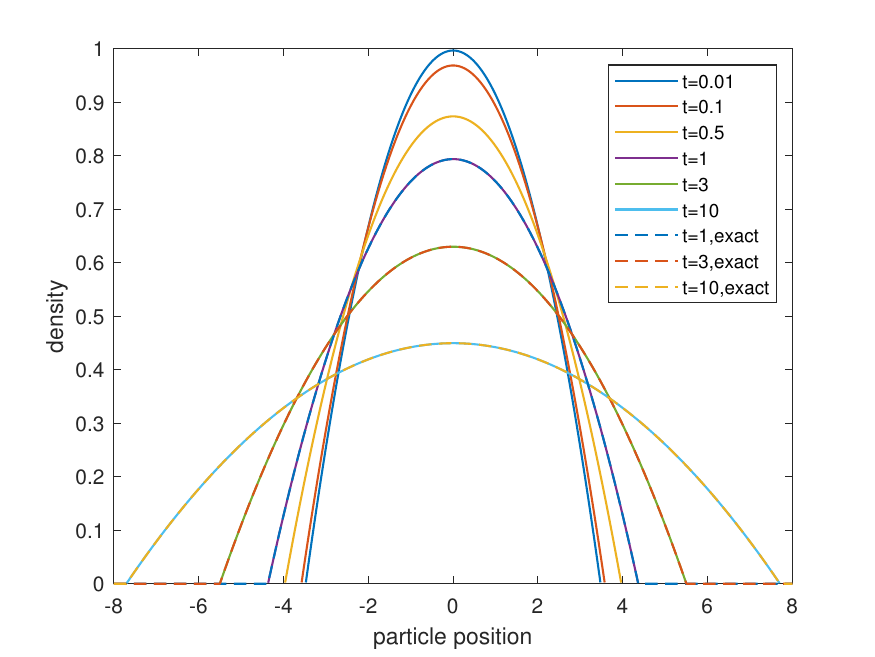}
	}
	\subfigure[Right boundary when  $m=2$, $3$]{
	\includegraphics[width=0.45\textwidth]{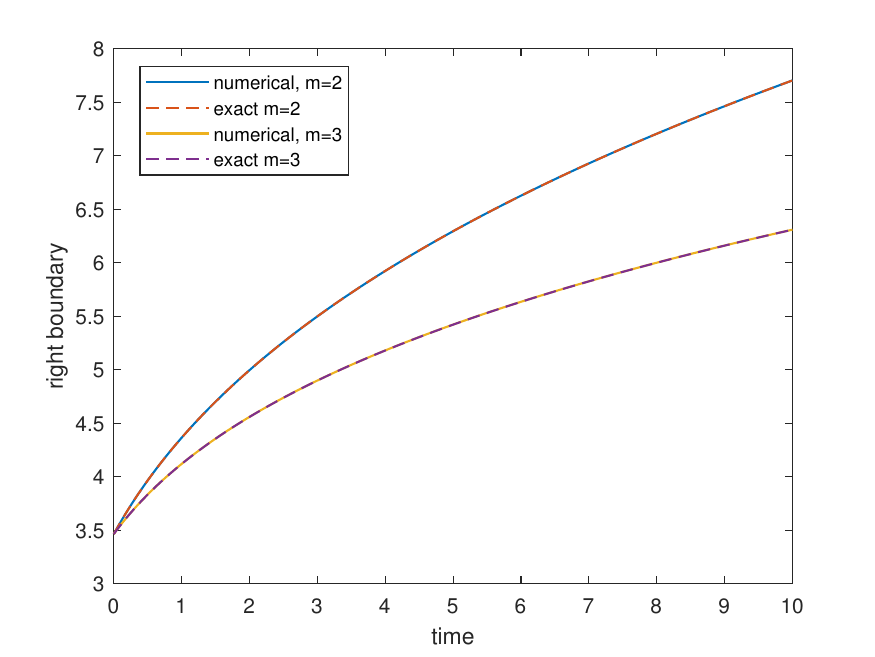}
}
	\subfigure[Energy  when $m=2$, $3$]{
		\includegraphics[width=0.45\textwidth]{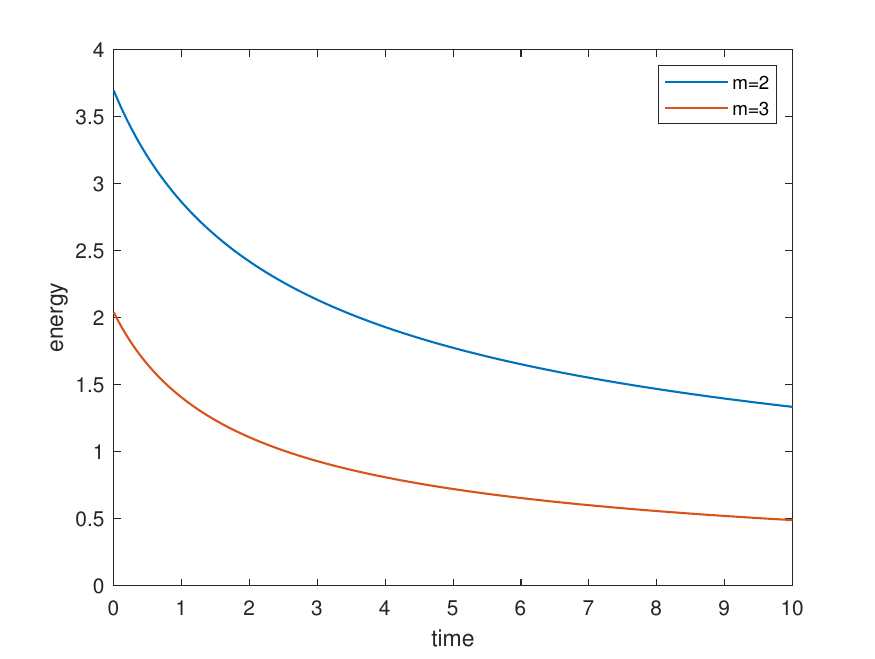}
	}
	\subfigure[Mass when  $m=2$]{
	\includegraphics[width=0.45\textwidth]{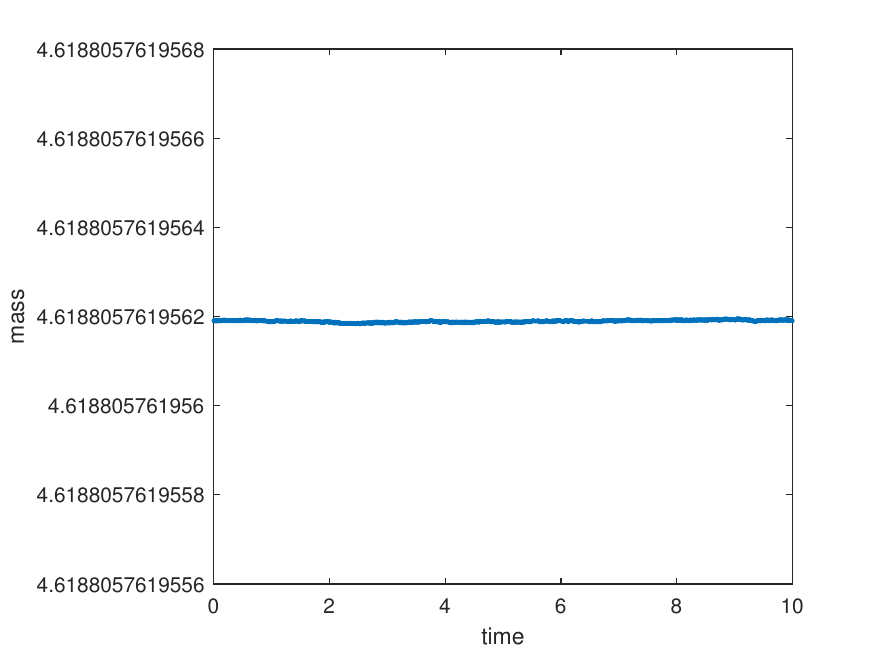}
}
	\caption{The evolutions of density, energy and mass  for the Barenblatt solution with $M=800$, $\delta t=1/6400$.}\label{fig:barenblatt}
\end{figure}

Convergence results of the trajectory and Barenblatt solution  with $m=2$ are shown in Table~\ref{convergence baren m=2}. The reference solution of  trajectory is obtained on  refine meshes, i.e.\  $M=8000$, $\delta t=1/64000$. The Barenblatt solution $B_m(x,0.5)$ is taken as the exact solution to test the convergence rate for  density $\rho$.  It is observed  in Table~\ref{convergence baren m=2} and Table~\ref{convergence baren m=2.5} that the convergence rates in time  remain to be  second order with $m=2$, and we can not achieve optimal convergence rate when we take $m=2.5$ due to  limitation of  regularity  for density $\rho$.  We also observe that the errors for density $\rho$  at $x=0$ away from the boundaries remains to be  second order convergence rate  when we take  $m=2.5$. 
\begin{table}[!htb]
	\centering
	\caption{Convergence order of trajectory $x$ and density $\rho$  with $B_m(x,T)$ at $T=0.5$, $m=2$.}
	\begin{tabular}{cccccccccc}
		\hline
		M &$\delta t$&$L_h^{2}$ error ($x$) & order & $L^{\infty}$ error ($x$) & order &\ $L_h^{2}$ error ($\rho$) & order\\ \hline
		100 &1/100& 0.0014	& & 8.3240e-04& & 5.5360e-04&\\
		200 &1/400& 3.4387e-04&2.0315 &1.8335e-04&2.1827&1.3922e-04&1.9915\\
		400 &1/1600& 8.3217e-05&2.0496& 4.4631e-05 &2.0385&	3.4935e-05& 1.9946\\
		800 &1/6400& 1.8721e-05&2.1522&  1.0040e-05&2.1523&8.7565e-06  &1.9963\\
		\hline
	\end{tabular}\label{convergence baren m=2}
\end{table}
\begin{table}[!htb]
	\centering
	\caption{Convergence order of trajectory $x$ and density $\rho$ with $B_m(x,T)$ at $T=0.5$, $m=2.5$.}
	\begin{tabular}{cccccccccc}
		\hline
		M &$\delta t$&$L_h^{2}$ error ($x$) & order &$L_h^{2}$ error ($\rho$)  & order &\ error ($0$) & order\\ \hline
		100 &1/100& 0.0011	& & 8.0390e-04& & 2.2763e-04&\\
		200 &1/400& 3.7503e-04&1.5509 &3.6544e-04&1.1374&5.7061e-05&1.9961\\
		400 &1/1600& 1.5814e-04&1.2458& 1.8045e-04 &1.0181&	1.4262e-05& 2.0003\\
		800 &1/6400& 6.5878e-05&1.2634&  8.7376e-05&1.0463&3.5611e-06  &2.0017\\
		\hline
	\end{tabular}\label{convergence baren m=2.5}
\end{table}

{\bf Waiting time}. It is known that solutions to the porous medium equation may show the phenomenon of  waiting time.  This phenomenon indicates that the support set of solutions will not expand during a positive time $t^*$, after which, it will start move at a finite speed.   $t^*$ is called the waiting time. 

To be specific, the propagation speed at the boundary for the porous medium equation can be calculated by  \cite{duan2019pme,duan2021structure}
\begin{align*}
\partial_tx=-\frac{m}{m-1}\frac{\partial_X(\rho(X,0))^{m-1}}{(\partial_Xx)^{m}}.
\end{align*}
The numerical waiting time can be calculated as the first instance such that $\partial_tx\neq0$ as stated in \cite{duan2019pme,duan2021structure}. 
Considering the following initial value:
\begin{align}\label{initial:wt}
	\rho_0(x)=\left(\frac{m-1}{m}\left((1-\theta)\sin^2(x)+\theta\sin^4(x)\right)\right)^{1/(m-1)},\qquad x\in[-\pi,0],
\end{align}
where $\theta\in[0,0.25]$. For the initial value \eqref{initial:wt}, $m=2$ and $\theta=0.25$, we use scheme \eqref{schem:1}-\eqref{schem:1-2} without regularization term to calculate the interior points, and use \eqref{appendix:pme_boundary1}-\eqref{appendix:pme_boundary2} to calculate the boundaries, the numerical results are displayed in Figure~\ref{fig:waitingtime}, the free boundaries remain to be static during $0<t\le0.22$. After the moment $t=0.22$,  the free  boundaries begin to move at a finite speed.
\begin{figure}[!htb]
	\centering
		\includegraphics[width=0.55\textwidth]{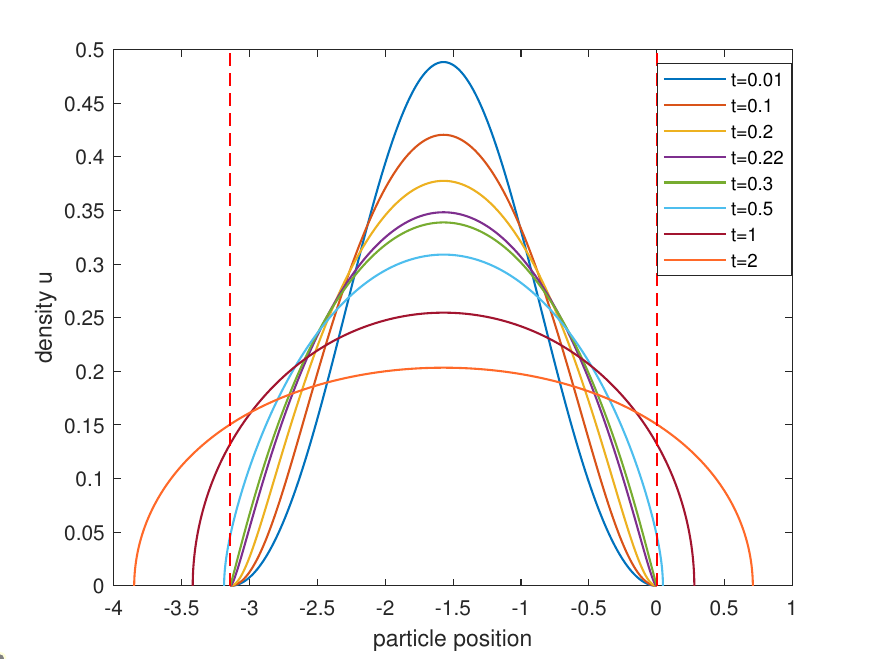}
	\caption{Density  plots for the initial value \eqref{initial:wt} with $m=2$, $\theta=0.25$, $M=800$, $\delta t=1/800$.}\label{fig:waitingtime}
\end{figure}

The waiting time for the initial value \eqref{initial:wt} is given theoretically in\cite{aronson1983initially} by $t_{w,e}:=\frac{1}{2(m+1)(1-\theta)}$. Now, by using scheme \eqref{schem:1}-\eqref{schem:1-2} without regularization term to calculate the interior points and \eqref{appendix:pme_boundary1}-\eqref{appendix:pme_boundary2} to compute the boundaries, we calculate the waiting time numerically with different $\theta$ and $m$ to compare the numerical waiting time with the exact formulation, the results are shown in Figure~\ref{fig:influencewaitingtime}, it can be observed that the tendency of the numerical waiting time is consistent with the theoretical result, and it will converge to the exact waiting time when we reduce time steps, as displayed in Table~\ref{table:wt}.
\begin{figure}[!htb]
	\centering
	\subfigure[Influence of $\theta$ with $m=2$ ]{
		\includegraphics[width=0.45\textwidth]{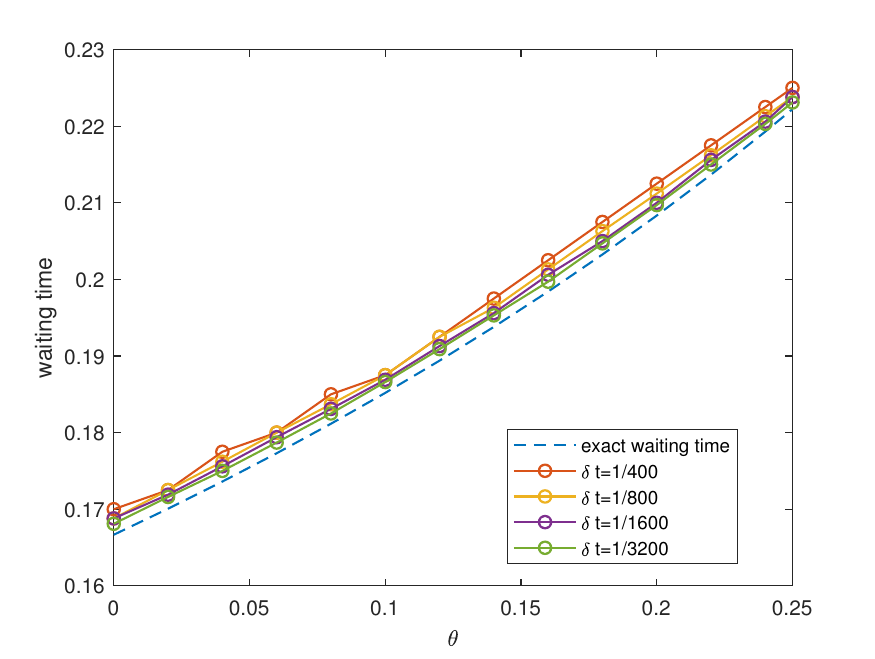}
	}
	\subfigure[Influence of $m$ with $\theta=0.25$]{
	\includegraphics[width=0.45\textwidth]{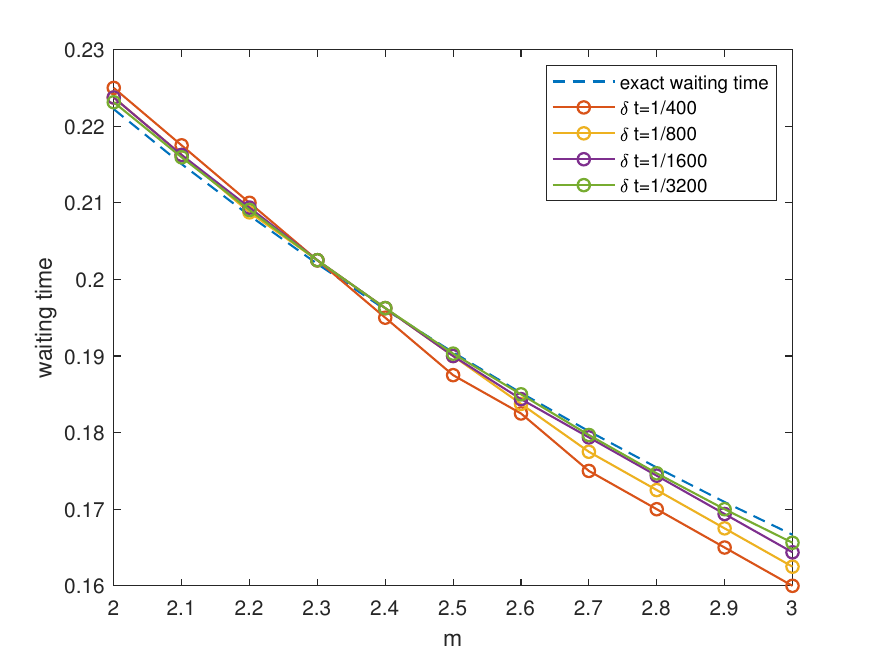}
}
	\caption{Influence of the parameter $\theta$ and $m$ with different $\delta t$ and $M=1/\delta t$.}\label{fig:influencewaitingtime}
\end{figure}
\begin{table}[!htb]
	\centering
	\caption{Convergence order of the waiting time with the initial value \eqref{initial:wt}, $m=2$, $\theta=0.25$, the exact waiting time is $t_{w,e}=\frac{1}{2(m+1)(1-\theta)}=\frac{2}{9}$.}
	\begin{tabular}{cccccccccc}
		\hline
		M &$\delta t$&$t_{w,h}$& $|t_{w,h}-t_{w,e}|$ & order\\ \hline
		1000 &1/1000& 0.2240	& 0.0018& \\
		2000 &1/2000& 0.2235& 0.0013 &0.4695\\
		4000 &1/4000& 0.2233&0.0011& 0.2410\\
		8000 &1/8000& 0.2229&0.0007&  0.6521\\
		\hline
	\end{tabular}\label{table:wt}
\end{table}

\subsubsection{Fokker-Planck equation}
In this subsection, we will discuss the Fokker-Planck equation with different potentials by choosing various $U(\rho)$, $V(x)$ and taking  $W(x)=0$. 

{\bf Nonlinear Fokker-Planck equation.} 
For the nonlinear Fokker-Planck equation, $U(\rho)$ is taken to be $\frac{1}{m-1}\rho^m$, and $V(x)$ will be taken as one-well and double-well potential, respectively. 
If $V(x)$ is a confining drift potential, all solutions will approach to  a unique steady state which is formulated as, see \cite{carrillo2001entropy,carrillo2000asymptotic,carrillo2022primal}
\begin{align}
	\rho_{\infty}(x)=\left(C_{fp}-\frac{m-1}{m}V(x)\right)_+^{\frac{1}{m-1}},
\end{align}
where $C_{fp}>0$ is determined by the mass of initial value such that $\int_{\Omega}\rho_0(x)\mathrm{d}x=\int_{\Omega}\rho_{\infty}(x)\mathrm{d}x$.
 In the following, one-well potential  $V(x)=\frac{|x|^2}{2}$ and double-well  potential $V(x)=\frac{|x|^4}{4}-\frac{|x|^2}{2}$ will be considered.  

{\bf One well}. Taking $V(x)=\frac{|x|^2}{2}$,  we consider the following energy with one-well potential:
\begin{align*}
	E(\rho)=\int_{\Omega}\frac{1}{m-1}\rho^m+\frac{|x|^2}{2}\rho\ \mathrm{d}x.
\end{align*}
Consider $m=2$ and the initial value to be 
\begin{align}
	\rho_0(x)=\max\{1-|x|,0\},
\end{align}
In this case, the  stationary solution $\rho_{\infty}$ is given  in \cite{duan2021structure}, formulated as 
\begin{align}
	\rho_{\infty}=\max\left\{\left(\frac{3}{8}\right)^{\frac{2}{3}}-\frac{x^2}{4},0\right\},
\end{align}
where $(\frac{3}{8})^{\frac{2}{3}}$ is determined by the mass conservative property, such that $\int_{\Omega}\rho_{\infty}\mathrm{d}x=\int_{\Omega}\rho_0\mathrm{d}x$. The relative energy is defined by $E(t)=E(t|\infty)/E(0|\infty)$ with $E(t|\infty)=E(t)-E(\infty)$.

Using scheme \eqref{schem:1}-\eqref{schem:1-2} and enforcing the  free boundary \eqref{fp:free boundary} without  regularization term to solve the Fokker-Planck model, the numerical results are displayed in Figure~\ref{fig:fp_one_well}. As time increases, the profile of density converges to the steady state when time goes to  $T=10$. The convergence rate for the density is computed at the stationary time $T=10$.
It can be observed that the convergence rates of $L^2$ error and error at $x=0$ are second-order.  The proposed scheme is mass-conserving, and is also energy dissipative in time. As shown in diagram (d) in Figure~\ref{fig:fp_one_well}, the scaling law of the relative energy is about $e^{-6t}$.
\begin{figure}[!htb]
	\centering
		\subfigure[Evolution of density ]{
		\includegraphics[width=0.44\textwidth]{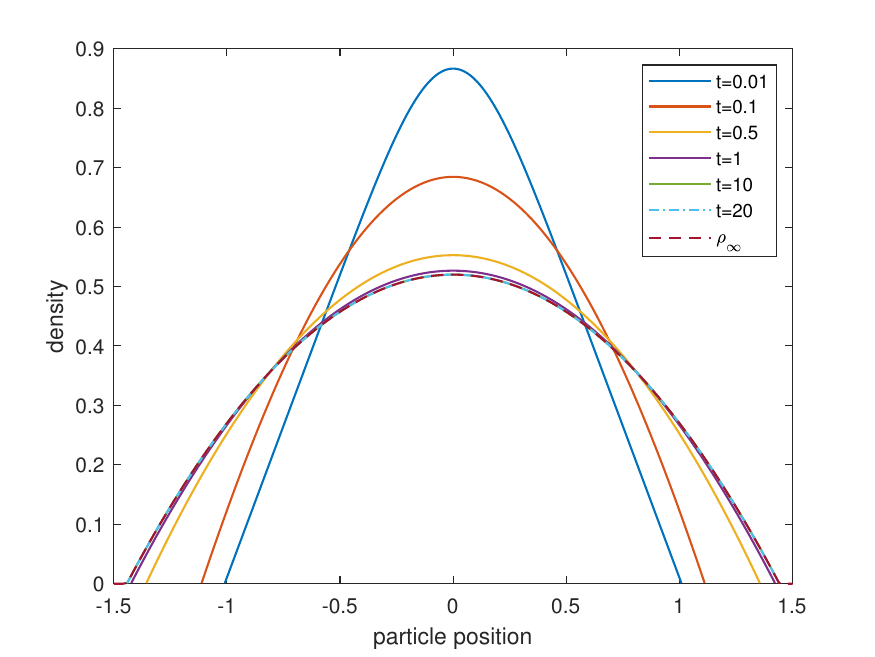}
	}	
\subfigure[Energy and mass ]{
	\includegraphics[width=0.44\textwidth]{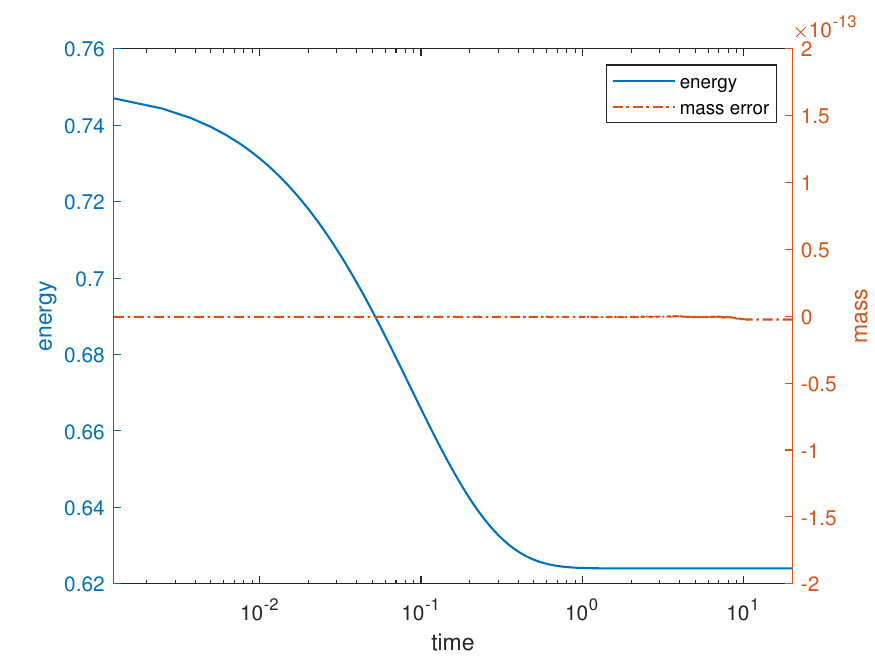}
}
		\subfigure[Convergence rates of $\rho$ at $T=10$ ]{
	\includegraphics[width=0.44\textwidth]{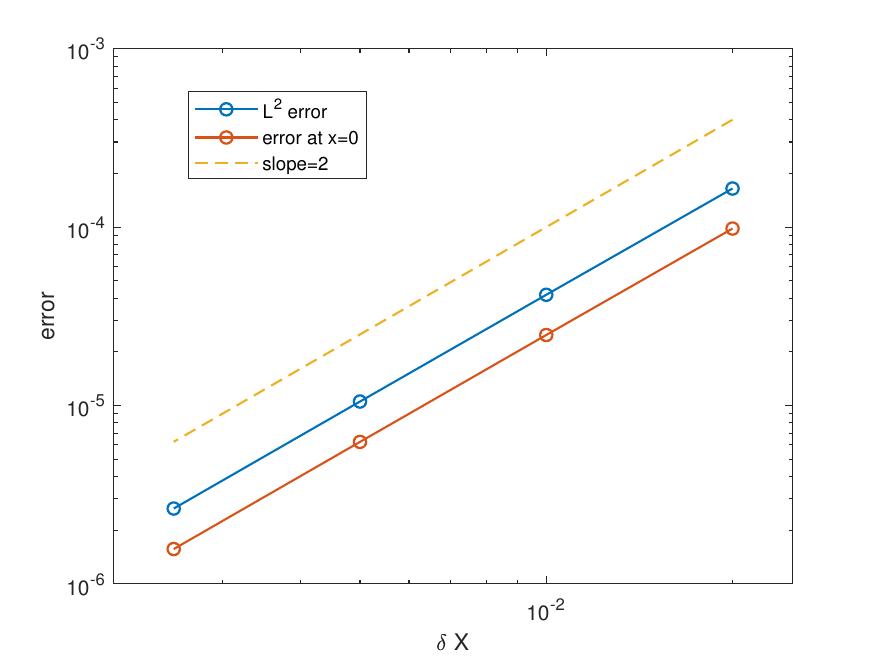}
}
	\subfigure[Scaling law ]{
	\includegraphics[width=0.44\textwidth]{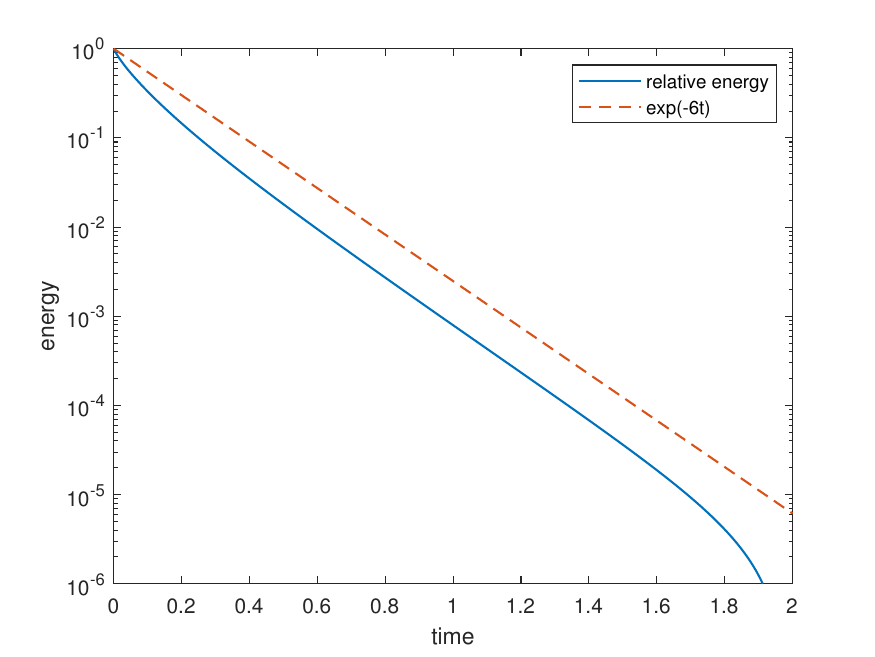}
}
	\caption{ Fokker-Planck equation with the one-well potential, $m=2$, $M=800$, $\delta t=1/800$.}\label{fig:fp_one_well}
\end{figure}

{\bf Double well}. 
Taking $V(x)=\frac{|x|^4}{4}-\frac{|x|^2}{2}$,  we consider the following energy with double-well potential:
\begin{align*}
E(\rho)=\int_{\Omega}\frac{1}{m-1}\rho^m+\left(\frac{|x|^4}{4}-\frac{|x|^2}{2}\right)\rho\ \mathrm{d}x.
\end{align*}
%
%
%
Let $m=2$, choosing the following initial value with $\sigma=1$:
\begin{align}\label{initial:double_well}
\rho_0(x)=(x^2+10^{-6}e^{-\frac{x^2}{2\sigma^2}})(1-x^2),\qquad x\in[-1,1].
\end{align}

Now, we implement numerical simulations with  initial value \eqref{initial:double_well} by using scheme \eqref{schem:1}-\eqref{schem:1-2} and \eqref{fp:free boundary} without regularization term,  where the potential is taken as $V(x)=\frac{|x|^4}{4}-\frac{|x|^2}{2}$ and $V(x)=\frac{|x|^2}{2}$, respectively, the results are shown in Figure~\ref{fig:fp_double_well}.
When the numerical solution is computed with double-well  potential $V(x)=\frac{|x|^4}{4}-\frac{|x|^2}{2}$,  the stationary state will also be  double-well.  If the potential is set to be one-well  $V(x)=\frac{|x|^2}{2}$,  the stationary state will also  be one-well.
\begin{figure}[!htb]
	\centering
		\subfigure[Density, one-well potential ]{
		\includegraphics[width=0.3\textwidth]{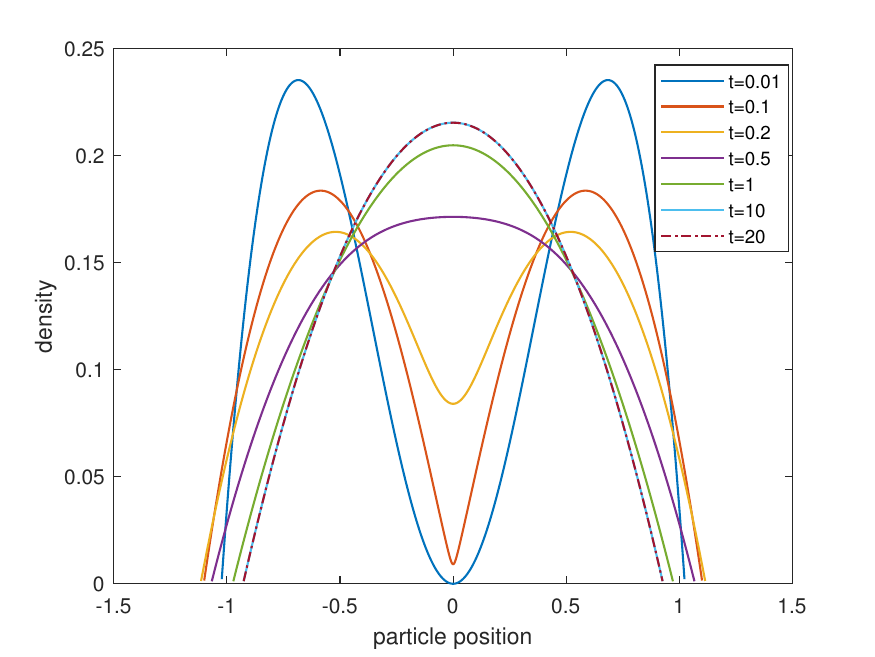}
	}
	\subfigure[Density, double-well potential]{
		\includegraphics[width=0.3\textwidth]{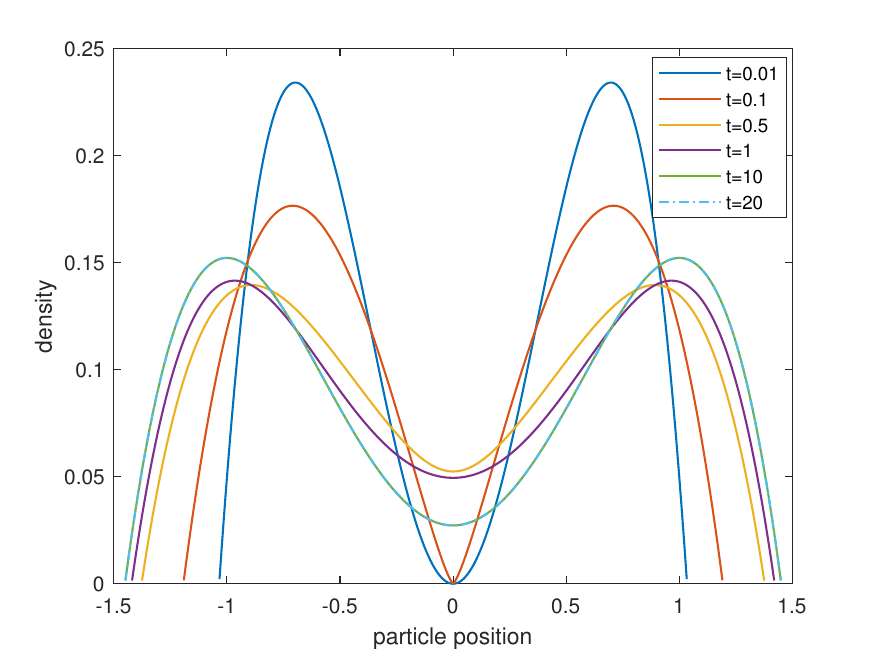}
	}
	\subfigure[Energy and mass]{
		\includegraphics[width=0.3\textwidth]{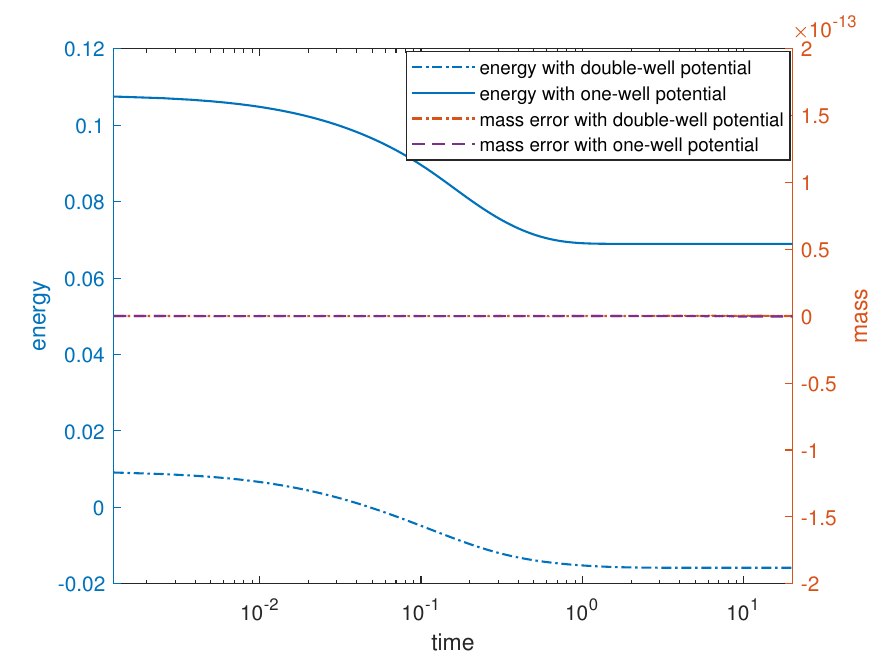}
	}
	\caption{Double-well initial problem for the Fokker-Planck equation with the one-well potential and the double-well potential, $m=2$, $M=800$, $\delta t=1/800$.}\label{fig:fp_double_well}
\end{figure}

If we take the following  one-well value as the  initial condition:
\begin{align}
	\rho_0(x)=1-x^2,\qquad x\in[-1,1],
\end{align}
and calculate the numerical solution with the double-well potential, it will also converge to a double-well stationary state, as displayed in Figure~\ref{fig:fp_one_well_initial,double_well_potential}.
\begin{figure}[!htb]
	\centering
	\subfigure[Density, double-well potential ]{
		\includegraphics[width=0.3\textwidth]{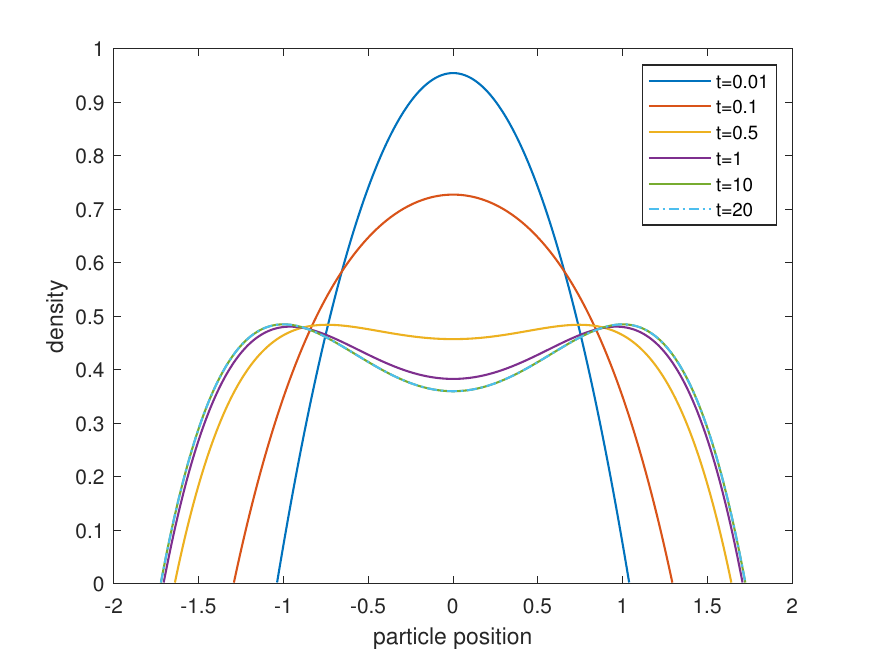}
	}
	\subfigure[Energy dissipation]{
		\includegraphics[width=0.3\textwidth]{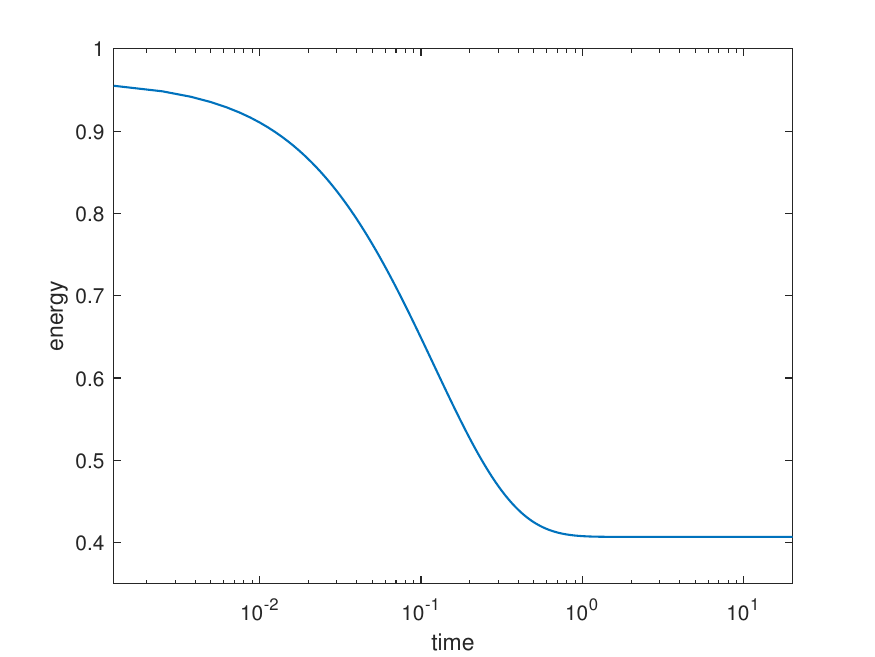}
	}
	\subfigure[Mass conversation]{
		\includegraphics[width=0.3\textwidth]{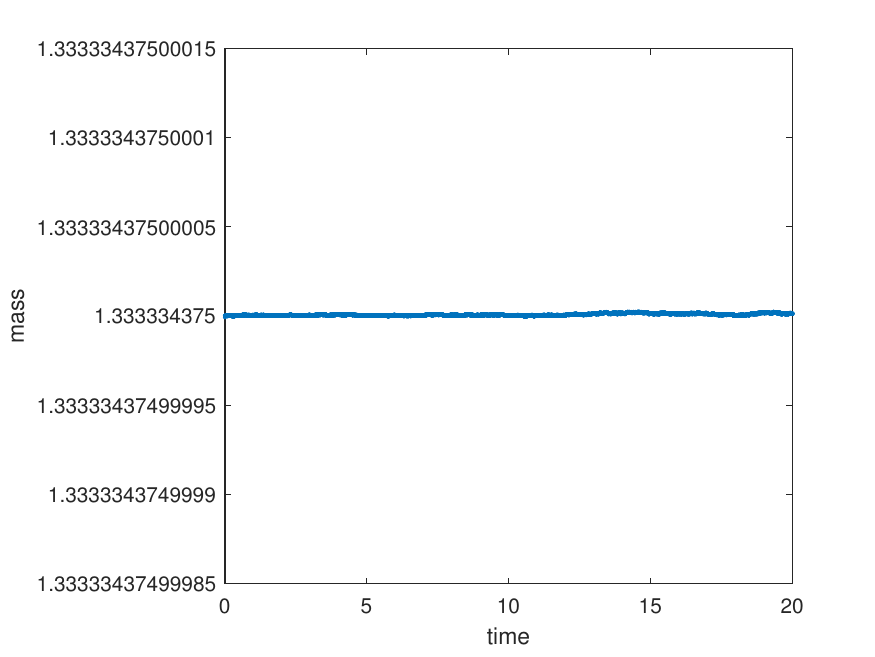}
	}

	\caption{One-well initial value for the Fokker-Planck equation with the double-well  potential, $m=2$, $M=800$, $\delta t=1/800$.}\label{fig:fp_one_well_initial,double_well_potential}
\end{figure}

{\bf Linear Fokker-Planck equation with the logarithmic potential.} 
Now, we consider the linear Fokker-Planck equation, i.e.\, taking $U(\rho)=\rho\log\rho$. 
Consider the following  energy with one-well logarithmic potential :
\begin{align*}
E(\rho)=\int_{\Omega}\rho\log\rho+\frac{|x|^2}{2}\rho\ \mathrm{d}x.
\end{align*}
Numerical experiments are implemented by using scheme \eqref{schem:1}-\eqref{schem:1-2} without regularization term, and the following initial condition is considered:
\begin{align}\label{initial:gauss}
	\rho_0(x)=\frac{C_{g}}{\sqrt{2\pi}}e^{-x^2/\sigma},\qquad x\in[-5,5],
\end{align}
with $C_{g}=\sigma=1$. The results  shown in Figure~\ref{fig:fp_log_potential} imply that numerical solution will converge to a one-well stationary state, and the curve of energy is dissipative with respective to time.
\begin{figure}[!htb]
	\centering
	\subfigure[Density ]{
		\includegraphics[width=0.3\textwidth]{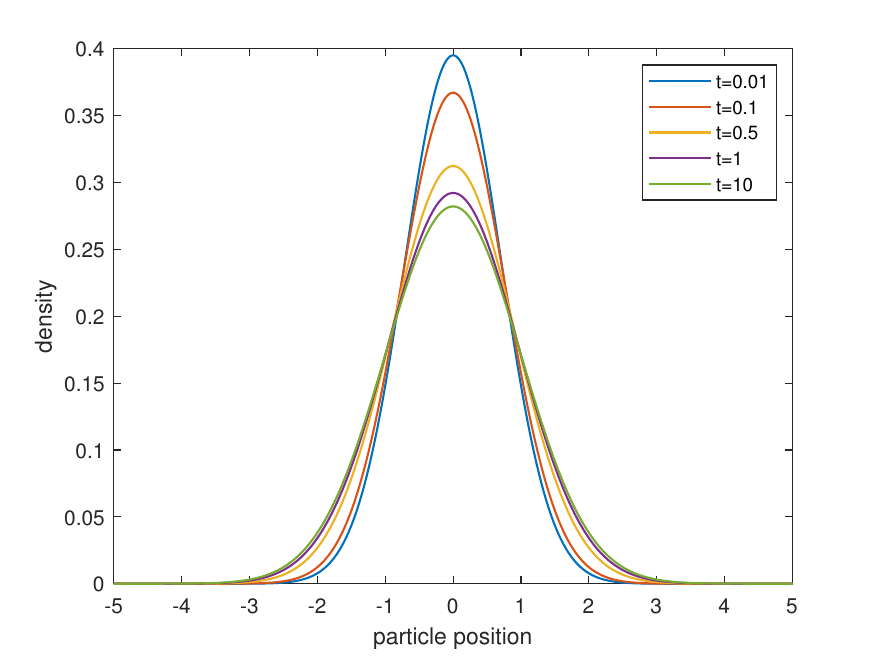}
	}
	\subfigure[Energy dissipation]{
		\includegraphics[width=0.3\textwidth]{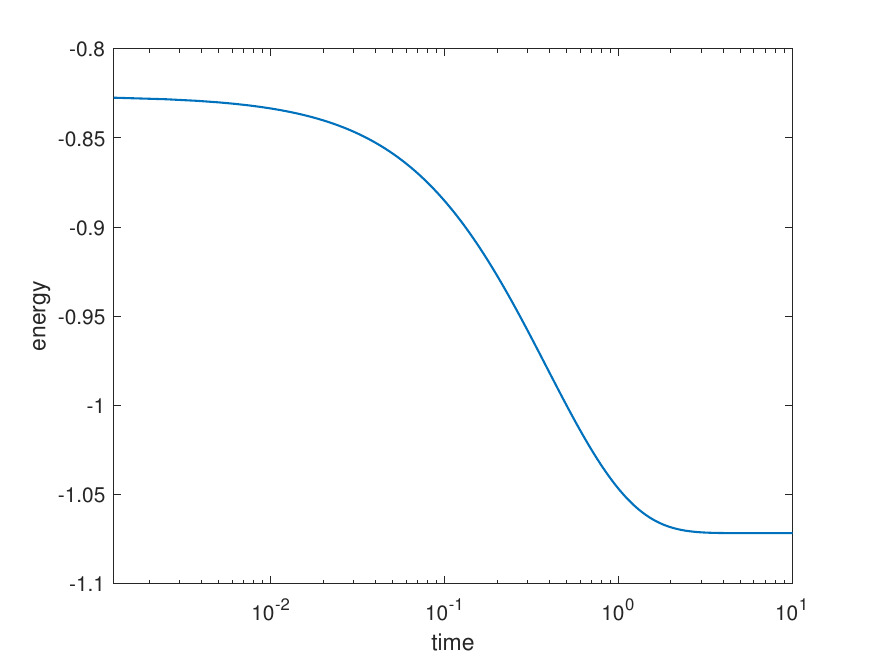}
	}
	\subfigure[Mass conservation]{
		\includegraphics[width=0.3\textwidth]{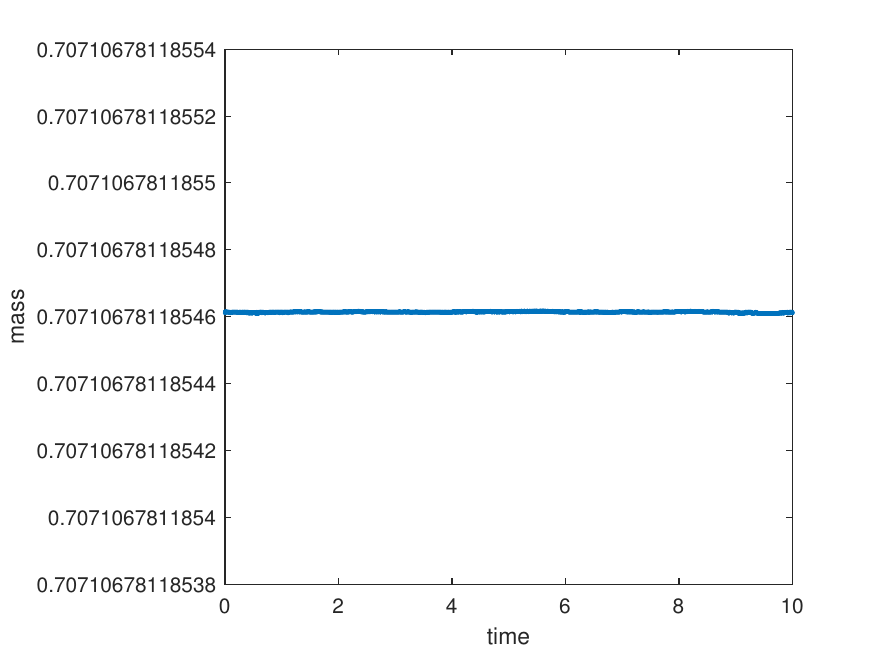}
	}
	
	\caption{One-well initial value for the Fokker-Planck equation with logarithmic  potential, $M=800$, $\delta t=1/800$.}\label{fig:fp_log_potential}
\end{figure}
\subsubsection{Aggregation equation}
Considering the aggregation equation with the energy defined by
\begin{align*}
E(\rho)=\int_{\Omega\times\Omega}W(x-y)\rho(x)\rho(y)\mathrm{d}x\mathrm{d}y,
\end{align*}
with $W(x)=\frac{|x|^2}{2}-\ln|x|$. 
We show that taking variational of $E(\rho(x))$ with respect to $x$ leads to the following equations in Eulerian coordinate:
 \begin{align*}
\frac{\delta E}{\delta x}=&\rho\nabla_xF'(\rho)=\rho\nabla_x\left(\int_{\Omega}W(x-y)\rho(y)\mathrm{d}y\right)=\rho\int_{\Omega}W'(x-y)\rho(y)\mathrm{d}y,
\end{align*}
and in Lagrangian coordinate:
\begin{align*}
\frac{\delta E}{\delta x}=\rho(X,0)\int_{\Omega}W'(x-y)\rho(y)\mathrm{d}y.
\end{align*}

Let's set $x$ to be implicit and $y$ to be explicit, details can be reached in Appendix. Taking the initial value  \eqref{initial:gauss} with $C_g=\sigma=1$, we apply scheme \eqref{schem:1}-\eqref{schem:1-2}, incorporating the regularization term $\epsilon\Delta_X x^{k+1}$, and define the last term of \eqref{schem:1} as provided in \eqref{appendix:aggre_xi_ye}, to simulate the numerical experiments. 
 The density plot at $t=10$ and energy plot with different regularization parameter $\epsilon$ are shown in Figure~\ref{fig:aggre_epsil}. It can be observed that the density achieves the equilibrium state \cite{carrillo2012mass}
\begin{align}\label{eq:infty}
\rho_{\infty}=\frac{C_{ag}}{\pi}\sqrt{(2-x^2)_+},
\end{align}
 as $\epsilon$ decreases, where $C_{ag}=1/\sqrt{2}$ is determined by the property of  mass conservation. 
 
Now we use scheme \eqref{schem:1}-\eqref{schem:1-2} with the regularization term $\epsilon\Delta_X x^{k+1}$, $\epsilon=10^{-4}\delta t$ to make numerical experiments by choosing the initial value \eqref{initial:gauss} with $\sigma=C_g=1$. As shown in Figure~\ref{fig:aggre1,0.5}, the solution will converge to the equilibrium state. If we take the initial value \eqref{initial:gauss} with $\sigma=0.1$, $C_{g}=1$, it will converge to $\rho_{\infty}$ \eqref{eq:infty} with $C_{ag}=\sqrt{0.1/2}$ where the numerical results are displayed in Figure~\ref{fig:aggre0.1}. 
Both results indicates that the numerical solution will converge to a stationary state  which is consistent with the theoretical result, and the proposed scheme is mass conserving and energy dissipating. As shown in diagram (d) in Figure~\ref{fig:aggre1,0.5} and Figure~\ref{fig:aggre0.1}, the scaling law of the relative energy  is also verified with theoretical results.

\begin{figure}[!htb]
	\centering
	\subfigure[Density at $t=10$]{
		\includegraphics[width=0.45\textwidth]{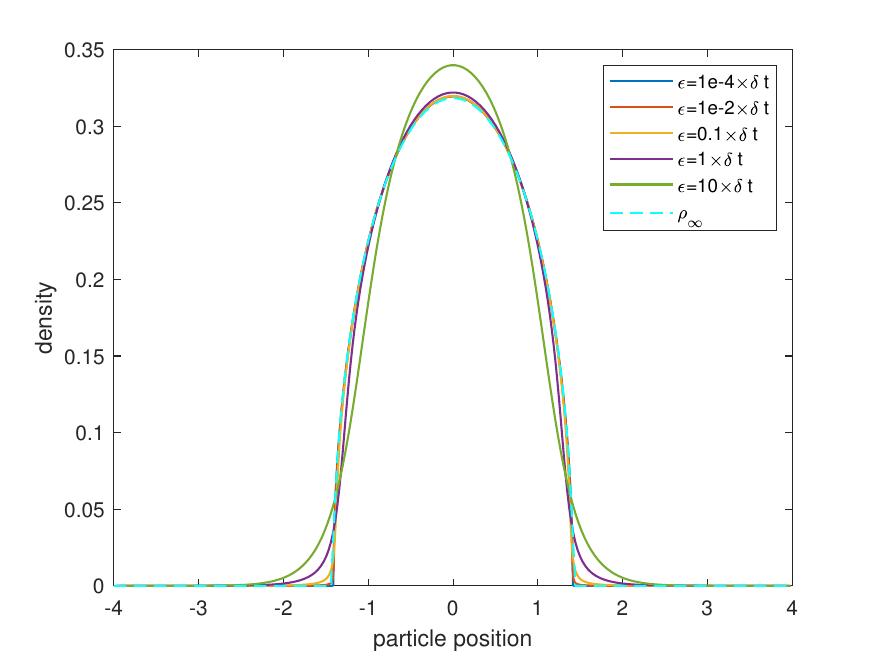}
	}
	\subfigure[Energy]{
		\includegraphics[width=0.45\textwidth]{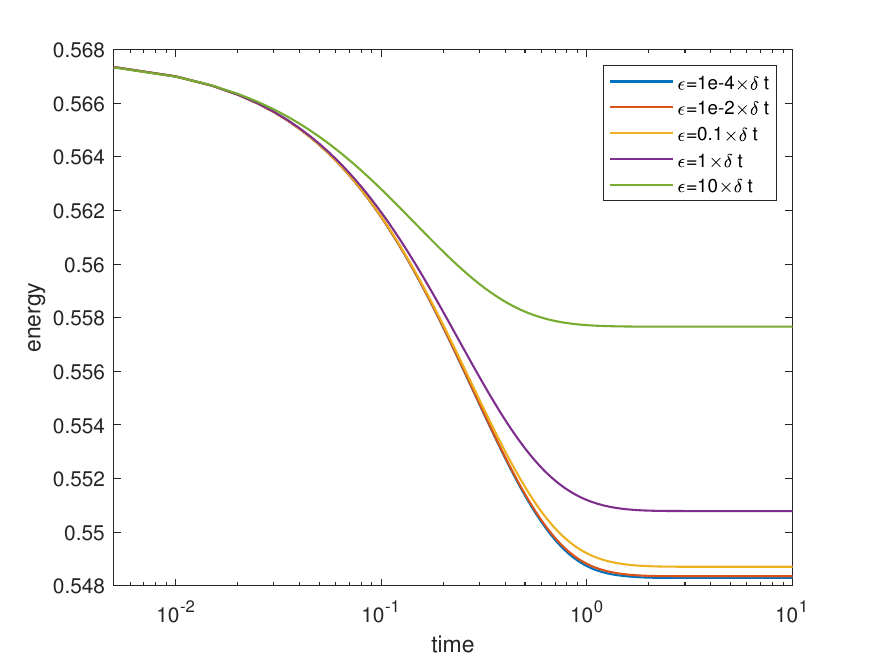}
	}
	\caption{Aggregation equation solved by \eqref{schem:1}-\eqref{schem:1-2} with the regularization term $\epsilon\Delta_X x^{k+1}$, the last term of \eqref{schem:1} is defined by \eqref{appendix:aggre_xi_ye}. The initial value \eqref{initial:gauss}, $C_g=\sigma=1$, $N=200$, $\delta t=1/200$.}\label{fig:aggre_epsil}
\end{figure}
\begin{figure}[!htb]
	\centering
	\subfigure[Density ]{
		\includegraphics[width=0.45\textwidth]{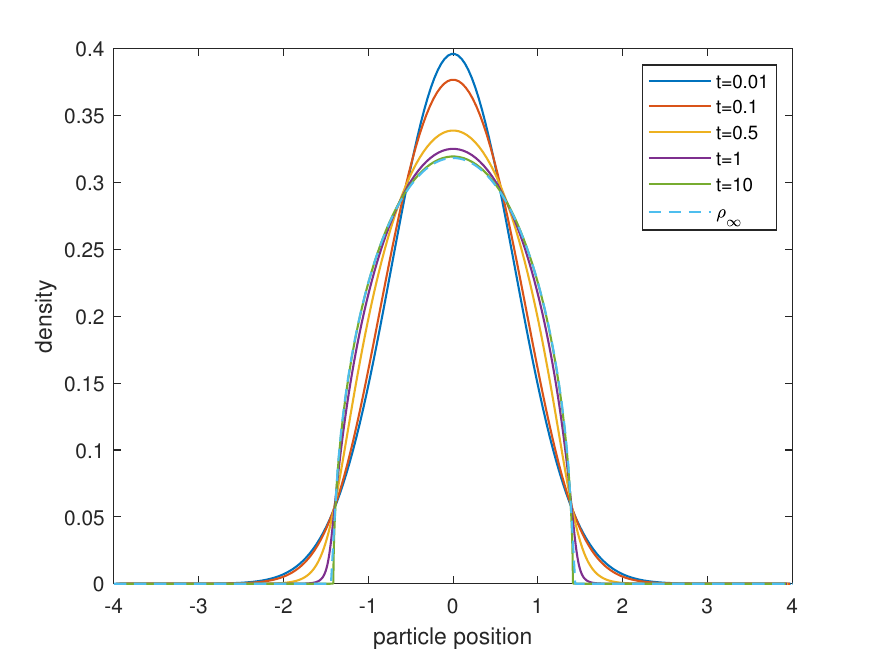}
	}
\subfigure[Mass conservation]{
	\includegraphics[width=0.45\textwidth]{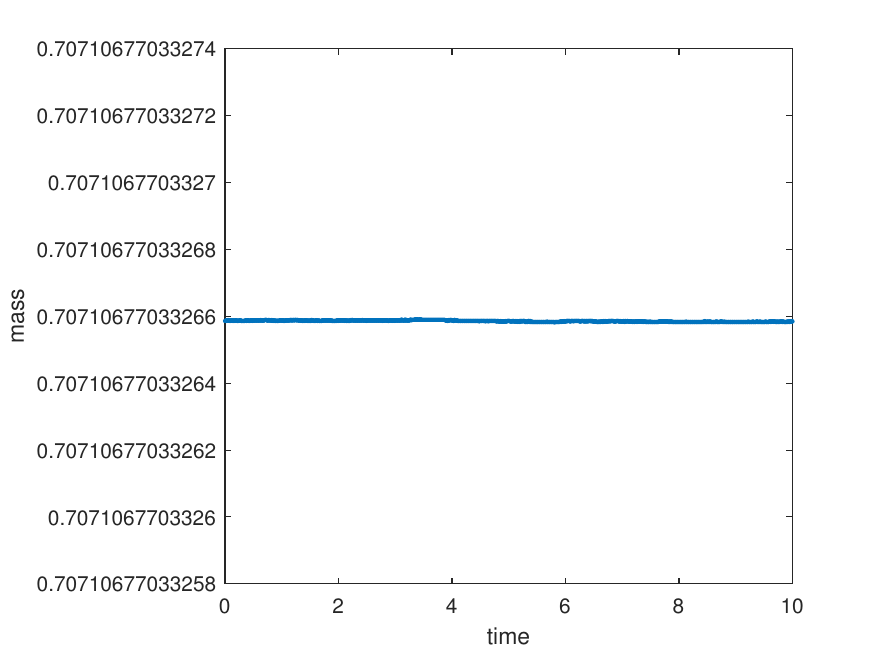}
}	
	\subfigure[Energy dissipation]{
	\includegraphics[width=0.45\textwidth]{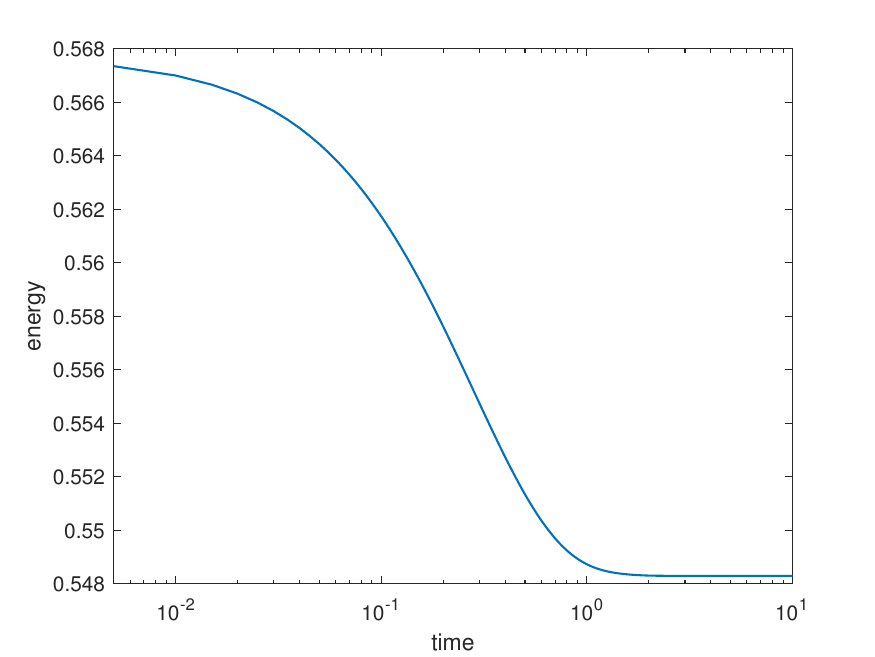}
}
	\subfigure[Relative energy]{
	\includegraphics[width=0.45\textwidth]{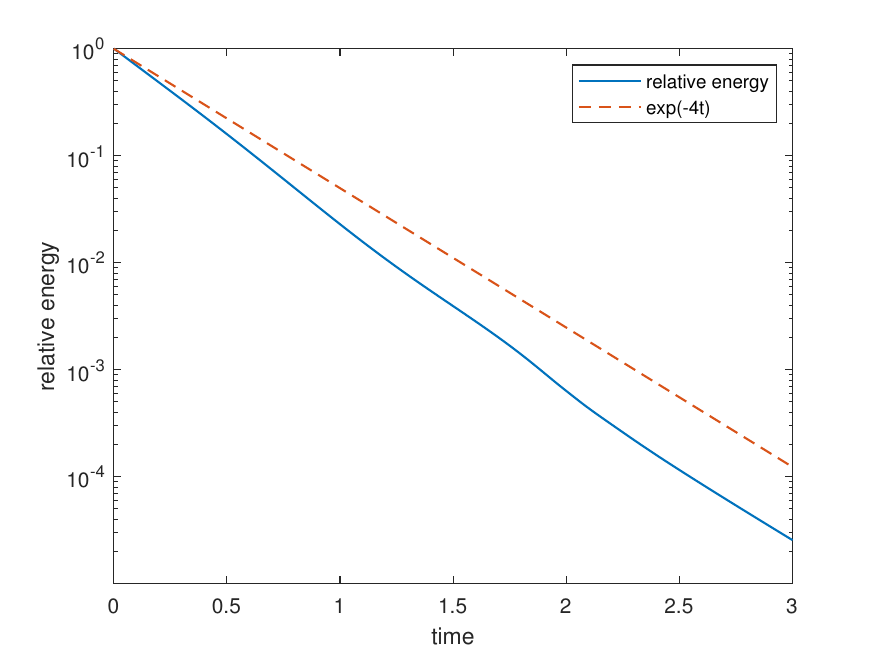}
}
	\caption{Aggregation equation solved by  \eqref{schem:1}-\eqref{schem:1-2} with the last term defined by \eqref{appendix:aggre_xi_ye}, the regularization term $\epsilon\Delta_X x^{k+1}$, $\epsilon=10^{-4}\delta t$. The initial value  \eqref{initial:gauss}, $C_g=\sigma=1$, $N=200$, $\delta t=1/200$.}\label{fig:aggre1,0.5}
\end{figure}

\begin{figure}[!htb]
	\centering
	\subfigure[Density ]{
		\includegraphics[width=0.45\textwidth]{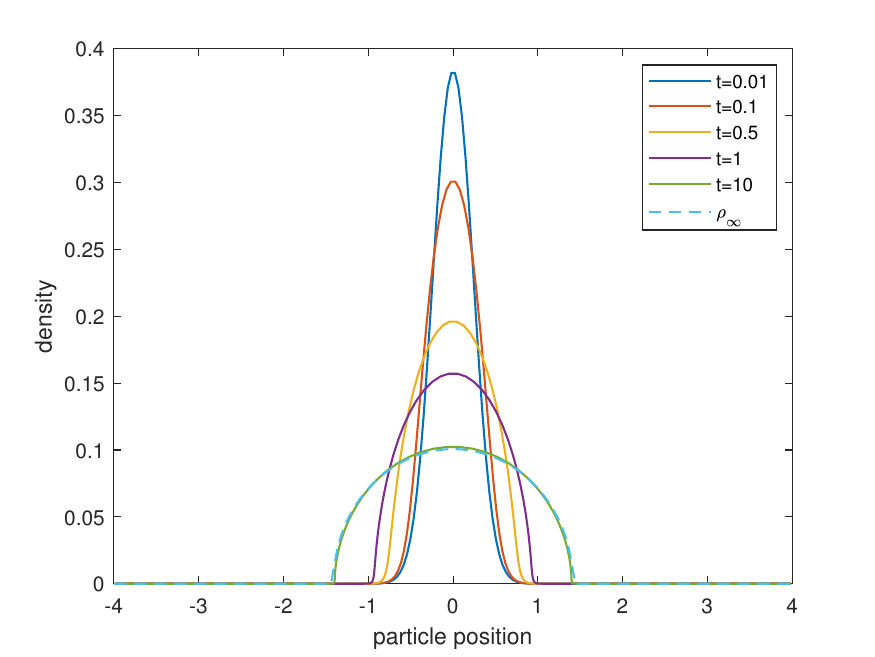}
	}
	\subfigure[Mass conservation]{
		\includegraphics[width=0.45\textwidth]{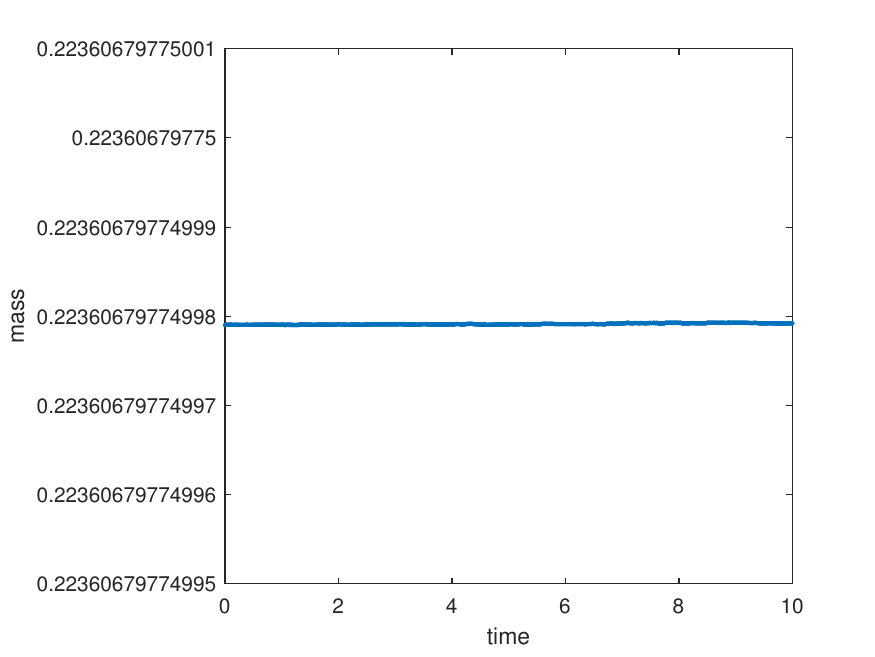}
	}	
	\subfigure[Energy dissipation]{
		\includegraphics[width=0.45\textwidth]{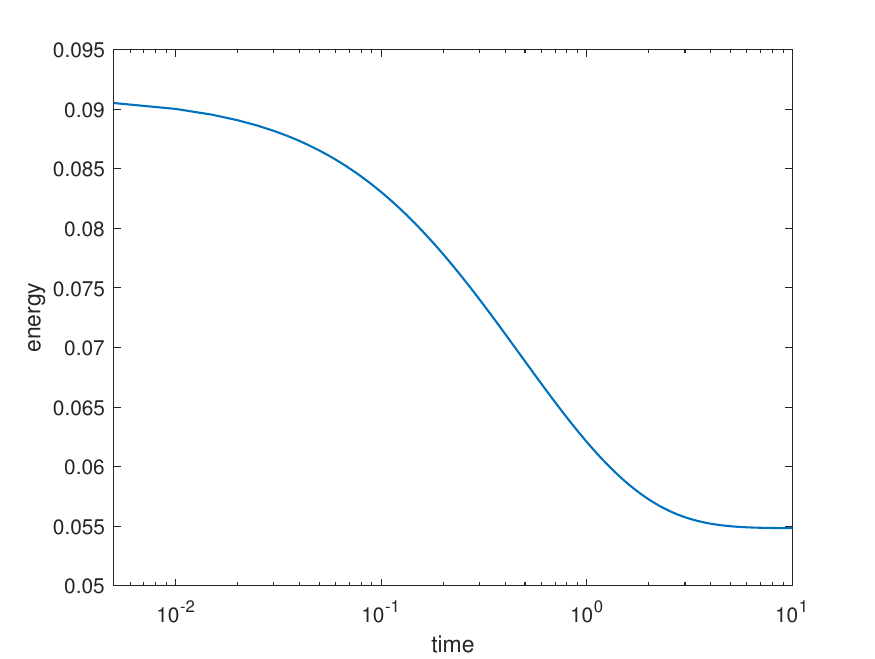}
	}
	\subfigure[Relative energy]{
		\includegraphics[width=0.45\textwidth]{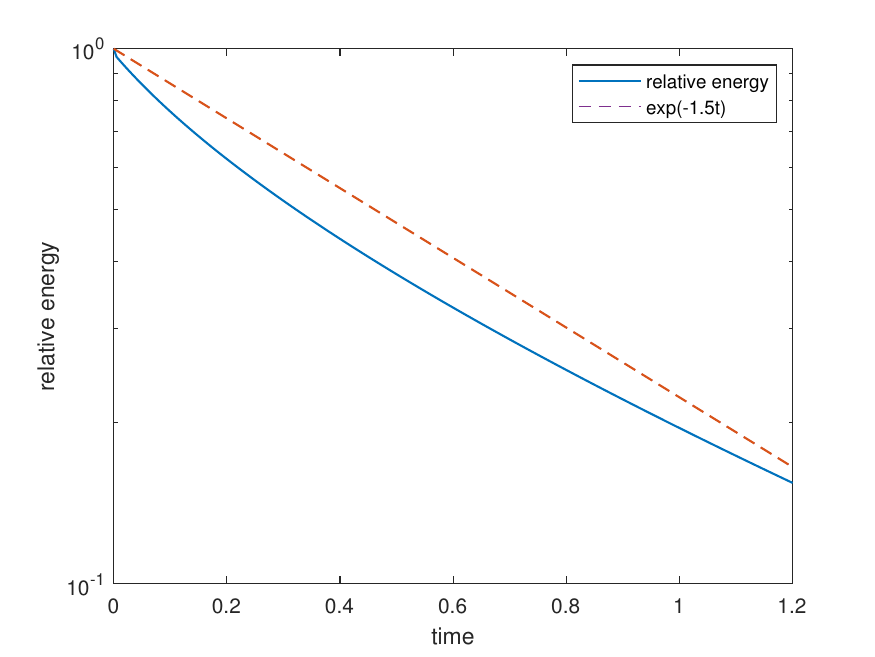}
	}
	\caption{Aggregation equation solved by \eqref{schem:1}-\eqref{schem:1-2} with the last term defined by \eqref{appendix:aggre_xi_ye}, the regularization term $\epsilon\Delta_X x^{k+1}$, $\epsilon=10^{-4}\delta t$. The initial value \eqref{initial:gauss}, $C_g=1$, $\sigma=0.1$, $N=200$, $\delta t=1/200$.}\label{fig:aggre0.1}
\end{figure}

For the fully explicit numerical scheme, the corresponding modified discrete energy can be taken by $\hat{E}_h^{k+1}=\sum_{j=0}^{N-1}\frac{\delta E_h}{\delta x_j}({\bm x}^k)(x_j^{k+1}-x_j^k)+E_h({\bm x}^k)$, then the last term in the numerical scheme \eqref{schem:1} will be taken as $\frac{\delta E_h}{\delta x_j}({\bm x}^k)$. Now we make the numerical experiments by using \eqref{schem:1}-\eqref{schem:1-2} with the regularization term $\epsilon\Delta _X x^{k+1}$,  $\epsilon=10^{-2}\delta t$ where numerical results are shown in Figure~\ref{fig:aggre_epsil_1e-2}. It can be observed that  the numerical solution  also converges to stationary state, and the total mass is preserved well.
\begin{figure}[!htb]
	\centering
	\subfigure[Density ]{
		\includegraphics[width=0.3\textwidth]{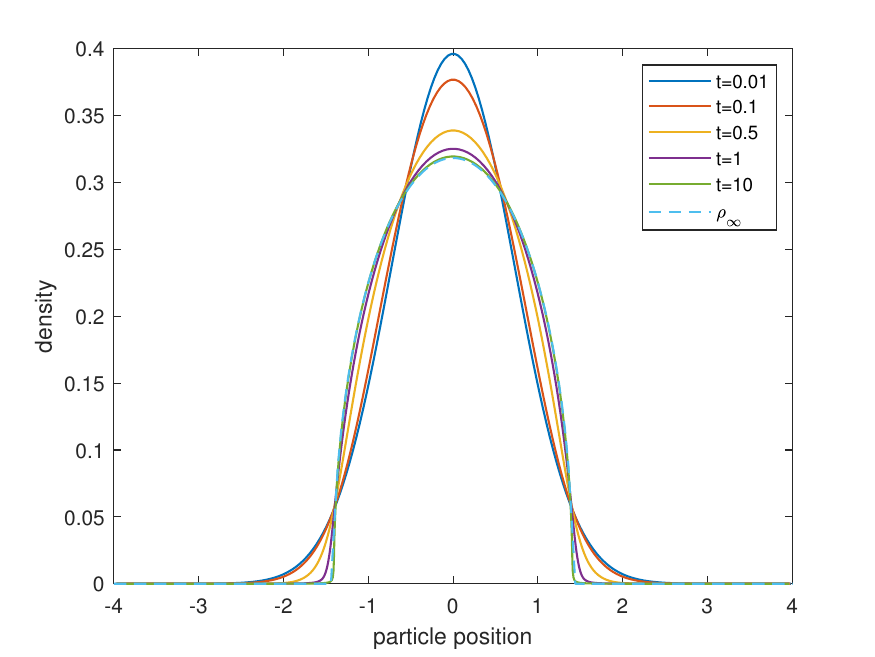}
	}
	\subfigure[Energy dissipation]{
	\includegraphics[width=0.3\textwidth]{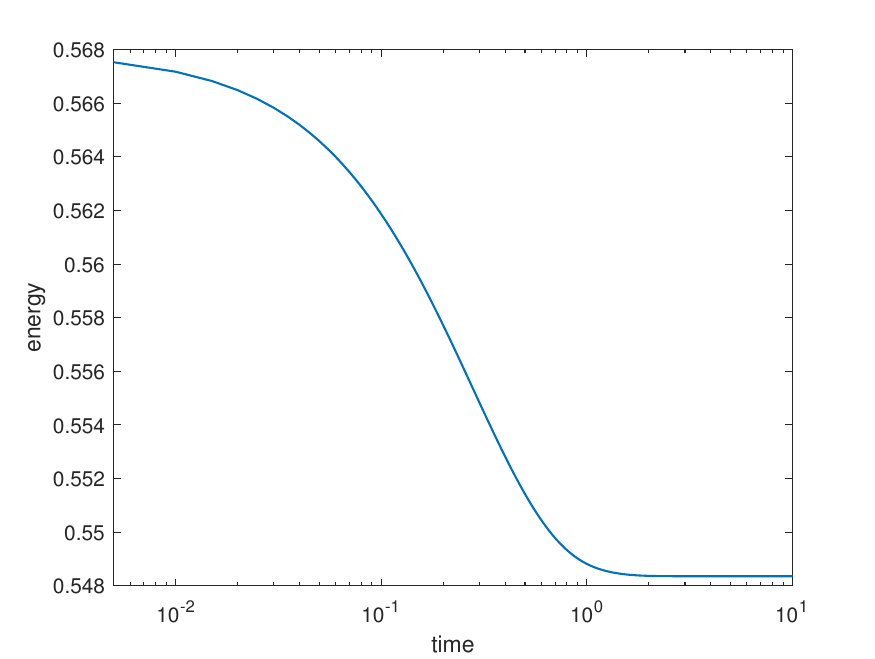}
}
	\subfigure[Mass conservation]{
		\includegraphics[width=0.3\textwidth]{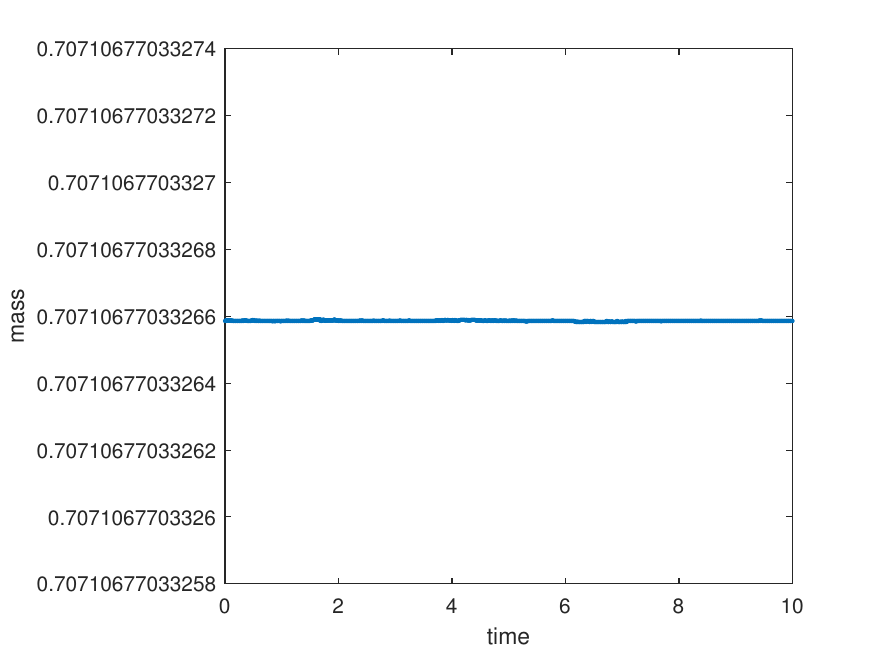}
	}	
	\caption{Aggregation equation solved by fully explicit scheme with the regularization term $\epsilon\Delta_X x^{k+1}$, $\epsilon=10^{-2}\delta t$, the initial value is taken as \eqref{initial:gauss} with $C_g=\sigma=1$, $M=200$, $\delta t=1/200$.}\label{fig:aggre_epsil_1e-2}
\end{figure}

If we set $x$ to be explicit and $y$ to be implicit in the numerical scheme, details can be found in Appendix. 
 We use scheme \eqref{schem:1}-\eqref{schem:1-2}, in which the last term is adjusted according to \eqref{appendix:aggre_xe_yi}, and the regularization term $\epsilon\Delta_X x^{k+1}$, $\epsilon=10^{-4}\delta t$ to conduct numerical experiments, the numerical results can be found in Figure~\ref{fig:aggre_x_explicit_y_implicit} which show that the scheme preserves total mass, and the numerical solution will converge to its stationary state. The energy is also dissipative with $E_h({\bm x}^{k+1})$ defined by \eqref{eq:energy} in Appendix.
\begin{figure}[!htb]
	\centering
	\subfigure[Density ]{
		\includegraphics[width=0.3\textwidth]{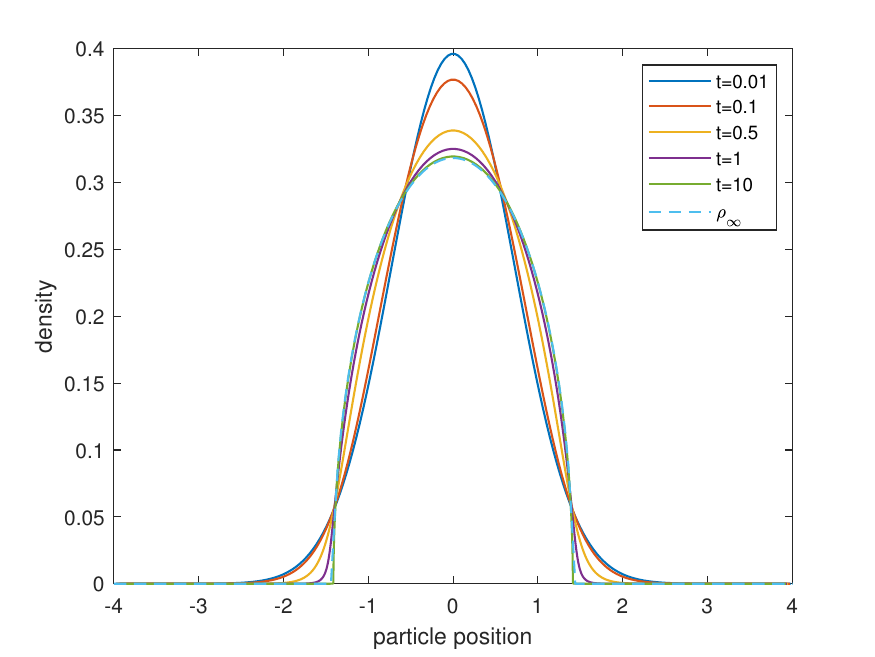}
	}
		\subfigure[Energy dissipation]{
	\includegraphics[width=0.3\textwidth]{energy_aggre_x_im_epsil_1e-4-eps-converted-to.pdf}
}
	\subfigure[Mass conservation]{
		\includegraphics[width=0.3\textwidth]{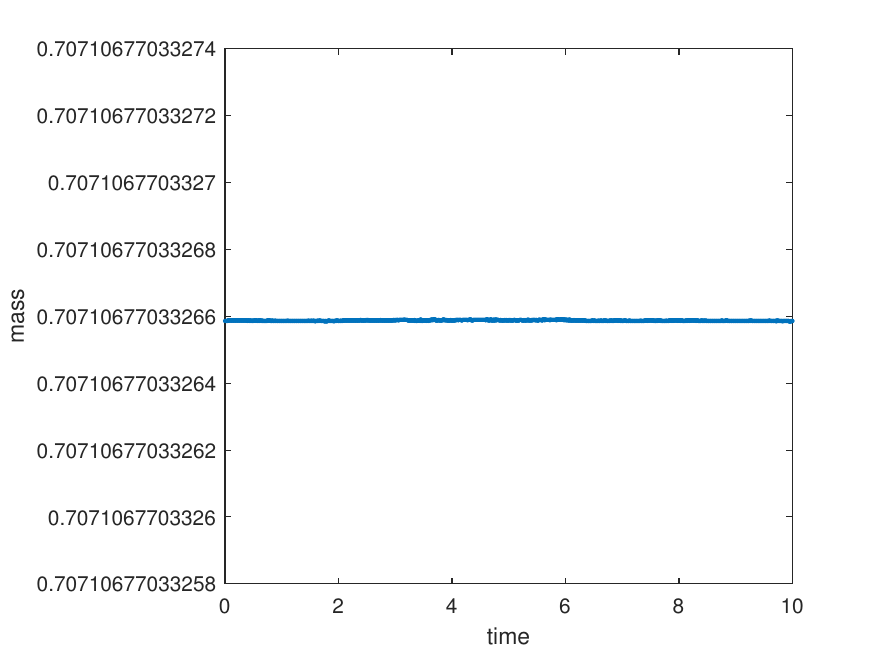}
	}	
	\caption{Aggregation equation solved by  \eqref{schem:1}-\eqref{schem:1-2} with the last term modified by \eqref{appendix:aggre_xe_yi}, the regularization term $\epsilon\Delta_X x^{k+1}$, $\epsilon=10^{-4}\delta t$. The initial value \eqref{initial:gauss}, $C_g=\sigma=1$, $N=200$, $\delta t=1/200$.}\label{fig:aggre_x_explicit_y_implicit}
\end{figure}

 We set $x$ and $y$  both implicit in numerical scheme \eqref{schem:1}-\eqref{schem:1-2}, and implement numerical experiments by using scheme \eqref{schem:1}-\eqref{schem:1-2} with the last term defined by \eqref{appendix:aggre_xi_yi}, and the regularization term $\epsilon\Delta_X x^{k+1}$, $\epsilon=10^{-2}\delta t$,	the numerical results are displayed in Figure~\ref{fig:fullly implicit} which implies that the proposed numerical scheme is also stable with $E_h({\bm x}^{k+1})$ defined  by \eqref{eq:energy} in Appendix.
\begin{figure}[!htb]
	\centering
	\subfigure[Density ]{
		\includegraphics[width=0.3\textwidth]{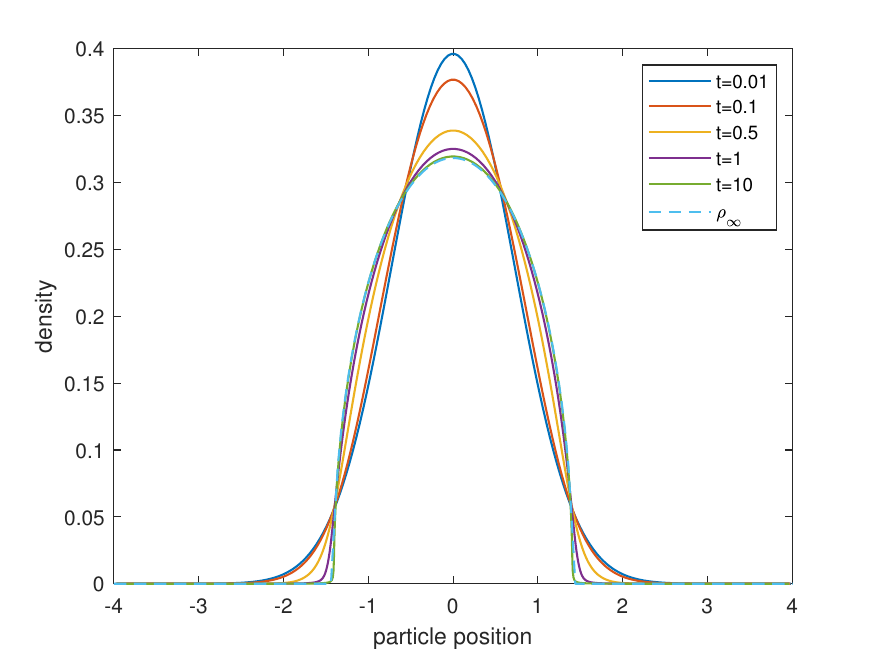}
	}
	\subfigure[Energy dissipation]{
		\includegraphics[width=0.3\textwidth]{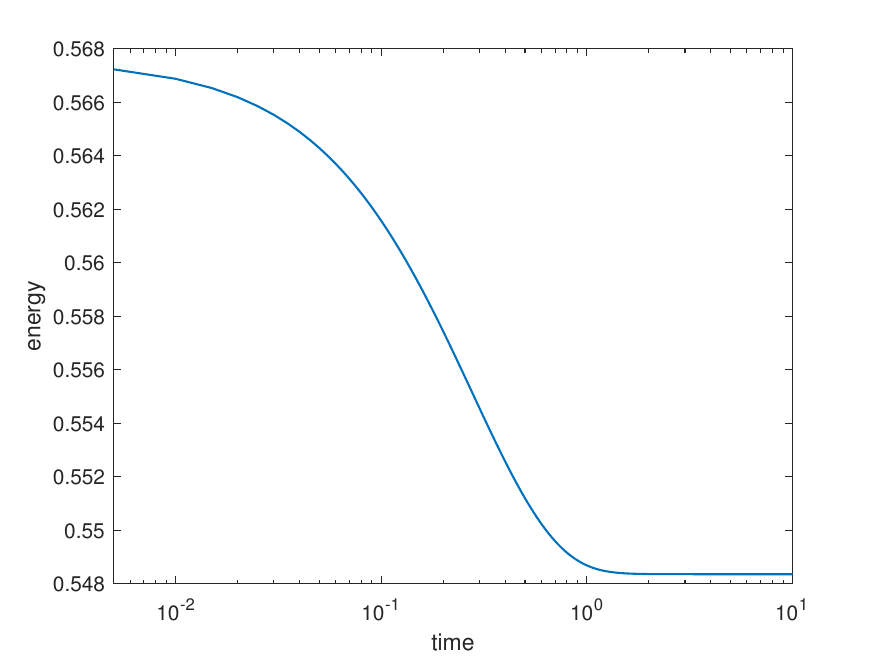}
	}
	\subfigure[Mass conservation]{
		\includegraphics[width=0.3\textwidth]{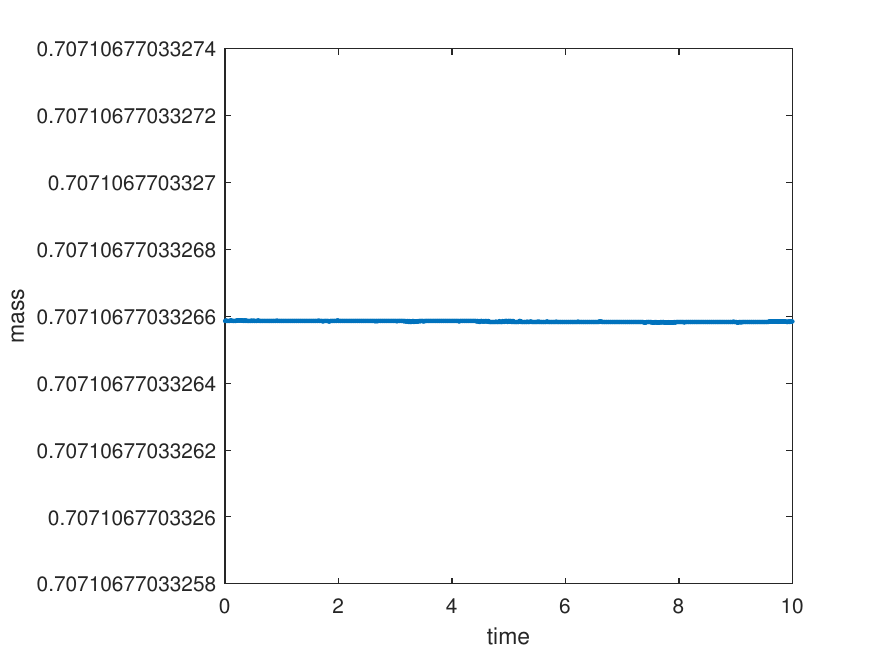}
	}	
	\caption{Aggregation equation solved by \eqref{schem:1}-\eqref{schem:1-2} with the last term defined by \eqref{appendix:aggre_xi_yi}, the regularization term $\epsilon\Delta_X x^{k+1}$, $\epsilon=10^{-2}\delta t$. The initial value \eqref{initial:gauss}, $C_g=\sigma=1$, $M=200$, $\delta t=1/200$.}\label{fig:fullly implicit}
\end{figure}

\subsubsection{Keller-Segel model}
Consider the Keller-Segel model \cite{horstmann20031970}  with the following energy:
\begin{align*}
	E(\rho)=\int_{\Omega}U(\rho(x))+\rho V(x)\mathrm{d}x+\int_{\Omega\times\Omega}W(x-y)\rho(x)\rho(y)\mathrm{d}x\mathrm{d}y,
\end{align*}
where $U=\rho\log\rho$, $V(x)=0$ and $W(x)=\frac{1}{2\pi}\ln|x|$. Taking variational derivative with respective to $x$, we have
\begin{align}\label{der}
	\frac{\delta E}{\delta x}
	=&\rho\nabla_xU'(\rho(x))+\rho\int_{\Omega}W'(x-y)\rho(y)\mathrm{d}y.
\end{align}
Rewriting \eqref{der}  into Lagrangian coordinate:
\begin{align}
\frac{\delta E}{\delta x}=\rho(X,0)\left(\frac{\partial_XU'(\frac{\rho(X,0)}{\partial_Xx})}{\partial_Xx}+\int_{\Omega}W'(x-y)\rho(Y,0)\mathrm{d}Y\right).
\end{align}

Consider the following initial value with one well :
\begin{align}\label{initial:ks}
    \rho_0(x)=\frac{C_{ks}}{\sqrt{2\pi}}e^{-x^2/2}+10^{-8},\qquad  x\in[-15,15],
 \end{align} 
 and the  initial condition with double well:
 \begin{align}\label{initial:ks_double_well}
 \rho_0(x)=\frac{C_{ks}}{\sqrt{\pi}}(e^{-4(x+2)^2}+e^{-4(x-2)^2})+10^{-8},\qquad x\in[-15,15],
 \end{align}
 combined with the Dirichlet boundary condition $x|_{\partial\Omega}=X|_{\partial\Omega}$. 
 We simulate numerical experiments by using scheme \eqref{schem:1}-\eqref{schem:1-2} without regularization term, and the last term of \eqref{schem:1} defined by \eqref{appendix:ks}. 
 
 Taking $C_{ks}=1$ in  \eqref{initial:ks} and  \eqref{initial:ks_double_well}  to implement numerical experiments, respectively.  As shown in Figure~\ref{fig:ks} and Figure~\ref{fig:ks2}, it can be observed that numerical solutions remain  to be positivity and bounded for all time. 
 If we choose $C_{ks}=5\pi$  to make  numerical simulations,  solutions also remain to be  positivity. Finally, we observe that numerical solutions  blow up in finite time shown in Figure~\ref{fig:ks} and Figure~\ref{fig:ks2}.
\begin{figure}[!htb]
	\centering
	\subfigure[Density with $C_{ks}=1$ ]{
		\includegraphics[width=0.45\textwidth]{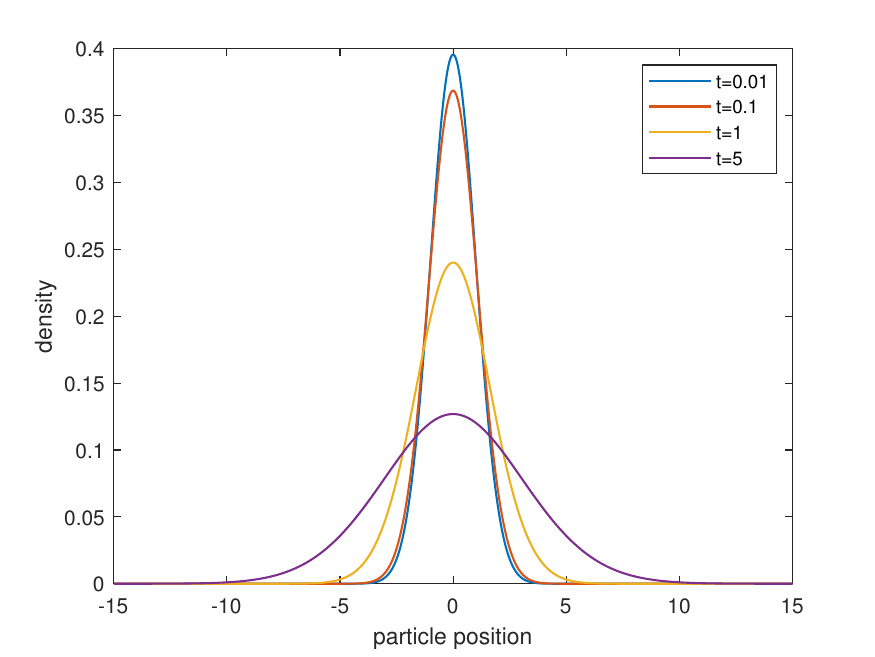}
	}
	\subfigure[Density with $C_{ks}=5\pi$ ]{
	\includegraphics[width=0.45\textwidth]{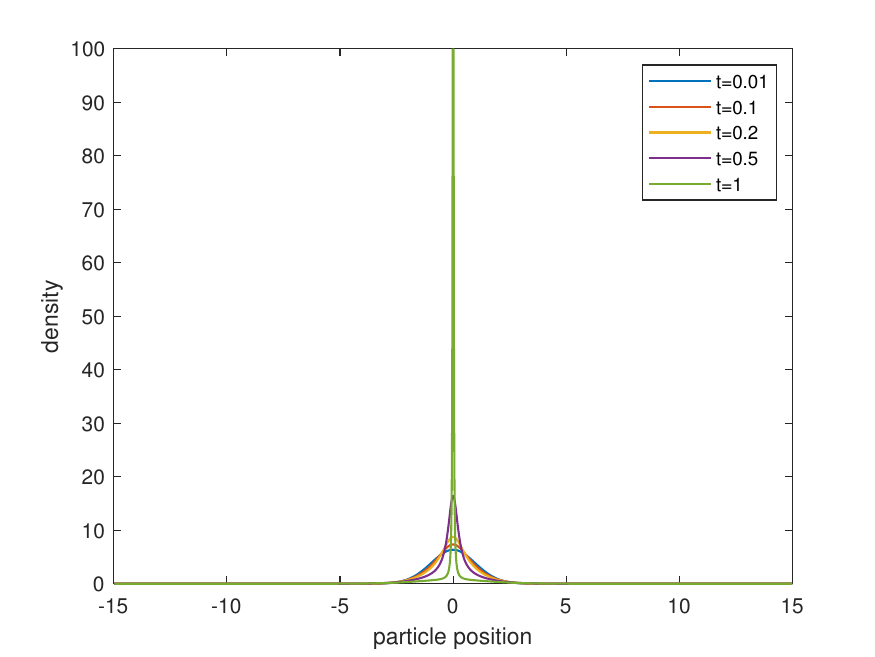}
}
	\caption{Keller-Segel model solved by  \eqref{schem:1}-\eqref{schem:1-2} without regularization term, the last term of \eqref{schem:1} is defined by \eqref{appendix:ks}. The initial value  \eqref{initial:ks} with $C_{ks}=1$ and $5\pi$,  $M=800$, $\delta t=1/800$.}\label{fig:ks}
\end{figure}

\begin{figure}[!htb]
		\centering
		\subfigure[Density  with $C_{ks}=1$]{
			\includegraphics[width=0.45\textwidth]{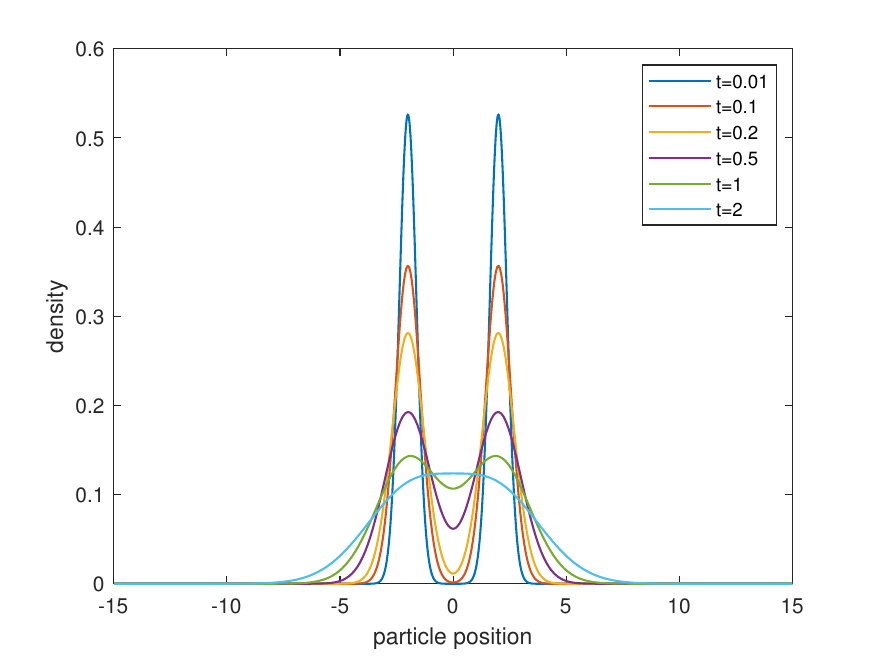}
		}
		\subfigure[Density  with $C_{ks}=5\pi$]{
		\includegraphics[width=0.45\textwidth]{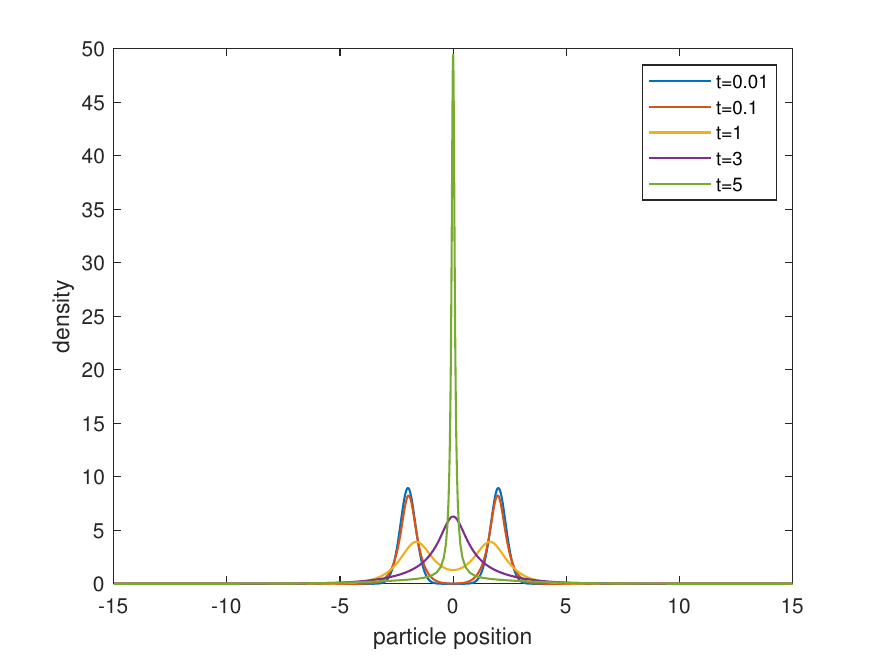}
	}
	\caption{Keller-Segel model solved by  \eqref{schem:1}-\eqref{schem:1-2} without  regularization term, the last term of \eqref{schem:1} is defined by \eqref{appendix:ks}. The double-well initial value \eqref{initial:ks_double_well}, $C_{ks}=1$ and $5\pi$, $M=800$, $\delta t=1/800$.}\label{fig:ks2}
\end{figure}

\subsection{Two dimension}\label{sec:num in 2d}
In this subsection, we focus on numerical simulations in 2D. For simplicity, to avoid solving  nonlinear equations,  we only  validate the accuracy and efficiency for explicit numerical schemes \eqref{ex:1}-\eqref{ex:2} proposed in Section \ref{sec:two dim}. We also take the Porous-Medium equation, Keller-Segel equation and Aggregation equation as examples. 
\subsubsection{Porous-Medium equation}
The 2-D Barenblatt solution takes on the following form: 
\begin{align}
	B_{m,2}(\bm{x},t)=\left(C_{B2}-\frac{\kappa(m-1)}{4m}\frac{|\bm{x}|^2}{(t+1)^{\kappa}}\right)_{+}^{1/(m-1)},
\end{align}
where $\kappa=1/m$, and $C_{B2}$ is a positive constant. The solution has a compact support $|{\bm x}|\le \xi_m(t)$ for any finite time with
\begin{align}\label{pme:2d support set}
	\xi_m(t)=\sqrt{\frac{4mC_{B2}}{\kappa(m-1)}}(t+1)^{\kappa/2}.
\end{align}
Now, we take the Barenblatt solution with $C_{B2}=0.1$ as initial condition to  make numerical experiments by using scheme \eqref{ex:1}-\eqref{ex:2} with the regularization term $\epsilon\Delta_{\bm X}{\bm x}^{k+1}$, $\epsilon=10^{-3}\delta t$ for $m=2$, $\epsilon=10^{-1}\delta t$ for $m=5$. 
The evolution of the numerical solution is shown in Figure~\ref{fig:pme,2} and Figure~\ref{fig:pme,5} for $m=2$, $m=5$, respectively, it can be found that  the profile of density changes as the trajectory moves outward,  and the free boundaries move at a finite speed. The plots of the absolute error between the numerical solution and the exact solution are also displayed, it can be observed that the main part of the error is around the free boundaries.

The trajectories at various time are shown in Figure~\ref{fig:pme,22} and Figure~\ref{fig:pme,52} for $m=2$ and $m=5$, respectively, which implies that the tendency of the numerical interface is consistent with the exact one calculated by \eqref{pme:2d support set}.
\begin{figure}[!htbp]
	\centering
	\subfigure[Numerical solution at $t=1$ ]{
		\includegraphics[width=0.3\textwidth]{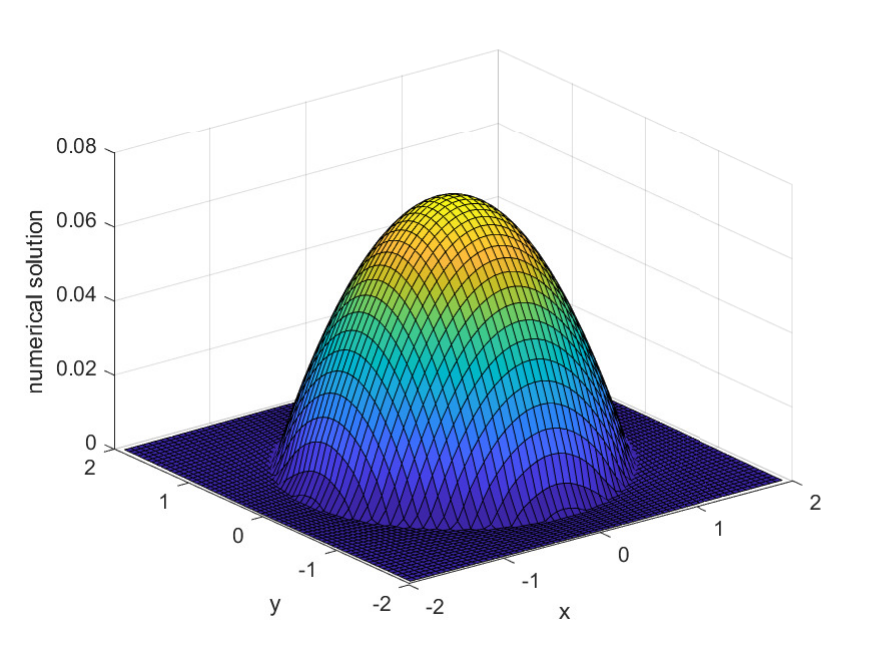}
	}
	\subfigure[Numerical solution at $t=1$]{
		\includegraphics[width=0.3\textwidth]{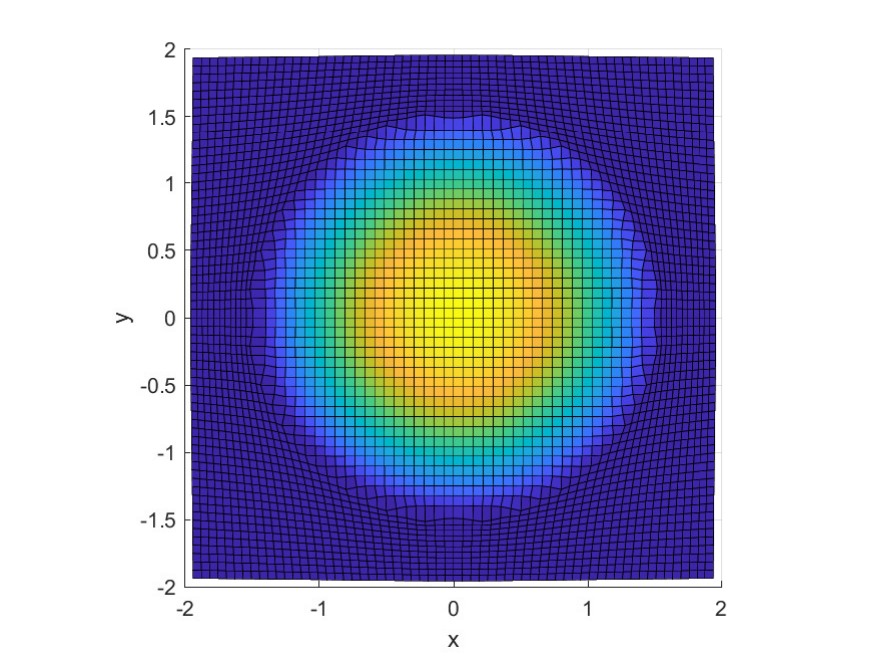}
	}
	\subfigure[$|B_{2,2}({\bm x},t)-\rho_h|$ at $t=1$]{
	\includegraphics[width=0.3\textwidth]{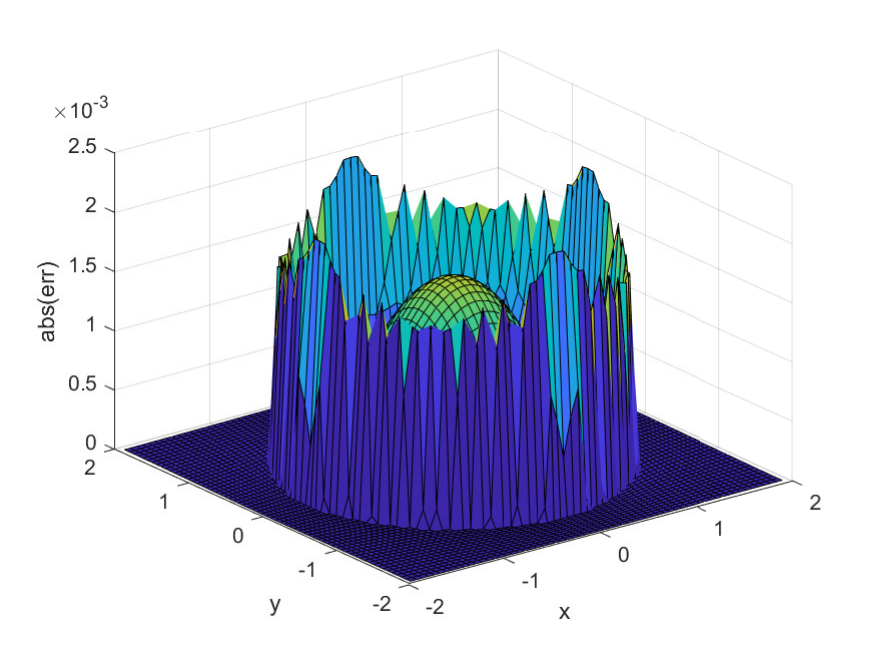}
}
	\subfigure[Numerical solution at $t=2$ ]{
		\includegraphics[width=0.3\textwidth]{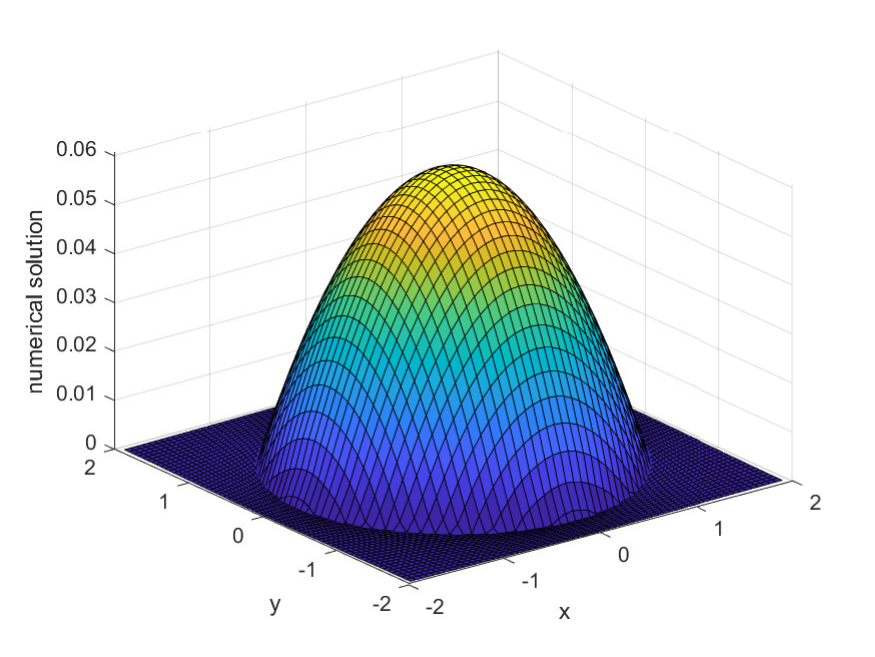}
	}
	\subfigure[Numerical solution at $t=2$]{
		\includegraphics[width=0.3\textwidth]{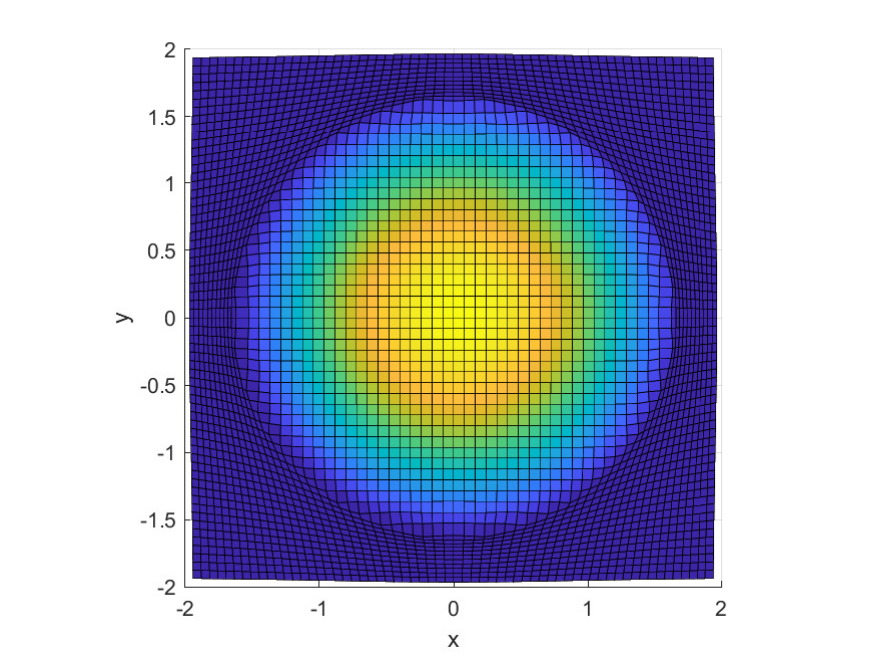}
	}
	\subfigure[$|B_{2,2}({\bm x},t)-\rho_h|$ at $t=2$]{
	\includegraphics[width=0.3\textwidth]{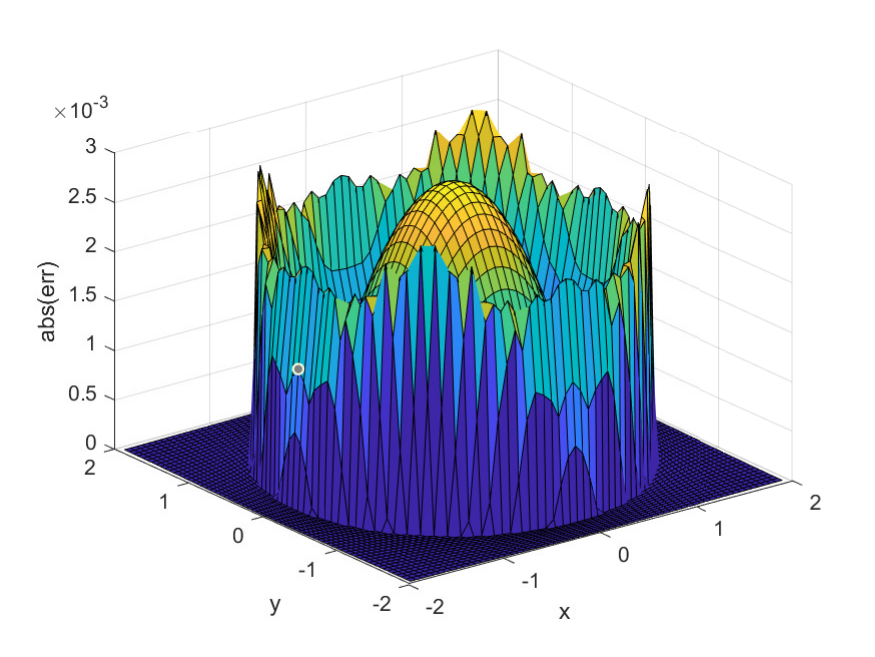}
}
	\subfigure[Numerical solution at $t=4$]{
		\includegraphics[width=0.3\textwidth]{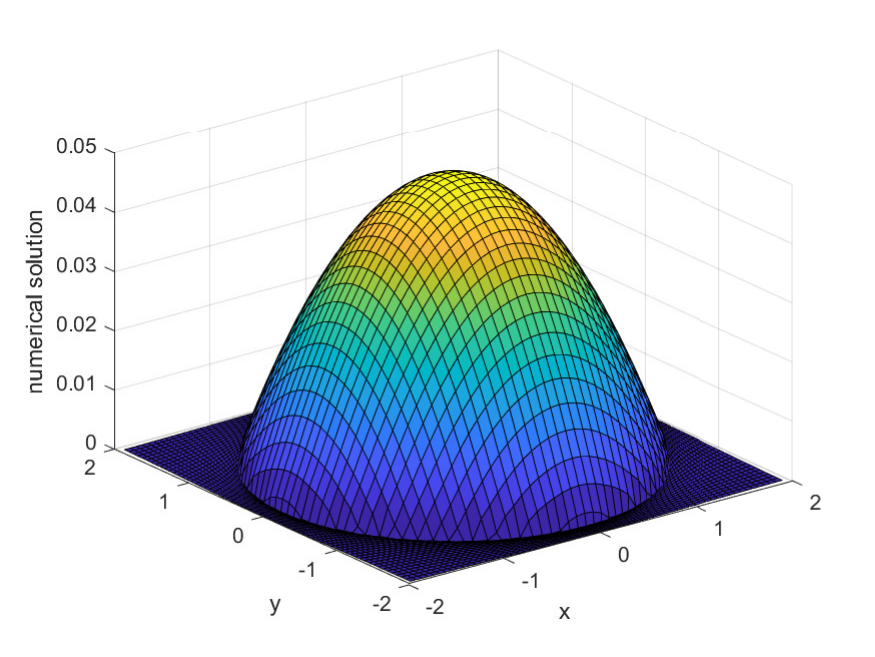}
	}
	\subfigure[Numerical solution at $t=4$]{
		\includegraphics[width=0.3\textwidth]{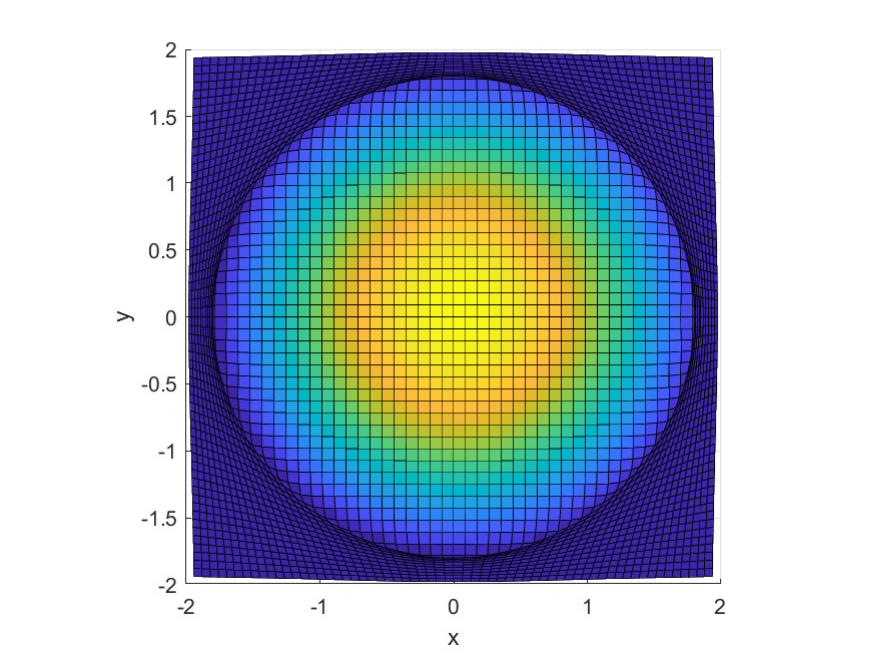}
	}
	\subfigure[$|B_{2,2}({\bm x},t)-\rho_h|$ at $t=4$]{
	\includegraphics[width=0.3\textwidth]{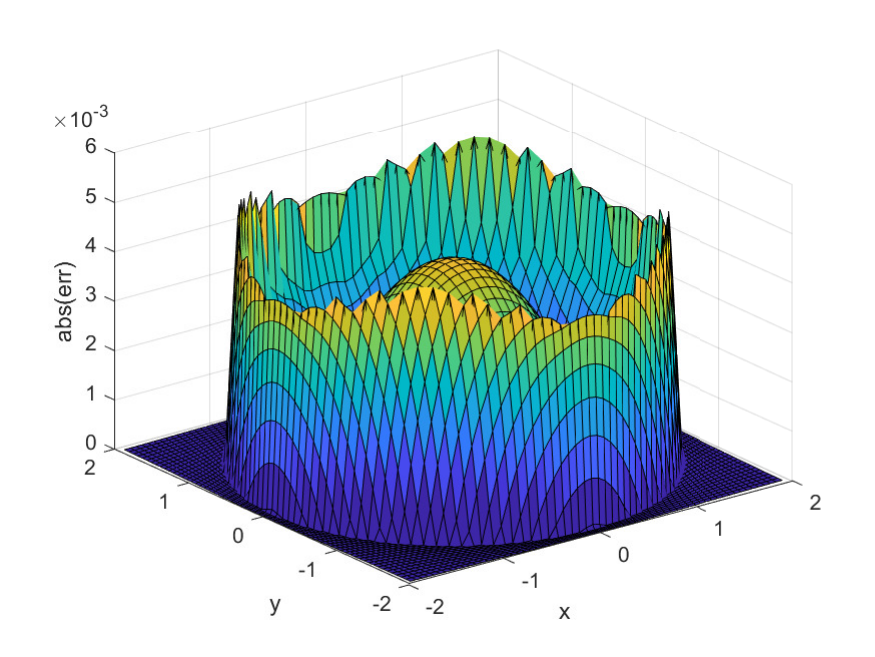}
}
	\caption{Numerical solution solved by \eqref{ex:1}-\eqref{ex:2} with  the regularization term $\epsilon\Delta_{\bm X}{\bm x}^{k+1}$, $\epsilon=10^{-3}\delta t$, the initial value is Barenblatt solution with $C_{B2}=0.1$, $m=2$, $M_x=M_y=64$, $\delta t=0.01$.}\label{fig:pme,2}
\end{figure}
\begin{figure}[!htbp]
	\centering
	\subfigure[Trajectory at $t=1$]{
		\includegraphics[width=0.3\textwidth]{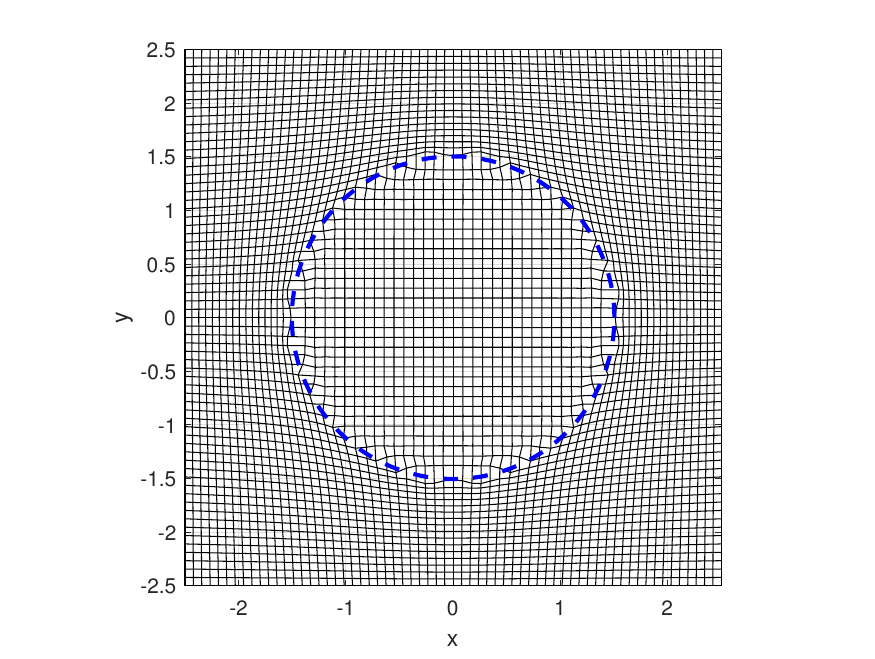}
	}
	\subfigure[Trajectory at $t=2$]{
		\includegraphics[width=0.3\textwidth]{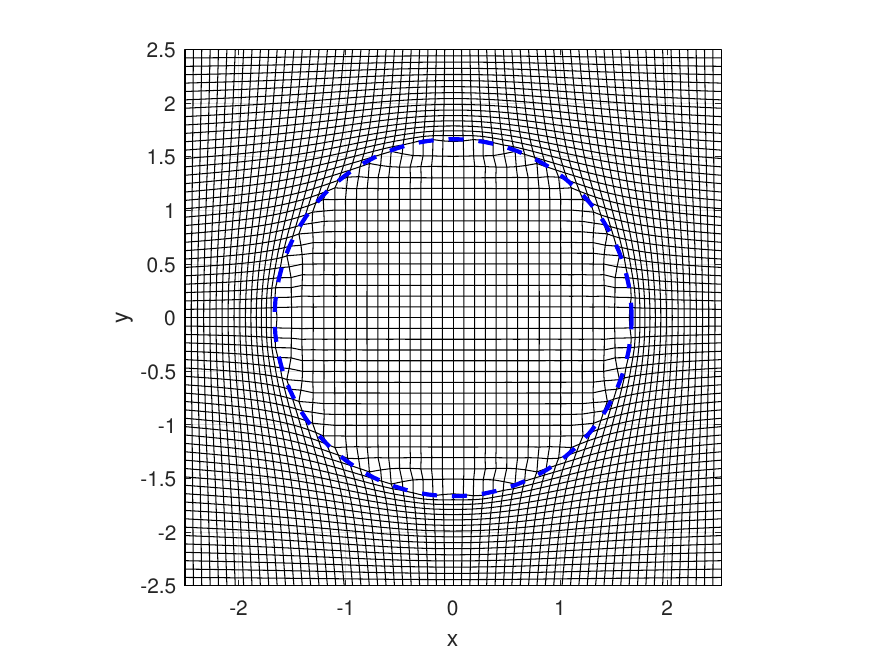}
	}
	\subfigure[Trajectory at $t=4$]{
		\includegraphics[width=0.3\textwidth]{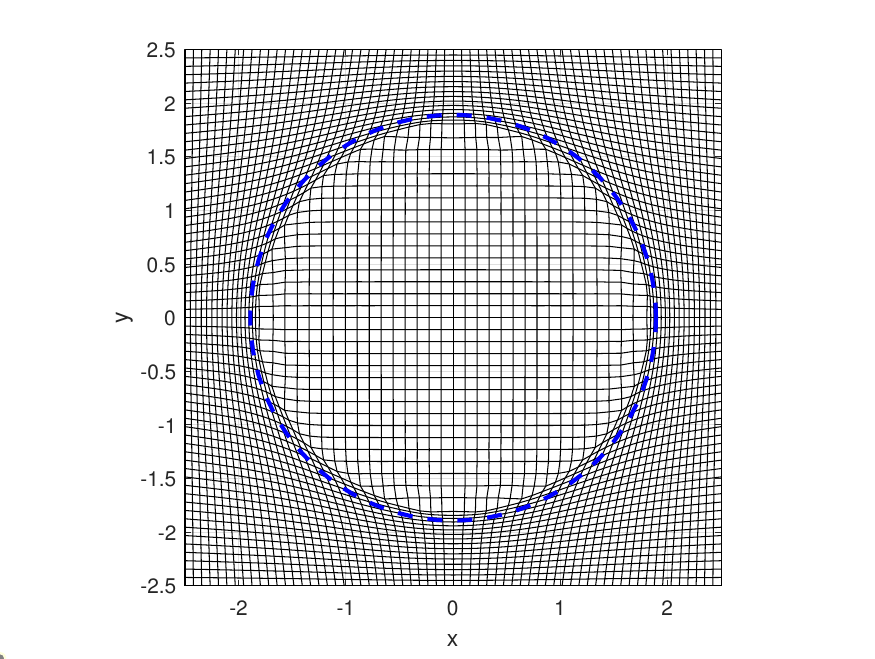}
	}
	\caption{Trajectories solved by \eqref{ex:1}-\eqref{ex:2} with  the regularization term $\epsilon\Delta_{\bm X}{\bm x}^{k+1}$, $\epsilon=10^{-3}\delta t$, the initial value is Barenblatt solution with $C_{B2}=0.1$, $m=2$, $M_x=M_y=64$, $\delta t=0.01$. The blue line represents the exact interface of the support set calculated by \eqref{pme:2d support set}.}\label{fig:pme,22}
\end{figure}

\begin{figure}[!htbp]
	\centering
	\subfigure[Numerical solution at $t=1$ ]{
		\includegraphics[width=0.3\textwidth]{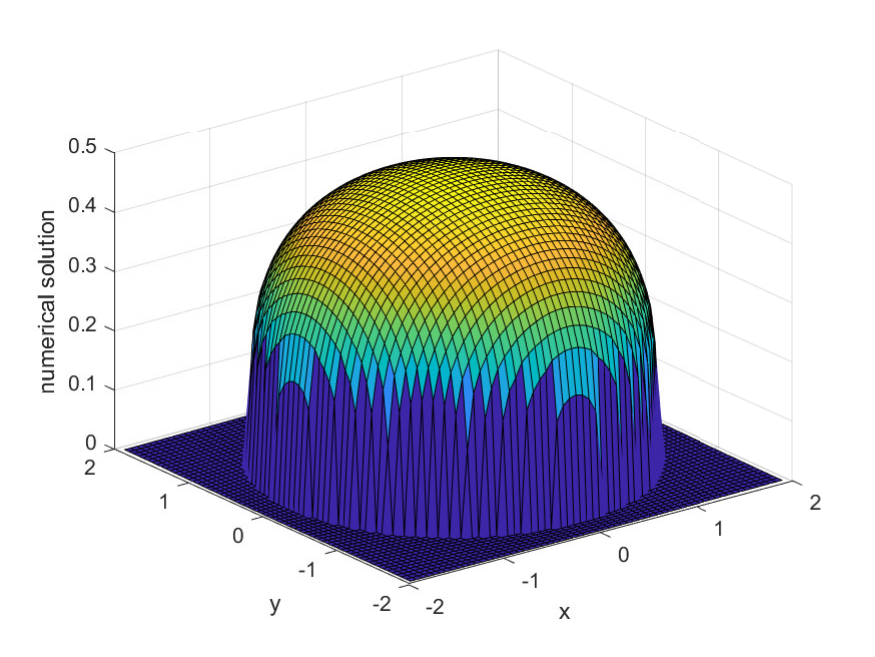}
	}
	\subfigure[Numerical solution at $t=1$]{
		\includegraphics[width=0.3\textwidth]{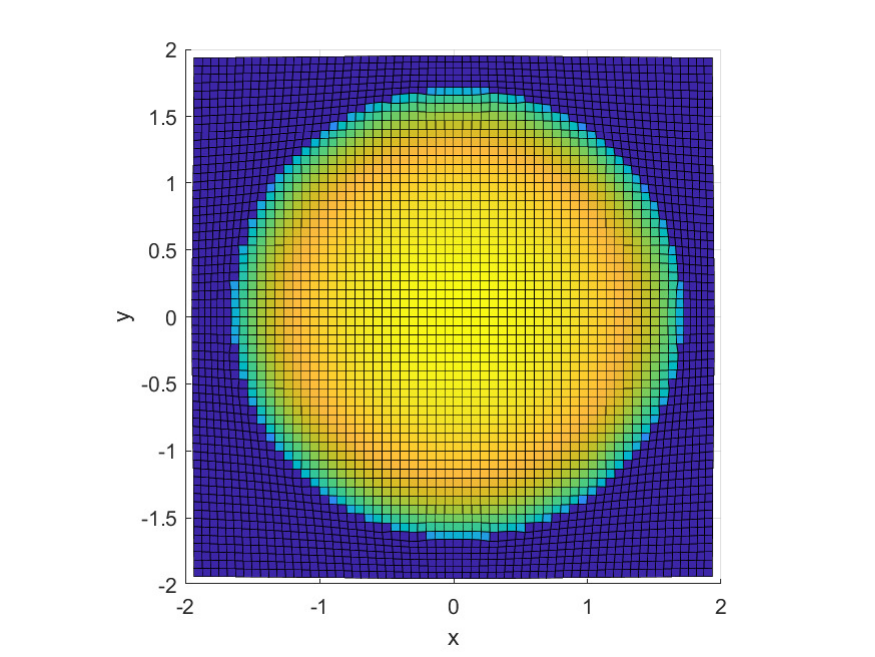}
	}
	\subfigure[$|B_{5,2}({\bm x},t)-\rho_h|$ at $t=1$]{
		\includegraphics[width=0.3\textwidth]{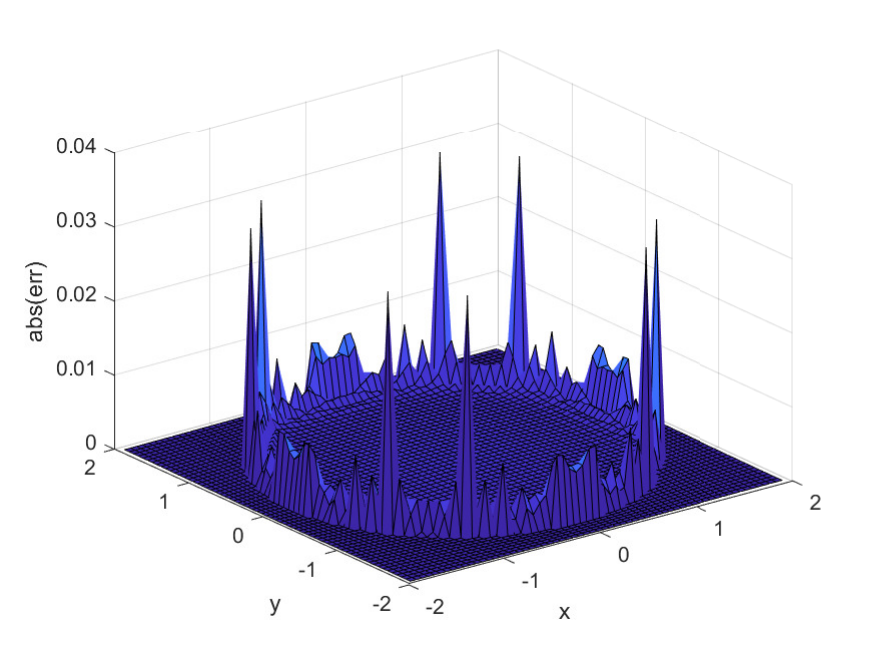}
	}
	\subfigure[Numerical solution at $t=2$ ]{
		\includegraphics[width=0.3\textwidth]{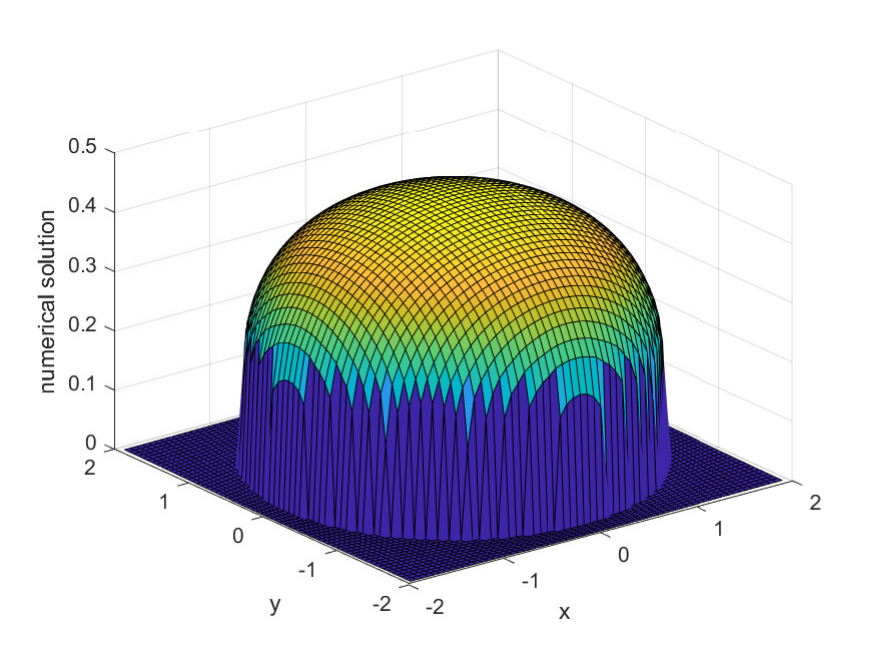}
	}
	\subfigure[Numerical solution at $t=2$]{
		\includegraphics[width=0.3\textwidth]{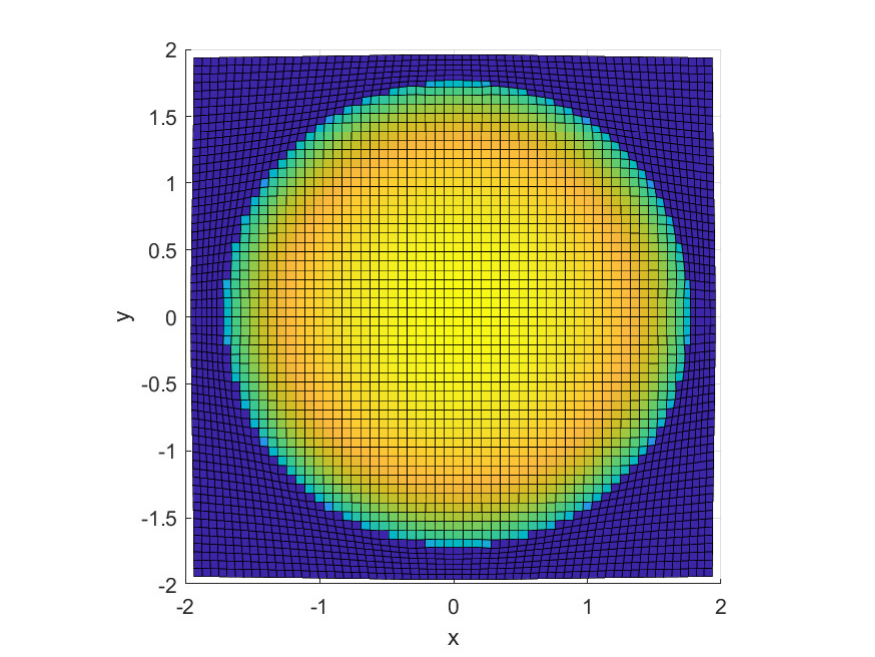}
	}
	\subfigure[$|B_{5,2}({\bm x},t)-\rho_h|$  at $t=2$]{
		\includegraphics[width=0.3\textwidth]{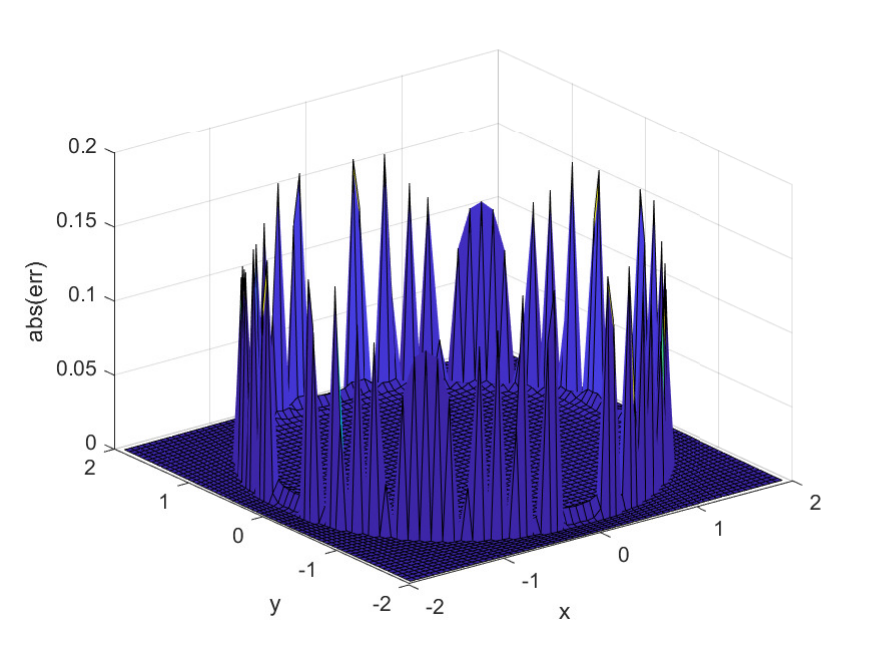}
	}
	\subfigure[Numerical solution at $t=4$]{
		\includegraphics[width=0.3\textwidth]{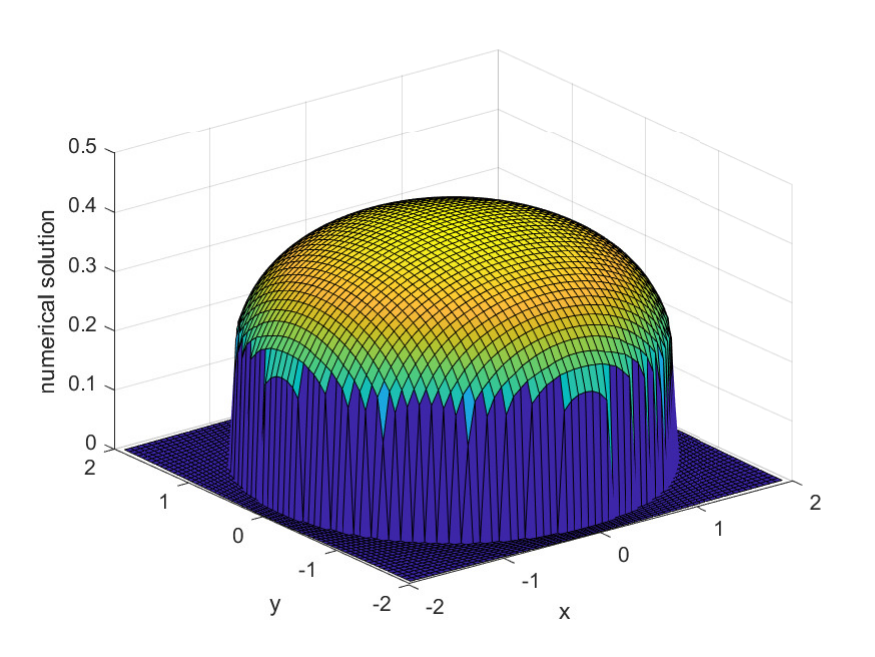}
	}
	\subfigure[Numerical solution at $t=4$]{
		\includegraphics[width=0.3\textwidth]{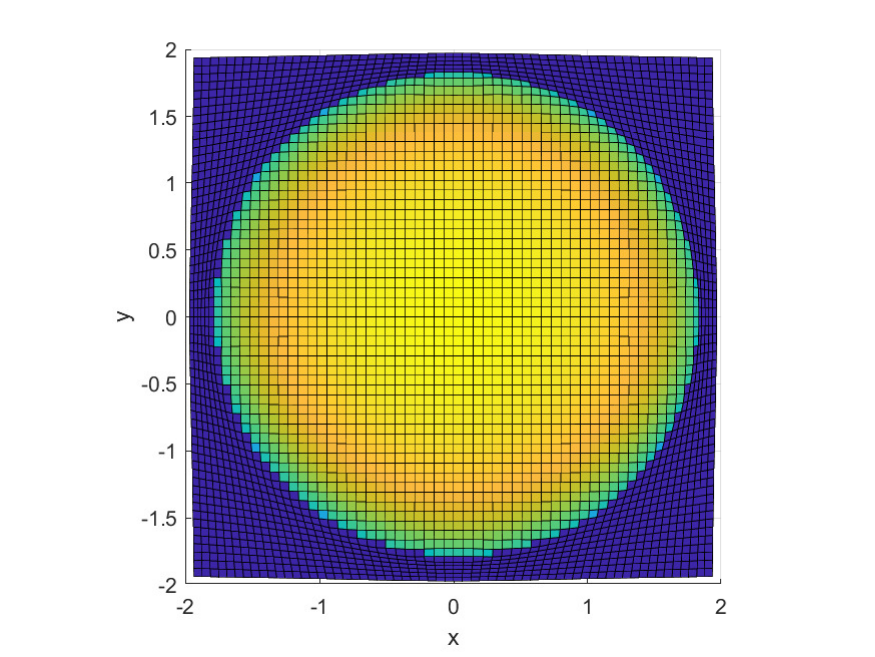}
	}
	\subfigure[$|B_{5,2}({\bm x},t)-\rho_h|$  at $t=4$]{
		\includegraphics[width=0.3\textwidth]{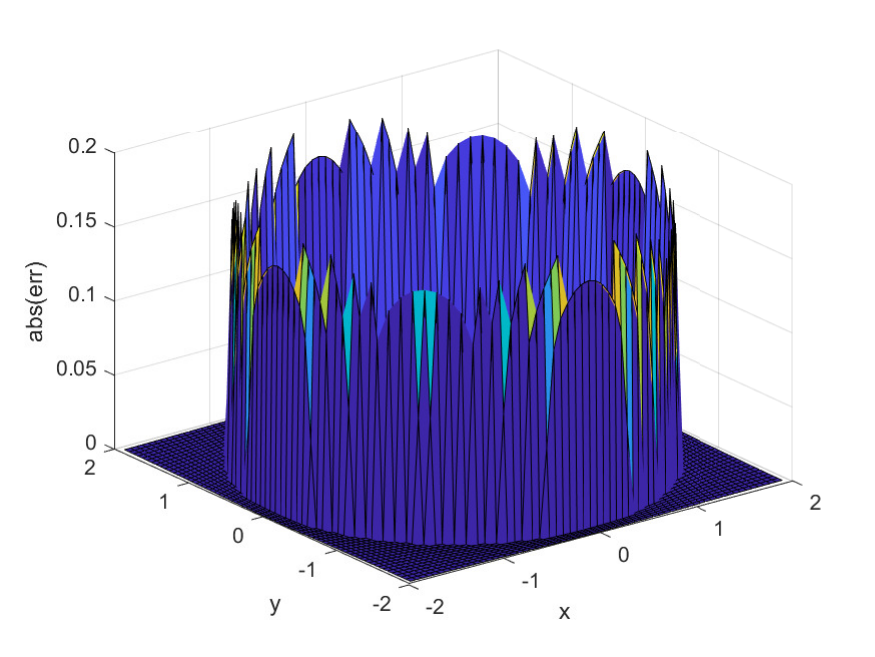}
	}
	\caption{Numerical solution solved by \eqref{ex:1}-\eqref{ex:2} with  the regularization term $\epsilon\Delta_{\bm X}{\bm x}^{k+1}$, $\epsilon=10^{-1}\delta t$, the initial value is Barenblatt solution with $C_{B2}=0.1$, $m=5$, $M_x=M_y=64$, $\delta t=0.01$.}\label{fig:pme,5}
\end{figure}
\begin{figure}[!htbp]
	\centering
	\subfigure[Trajectory at $t=1$]{
		\includegraphics[width=0.3\textwidth]{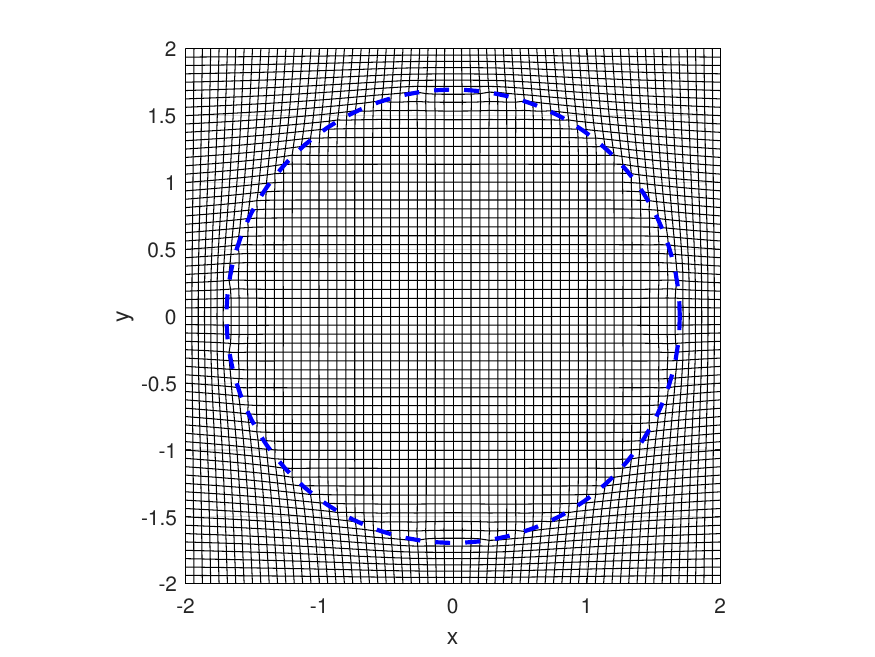}
	}
	\subfigure[Trajectory at $t=2$]{
		\includegraphics[width=0.3\textwidth]{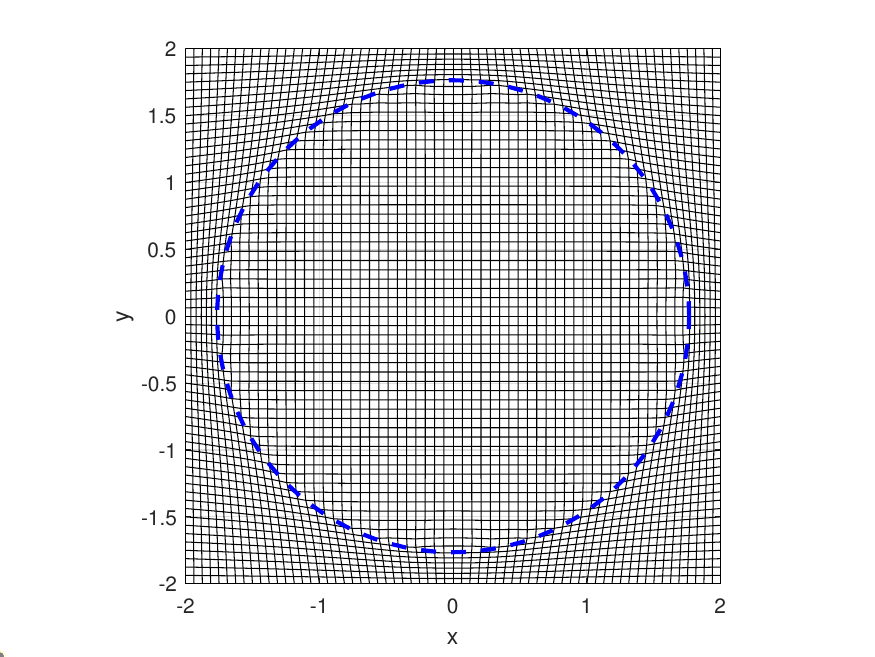}
	}
	\subfigure[Trajectory at $t=4$]{
		\includegraphics[width=0.3\textwidth]{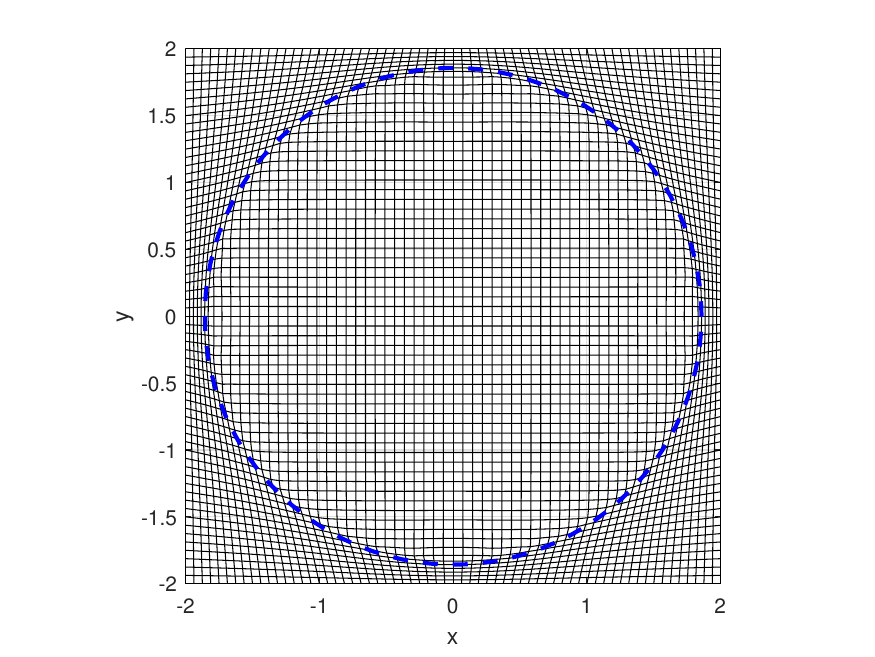}
	}
	\caption{Trajectories solved by \eqref{ex:1}-\eqref{ex:2} with  the regularization term $\epsilon\Delta_{\bm X}{\bm x}^{k+1}$, $\epsilon=10^{-1}\delta t$, the initial value is Barenblatt solution with $C_{B2}=0.1$, $m=5$, $M_x=M_y=64$, $\delta t=0.01$. The blue line represents the exact interface of the support set calculated by \eqref{pme:2d support set}.}\label{fig:pme,52}
\end{figure}


 The energy curves are shown in Figure~\ref{fig:pme,energy} which indicates that the energy dissipation law holds in both cases where $m=2$ and $m=5$. 
	\begin{figure}[!htb]
		\centering
		\subfigure[$m=2$]{
			\includegraphics[width=0.4\textwidth]{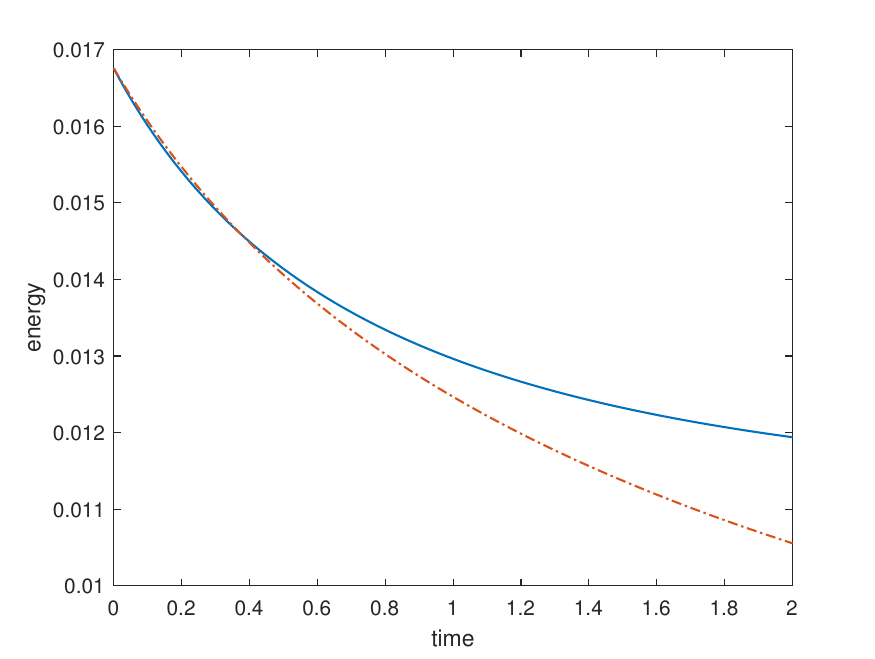}
		}
		\subfigure[$m=5$]{
			\includegraphics[width=0.4\textwidth]{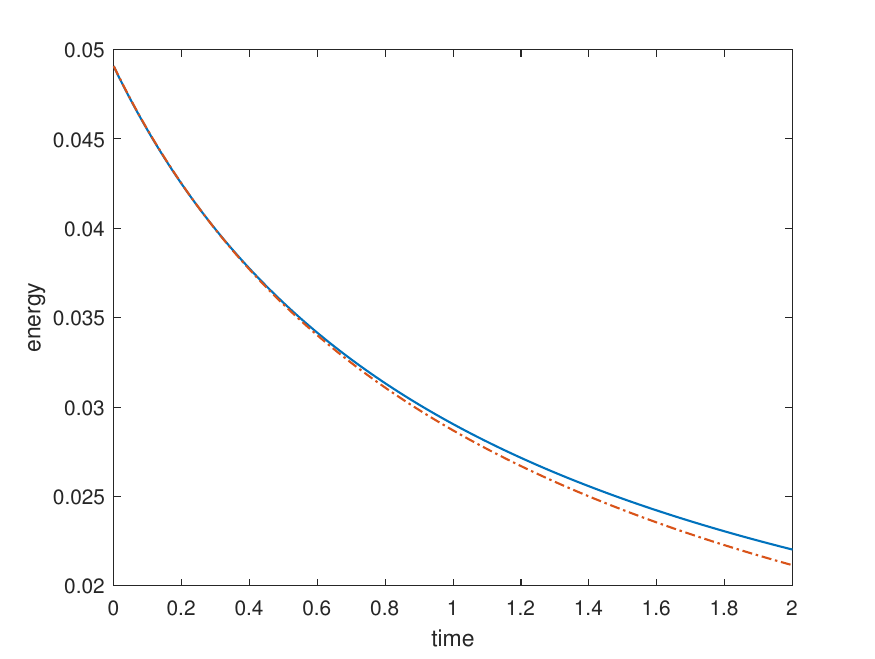}
		}
		\caption{Energy solved by \eqref{ex:1}-\eqref{ex:2} with different regularization terms. Set $M_x=M_y=64$, $\delta t=0.01$, the initial value is Barenblatt solution with $C_{B2}=0.1$, $m=2$, $m=5$. The blue solid lines:  $\epsilon\Delta_{\bm X}{\bm x}^{k+1}$, $\epsilon=h_x^2$. The red dashed lines:  $\epsilon_k\Delta_{\bm X}({\bm x}^{k+1}-{\bm x}^{k})$, $\epsilon_k=0.1$. }\label{fig:pme,energy}
	\end{figure}
	Figure~\ref{fig:pme,det} illustrates the minimum and maximum determinant values for both $m=2$ and $m=5$ with various regularization terms, respectively. It is evident from the plots that the determinant value maintains its positivity. Further, a positive distance can be observed between the minimum value and zero.
	\begin{figure}[!htb]
		\centering
		\subfigure[$m=2$]{
			\includegraphics[width=0.4\textwidth]{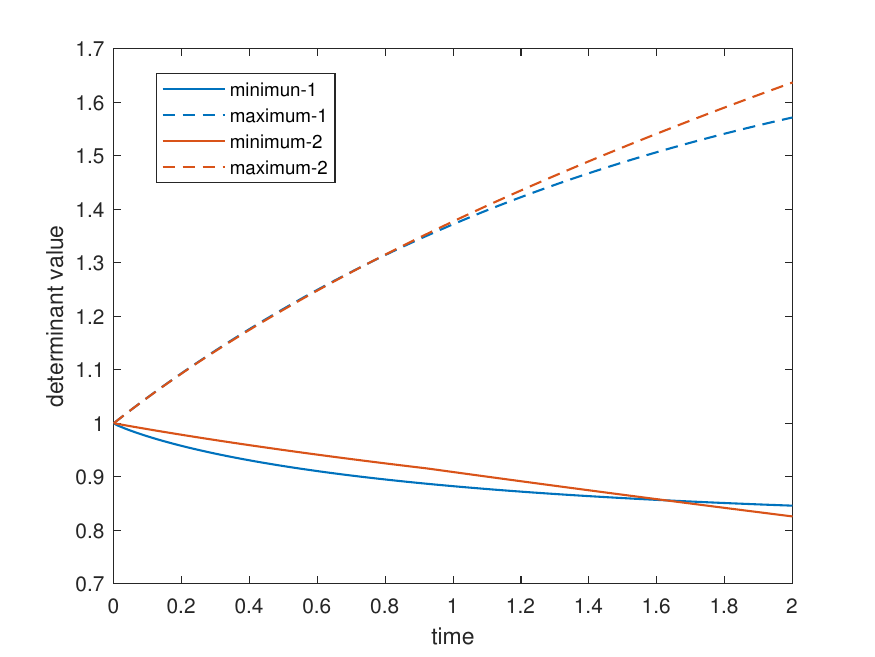}
		}
		\subfigure[$m=5$]{
			\includegraphics[width=0.4\textwidth]{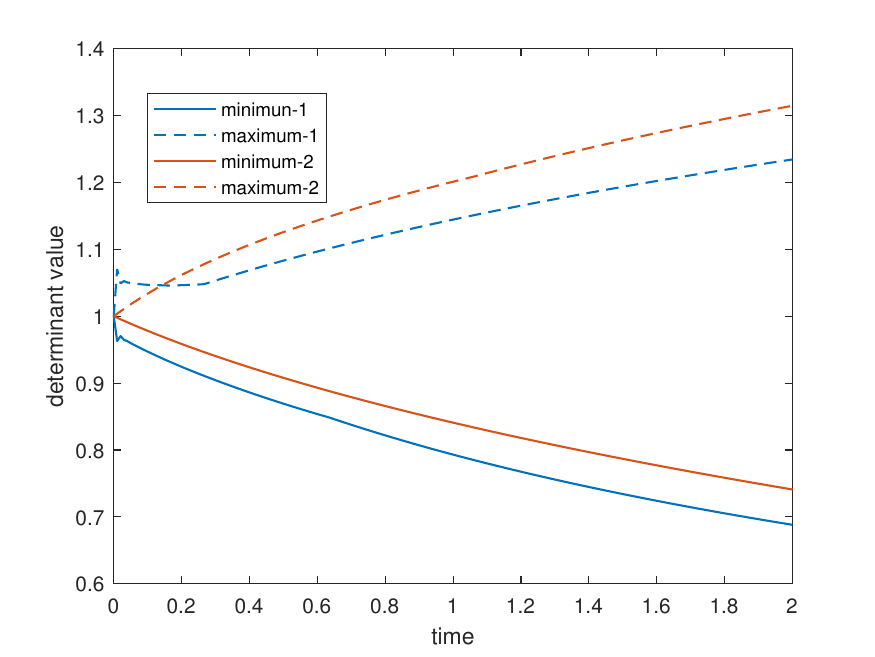}
		}
		\caption{Determinant values solved by \eqref{ex:1}-\eqref{ex:2} with different regularization terms. Set $M_x=M_y=64$, $\delta t=0.01$, the initial value is Barenblatt solution with $C_{B2}=0.1$, $m=2$, $m=5$. The blue lines:  $\epsilon\Delta_{\bm X}{\bm x}^{k+1}$, $\epsilon=h_x^2$. The red lines:  $\epsilon_k\Delta_{\bm X}({\bm x}^{k+1}-{\bm x}^{k})$, $\epsilon_k=0.1$.}\label{fig:pme,det}
\end{figure}
The convergence order of the numerical solution with the exact solution $B_{m,2}({\bm x},0.1)$, $m=2$ and $m=2.5$ are shown in Table~\ref{convergence pme_2d}.
\begin{table}[!htb]
	\centering
	\caption{Convergence order with $B_{m,2}({\bm x},0.1)$, $C_{B2}=0.1$, $m=2$ and $m=2.5$ in $[-2,2]\times[-2,2]$. Numerical solutions are solved by \eqref{ex:1}-\eqref{ex:2} with regularization term $\epsilon\Delta_{\bm X}{\bm x}^{k+1}$, $\epsilon=h_x^2$.}
	\begin{tabular}{cc||cc||cc}
		\hline
		&	&	$m=2$& & $m=2.5$ & \\
		\hline
		$M_x\times M_y$ &$N_t$&$L_h^{2}$ error ($\rho$) & order & $L_h^{2}$ error ($\rho$) & order \\ \hline
		16x16 &16  & 0.0029&      &0.0054&      \\
		32x32 &32 & 0.0014&   1.0506   &0.0026&      1.0544  \\
		64x64 &64&7.5117e-04& 0.8982   & 0.0014&  	0.8931   \\
		128x128 &128& 3.6211e-04&1.0527&  6.5371e-04&1.0987\\
		256x256 &256& 1.7947e-04&1.0127&  3.0605e-04&1.0949\\
		\hline
	\end{tabular}\label{convergence pme_2d}
\end{table}

  {\bf Non-radial problem.} Consider the following initial value:
	\begin{equation}\label{ini:nonradial}
	\rho_0(x,y)=\begin{cases}
	25(0.25^2-(\sqrt{x^2+y^2}-0.75)^2)^{\frac{3}{2}},\ &\sqrt{x^2+y^2}\in[0.5,1]\ \text{and}\ (x<0\ \text{or}\ y<0),\\
	25(0.25^2-x^2-(y-0.75)^2)^{\frac{3}{2}},\ &x^2+(y-0.75)^2\le 0.25^2\ \text{and}\ x\ge 0,\\
	25(0.25^2-(x-0.75)^2-y^2)^{\frac{3}{2}},\ &(x-0.75)^2+y^2\le 0.25^2\ \text{and}\ y\ge 0,\\
	0,\ &\text{otherwise},
	\end{cases}
	\end{equation}
	which has a partial donut-shaped support \cite{liu2020lagrangian}. 
	Let's set $m=3$, $M_x=M_y=64$, $\delta t=0.001$, and  use scheme \eqref{ex:1}-\eqref{ex:2} with the regularization term $\epsilon\Delta_{\bm X}{\bm x}^{k+1}$, $\epsilon=0.1\delta t$  to implement the numerical experiments, numerical results are shown in Figure~\ref{fig:nonradial,0}, where the evolution of the trajectories can be found.  We can observe that the scheme \eqref{ex:1}-\eqref{ex:2} handles this situation well. However, it should be noted that this approach cannot handle topological changes automatically, which serves as a limitation of our method.

\begin{figure}[!htb]
	\centering
	\subfigure[Numerical solution at $t=0.1$]{
		\includegraphics[width=0.3\textwidth]{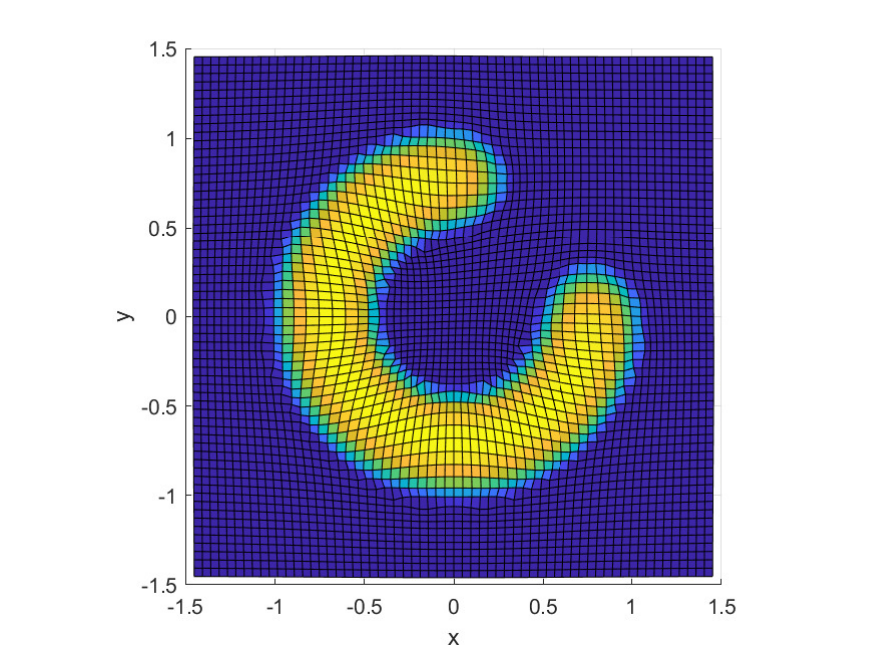}
	}
	\subfigure[Numerical solution at $t=0.2$]{
		\includegraphics[width=0.3\textwidth]{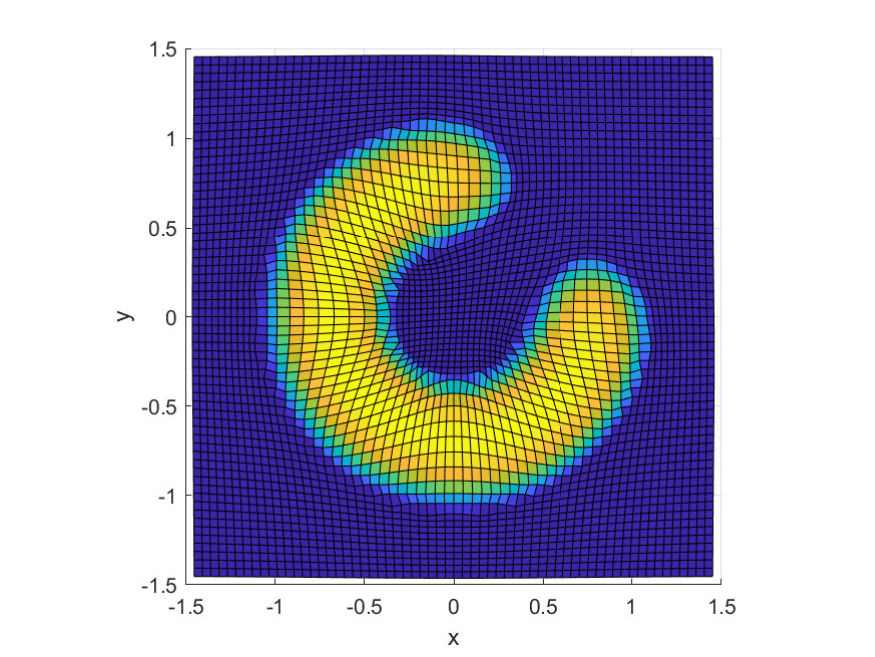}
	}
	\subfigure[Numerical solution at $t=0.5$]{
		\includegraphics[width=0.3\textwidth]{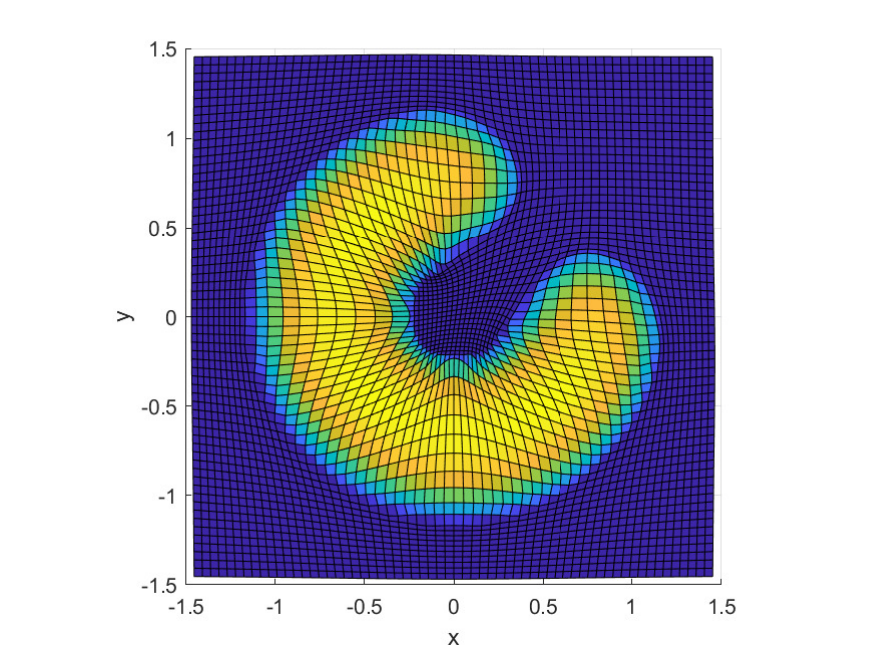}
	}
	\subfigure[Trajectory at $t=0.1$]{
		\includegraphics[width=0.3\textwidth]{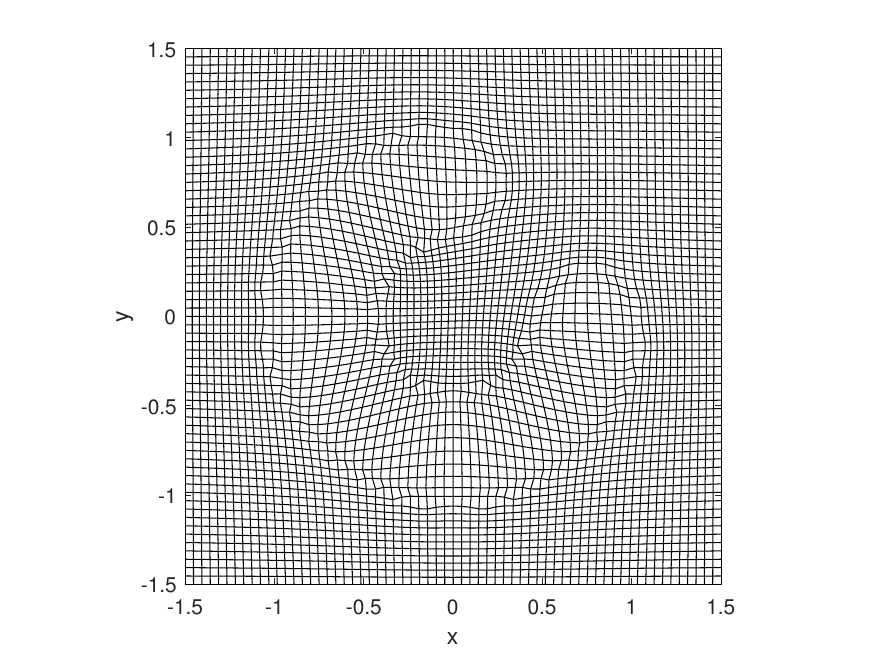}
	}
	\subfigure[Trajectory at $t=0.2$]{
		\includegraphics[width=0.3\textwidth]{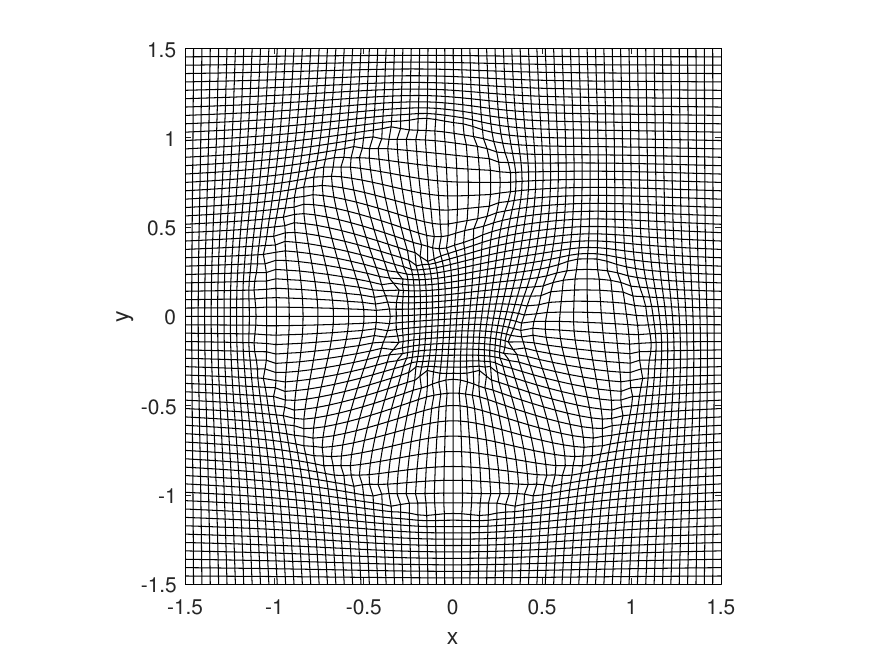}
	}
	\subfigure[Trajectory at $t=0.5$]{
		\includegraphics[width=0.3\textwidth]{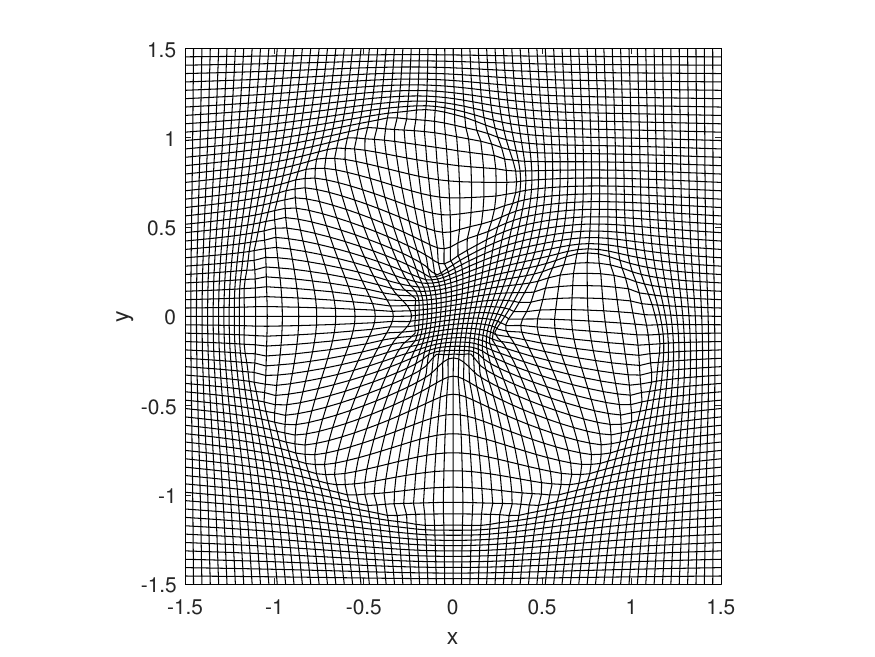}
	}
	\caption{Numerical solution solved by \eqref{ex:1}-\eqref{ex:2} with the regularization term $\epsilon\Delta_{\bm X}{\bm x}^{k+1}$, $\epsilon=0.1\delta t$, $m=3$, the initial value \eqref{ini:nonradial} in $[-1.5,1.5]\times[-1.5,1.5]$, $M_x=M_y=64$, $\delta t=0.001$.}\label{fig:nonradial,0}
\end{figure}

Now we take the initial value as
	\begin{equation}\label{ini:nonradial1}
	\rho_0(x,y)=e^{-20((x-0.5)^2+(y-0.5)^2)},\ x\in[-2,2]\times[-2,2].
	\end{equation}
	 Let's set $m=2$, $M_x=M_y=64$, $\delta t=0.01$, and use scheme \eqref{ex:1}-\eqref{ex:2} with the regularization term $\epsilon\Delta_{\bm X}{\bm x}^{k+1}$, $\epsilon=0.1\delta t$ to simulate the numerical experiments, results are shown in Figure~\ref{fig:nonradial,1}.
	 It can be found that the proposed scheme \eqref{ex:1}-\eqref{ex:2} can handle the nonradial case well.

\begin{figure}[!htb]
	\centering
	\subfigure[Numerical solution at $t=0.5$]{
		\includegraphics[width=0.3\textwidth]{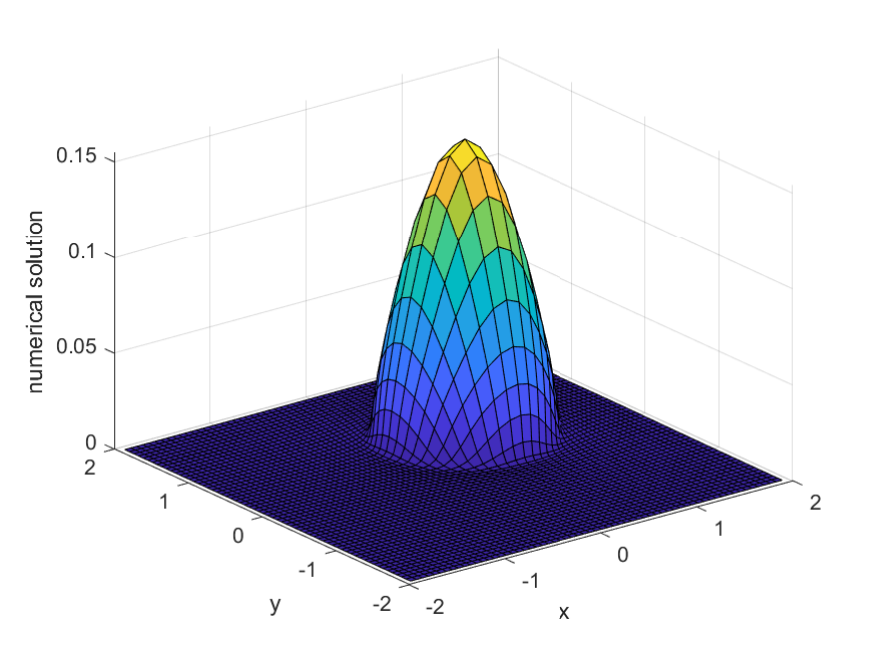}
	}
	\subfigure[Numerical solution at $t=1$]{
		\includegraphics[width=0.3\textwidth]{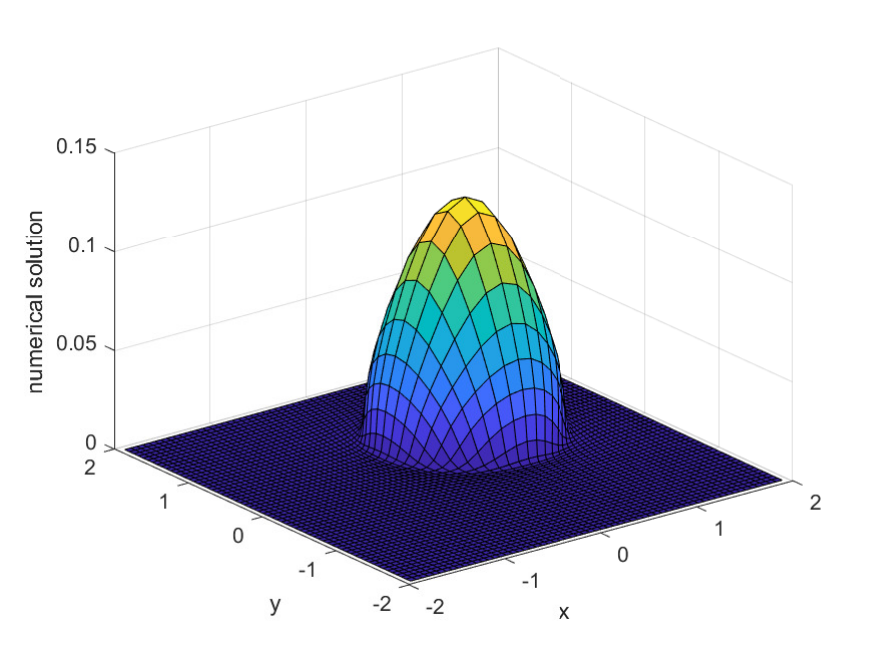}
	}
	\subfigure[Numerical solution at $t=5$]{
		\includegraphics[width=0.3\textwidth]{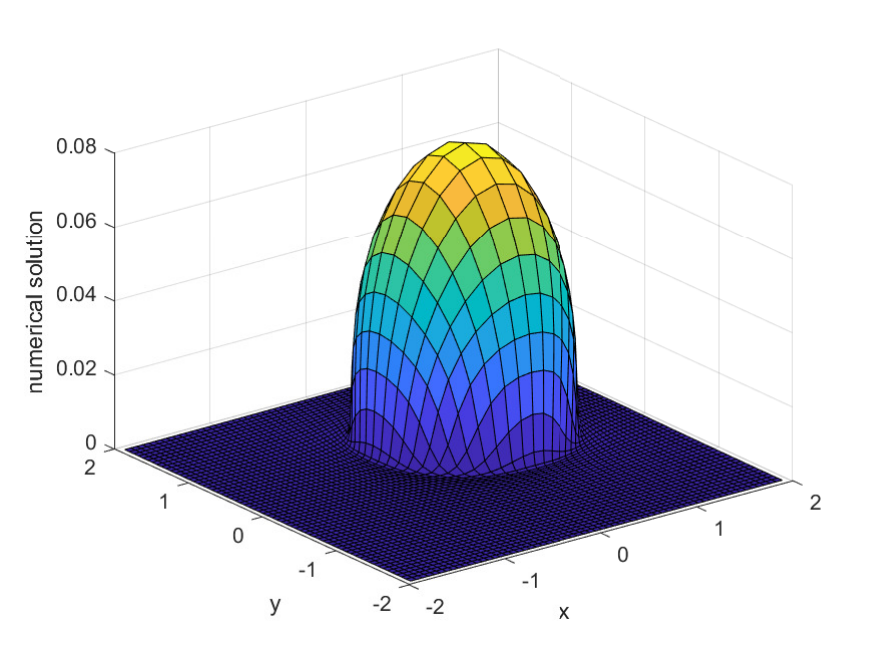}
	}
	\subfigure[Trajectory at $t=0.5$]{
		\includegraphics[width=0.3\textwidth]{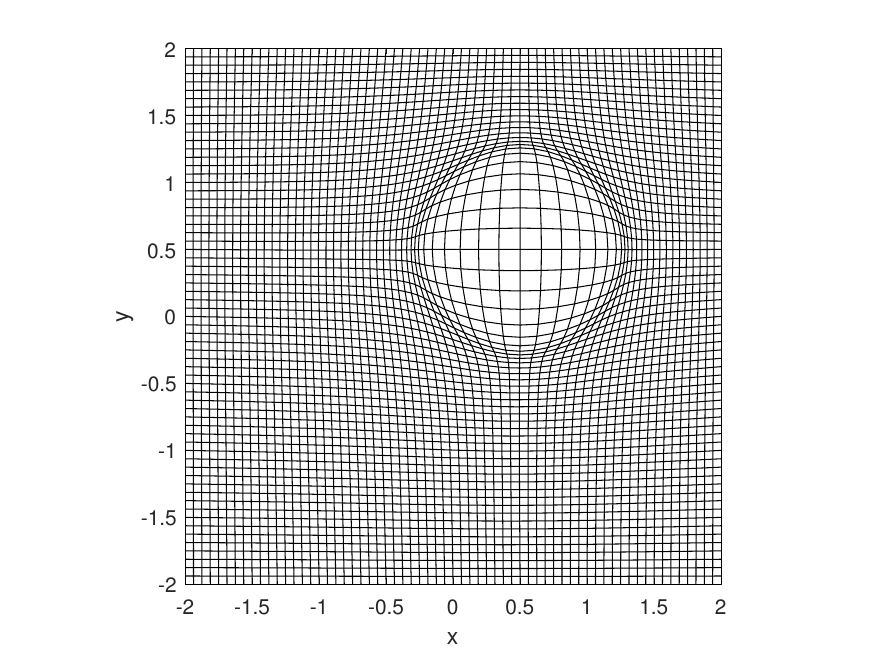}
	}
	\subfigure[Trajectory at $t=1$]{
		\includegraphics[width=0.3\textwidth]{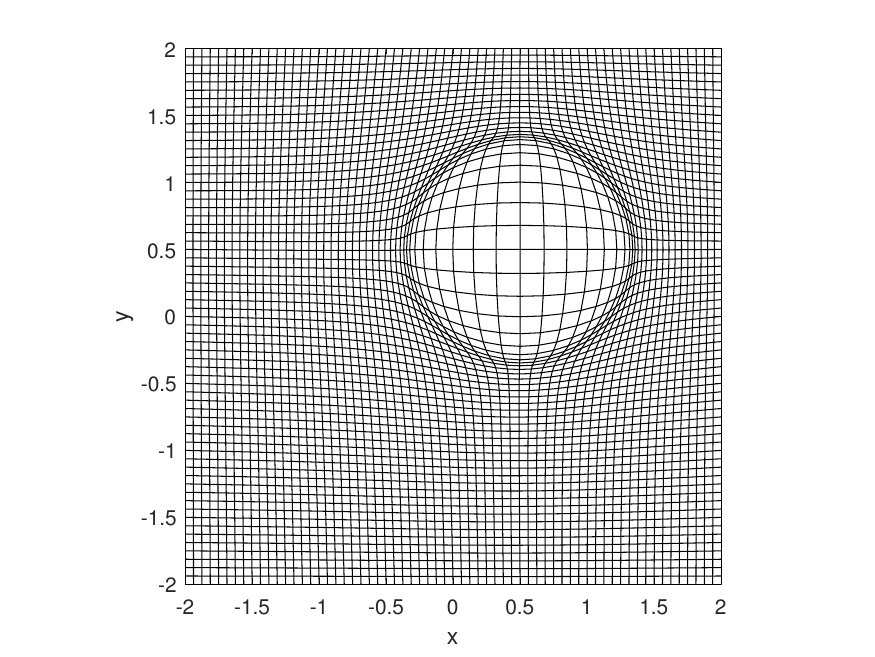}
	}
	\subfigure[Trajectory at $t=5$]{
		\includegraphics[width=0.3\textwidth]{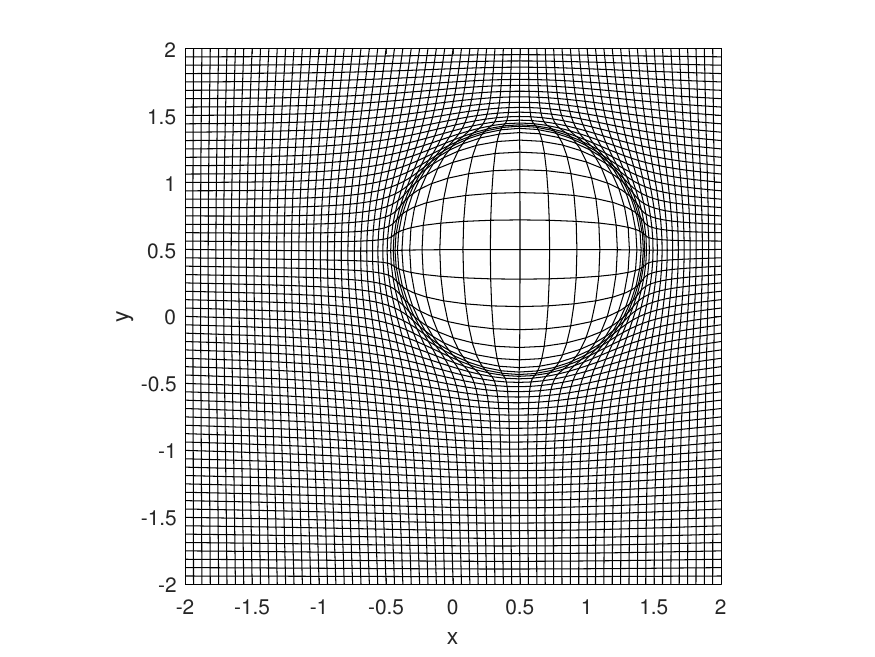}
	}
	\caption{Numerical solution  solved by \eqref{ex:1}-\eqref{ex:2} with the regularization term $\epsilon\Delta_{\bm X}{\bm x}^{k+1}$, $\epsilon=0.1\delta t$, $m=2$, the initial value \eqref{ini:nonradial1} in $[-2,2]\times[-2,2]$, $M_x=M_y=64$, $\delta t=0.01$.		
	}\label{fig:nonradial,1}
\end{figure}

\subsubsection{Aggregation equation}
Consider the following aggregation equation:
\begin{align}\label{aggre}
	\partial_t\rho=\nabla\cdot(\rho\nabla W*\rho),\qquad W:\mathbb{R}^2\rightarrow \mathbb{R}.
\end{align}
We simulate the evolution of solutions to \eqref{aggre} with
\begin{align*}
	W({\bm x})=\frac{|{\bm x}|^2}{2}-\ln |{\bm x}|,
\end{align*}
and the initial value 
\begin{align}\label{initial:ks2d}
\rho_0(x,y)=C_{2d}e^{-x^2-y^2},
\end{align}
 in $[-2,2]\times[-2,2]$, $C_{2d}=1$, $M_x=M_y=64$, $\delta t=0.01$.
  Using scheme \eqref{ex:1}-\eqref{ex:2} with the regularization term $\epsilon\Delta_{\bm X}{\bm x}^{k+1}$, $\epsilon=0.1\delta t$ to simulate the numerical experiments, the numerical solution and trajectory plots at $t=1$ are shown in Figure~\ref{fig:aggre,0}. The blue line in the trajectory plot represents the unit circle. As can be observed the solution converges to a characteristic function on the disk of radius $1$, centered at $(0,0)$, recovering analytic results on solutions of the aggregation equation with Newtonian repulsion \cite{carrillo2022primal,fetecau2011swarm,li2020fisher}. 
 Moreover, the determinant value plot is also presented in Figure~\ref{fig:aggre,0}, we find that despite the minimum value diminishing over time, the distance between the minimum and zero remains positive as the density tends to a steady state.

\begin{figure}[!htb]
	\centering
	\subfigure[Numerical solution at $t=1$]{
		\includegraphics[width=0.3\textwidth]{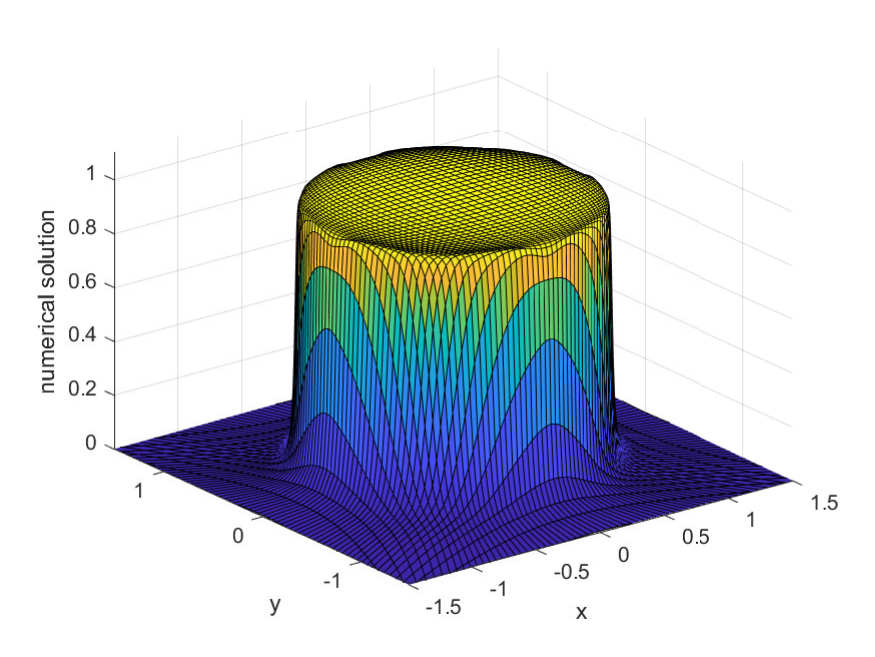}
	}
	\subfigure[Trajectory at $t=1$]{
		\includegraphics[width=0.3\textwidth]{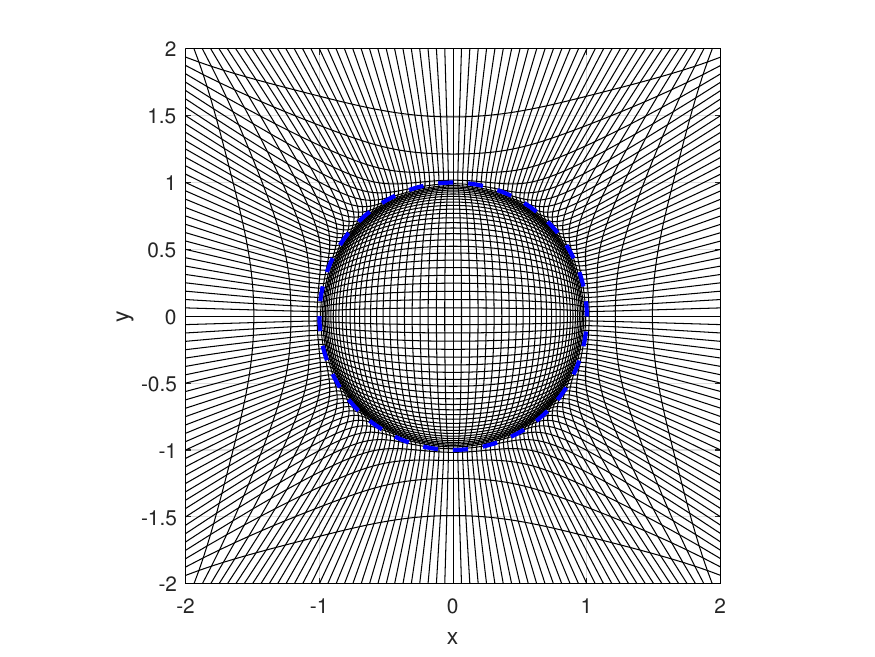}
	}
	\subfigure[Determinant value]{
	\includegraphics[width=0.3\textwidth]{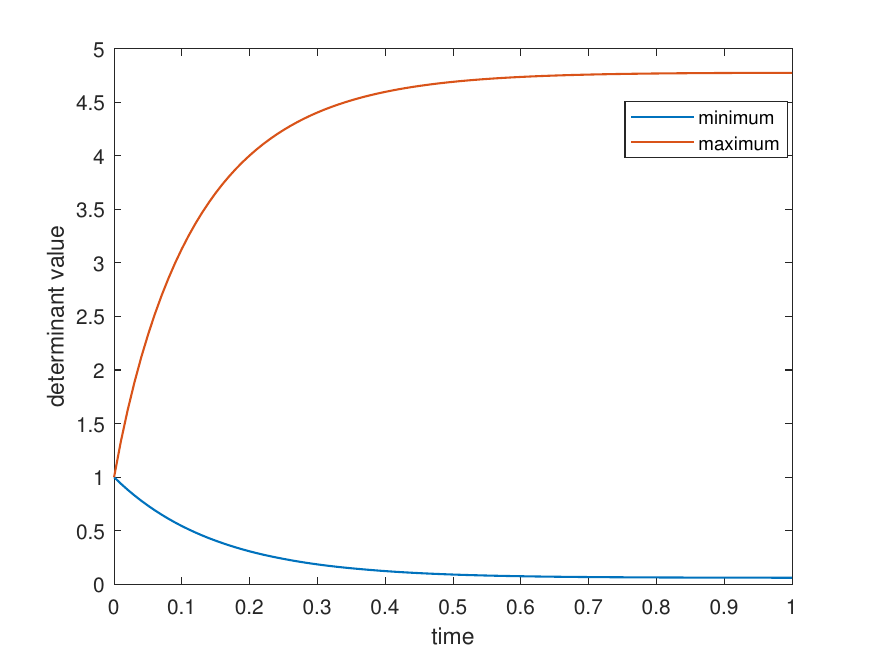}
}
	\caption{Aggregation equation solved by \eqref{ex:1}-\eqref{ex:2} with the regularization term $\epsilon\Delta_{\bm X}{\bm x}^{k+1}$, $\epsilon=0.1\delta t$, the initial value \eqref{initial:ks2d} in $[-2,2]\times[-2,2]$, $M_x=M_y=64$, $\delta t=0.01$.}\label{fig:aggre,0}
\end{figure}

\subsubsection{Aggregation diffusion equation}
Now we simulate several examples of aggregation-diffusion equation:
\begin{align}\label{eq:agg}
\partial_t\rho=\nabla\cdot(\rho\nabla W*\rho)+\nu\Delta\rho^m,\qquad W:\mathbb{R}^2\rightarrow \mathbb{R},\quad m\ge1.
\end{align}

For the aggregation diffusion equation \eqref{eq:agg}, we take $W({\bm x})=-\frac{1}{\pi}e^{-|{\bm x}|^2}$, $m=3$ and $\nu=0.1$. The initial value is taken as 
\begin{align}\label{initial:aggre}
	\rho_0(x,y)=\frac{1}{2}\chi_{|x|\le 2.5,|y|\le 2.5}(x,y).
\end{align} 
Using scheme \eqref{ex:1}-\eqref{ex:2} with the regularization term $\epsilon\Delta_{\bm X}{\bm x}^{k+1}$, $\epsilon=0.1\delta t$  to carry out numerical experiments, 
the evolution of density can be found in Figure~\ref{fig:aggre,2}, the solution tends to form four bumps  at the four angles at the beginning, and finally approaches a single bump equilibrium \cite{carrillo2022primal,carrillo2015finite}. The trajectory and determinant value plots are depicted in Figure~\ref{fig:aggre_diffu,det}, which implies that there is no distortion or swap during the evolution of the solution, and that the minimum value of the determinant is lower bounded away from zero.

\begin{figure}[!htb]
	\centering
	\subfigure[Initial value $\rho_0(x,y)$ ]{
		\includegraphics[width=0.3\textwidth]{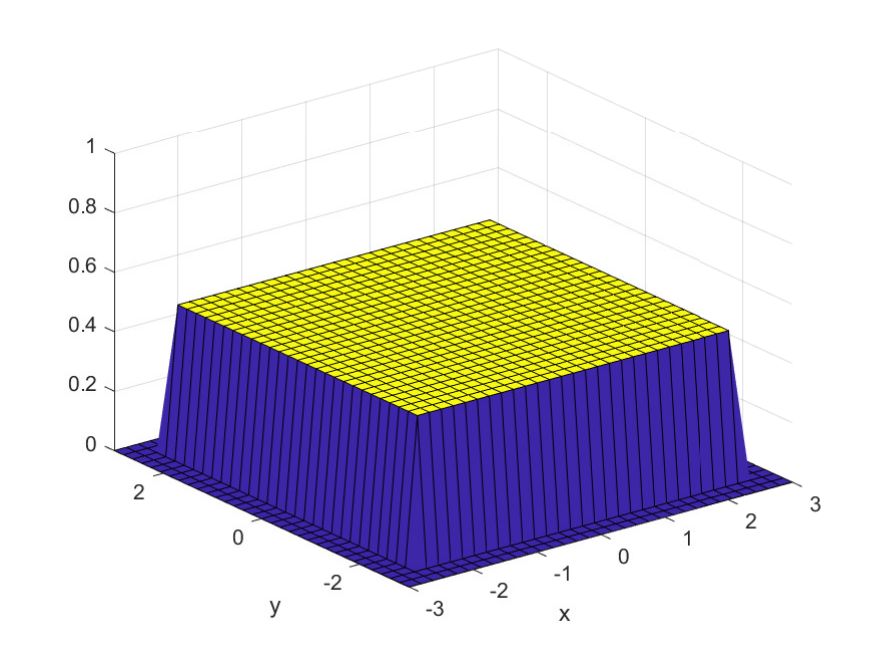}
	}
	\subfigure[Numerical solution at $t=4$]{
		\includegraphics[width=0.3\textwidth]{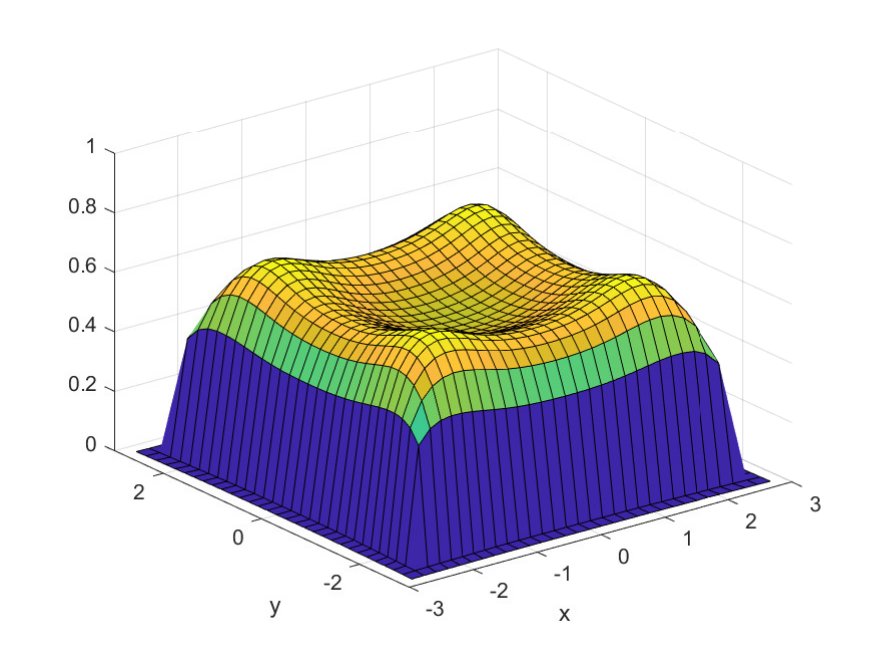}
	}
	\subfigure[Numerical solution at $t=10$]{
		\includegraphics[width=0.3\textwidth]{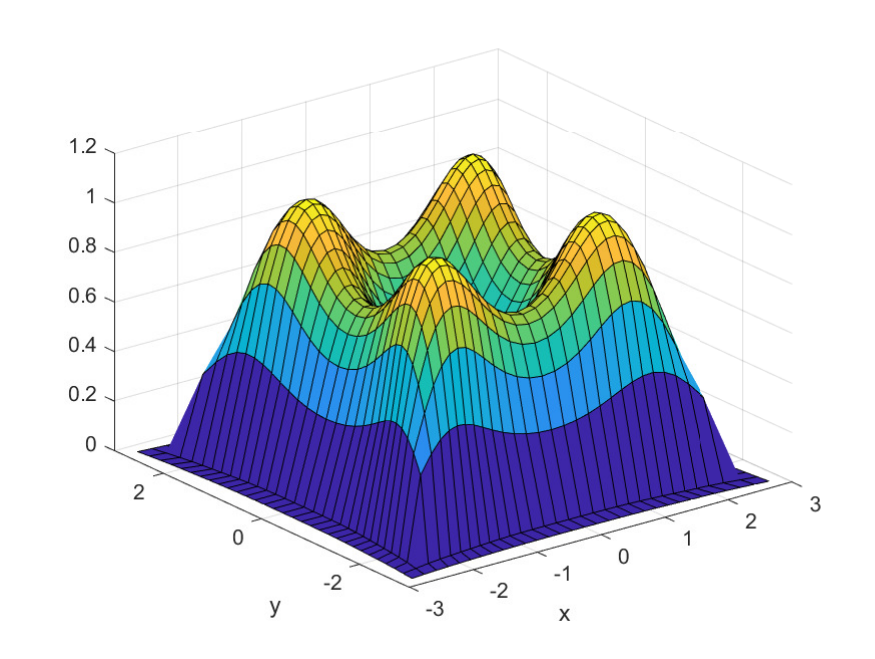}
	}
	\subfigure[Numerical solution at $t=15$]{
		\includegraphics[width=0.3\textwidth]{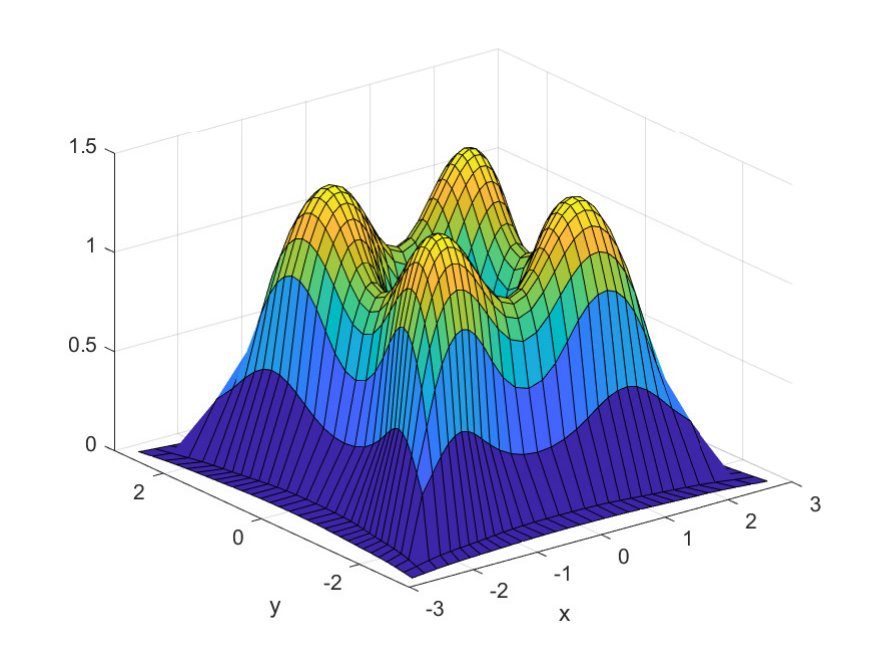}
	}
	\subfigure[Numerical solution at $t=20$]{
		\includegraphics[width=0.3\textwidth]{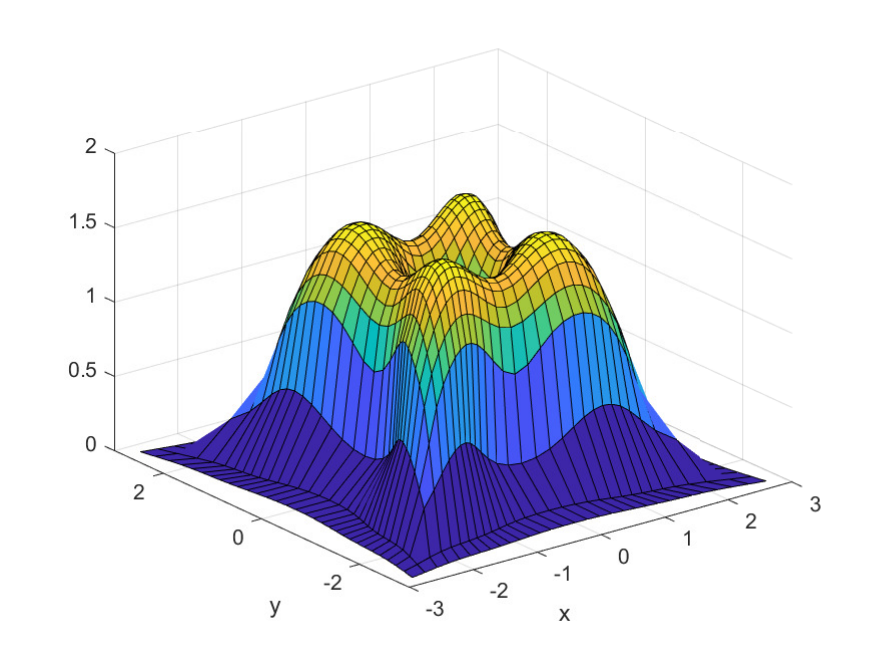}
	}
	\subfigure[Numerical solution at $t=25$]{
		\includegraphics[width=0.3\textwidth]{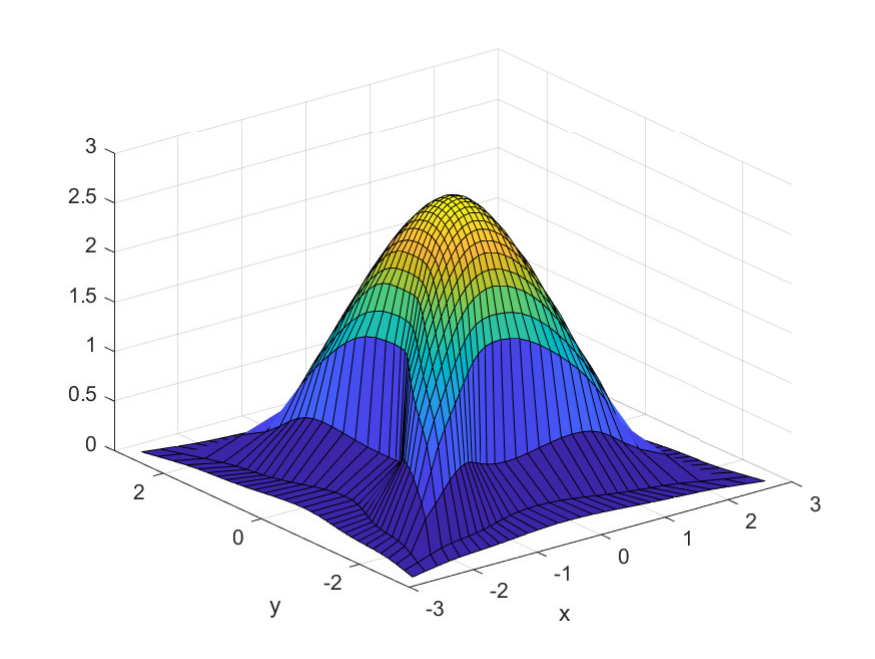}
	}
	\caption{Evolution of the numerical solution for the aggregation diffusion equation solved by \eqref{ex:1}-\eqref{ex:2} with the regularization term $\epsilon\Delta_{\bm X}{\bm x}^{k+1}$, $\epsilon=0.1\delta t$, initial value \eqref{initial:aggre} in $[-3,3]\times[-3,3]$, $m=3$, $\nu=0.1$, $M_x=M_y=32$, $\delta t=0.01$.}\label{fig:aggre,2}
\end{figure}

 \begin{figure}[!htb]
	\centering
	\subfigure[Trajectory at $t=40$]{
		\includegraphics[width=0.45\textwidth]{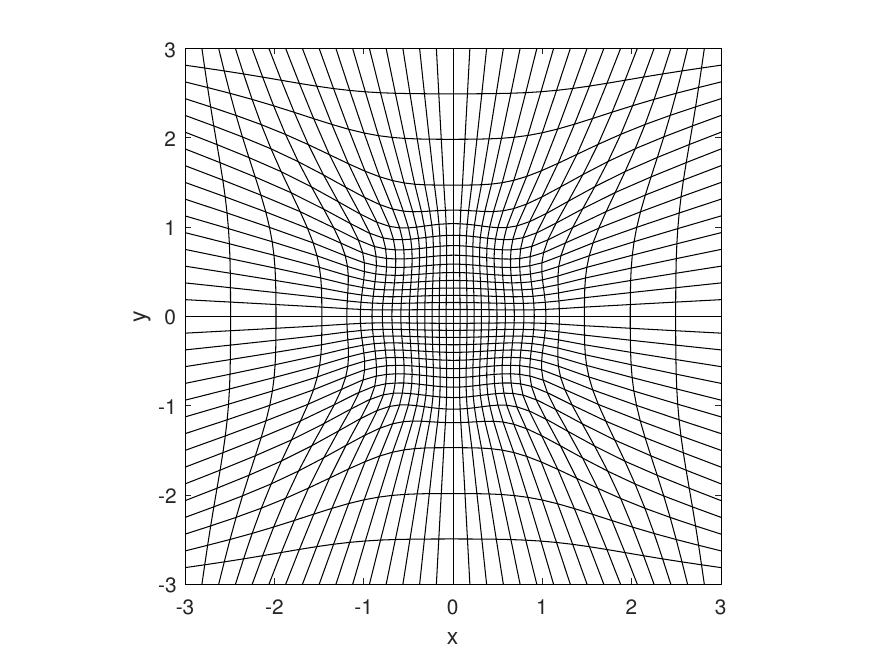}
	}
	\subfigure[Determinant value]{
		\includegraphics[width=0.45\textwidth]{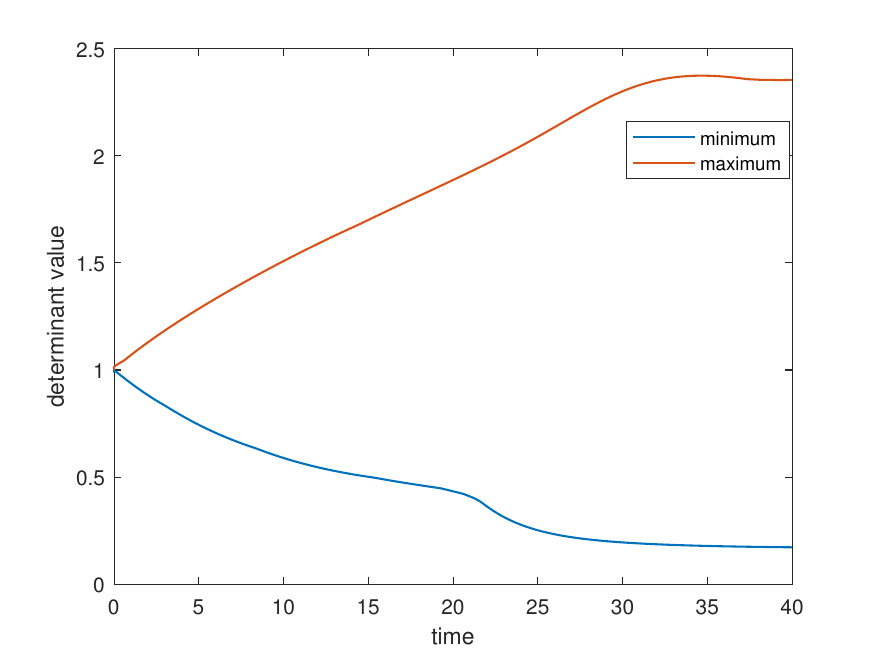}
	}
	\caption{Trajectory and determinant value for the aggregation diffusion equation  solved by \eqref{ex:1}-\eqref{ex:2} with  the regularization term $\epsilon\Delta_{\bm X}{\bm x}^{k+1}$,  $\epsilon=\delta t$, initial value \eqref{initial:aggre} in $[-3,3]\times[-3,3]$, $m=3$, $\nu=0.1$, $M_x=M_y=32$, $\delta t=0.01$. }\label{fig:aggre_diffu,det}  
\end{figure}
 We also simulate the evolution of the solution for the Keller-Segel equation, which is the aggregation-diffusion equation \eqref{eq:agg} with the kernel $W({\bm x})=\frac{1}{2\pi}\ln(|{\bm x}|)$  for $\nu=1$, $m=1$ and $m=2$, the global existence and blow-up of solutions are displayed. 
Taking the initial value \eqref{initial:ks2d}, 
and the constant $C_{2d}$ will be chosen as $1$ and $20$ in the following numerical experiments. The numerical solution is solved by scheme \eqref{ex:1}-\eqref{ex:2}  with regularization term $\epsilon\Delta_{\bm X}{\bm x}^{k+1}$, $\epsilon=0.1\delta t$. 

For the case when $m=1$, the numerical solution can be found in Figure~\ref{fig:ks2d,1}. The solution decays to zero as time increases, given $C_{2d}=1$. Conversely, when $C_{2d}=20$, the solution becomes sharply peaked at the origin, which can be regarded as the blow-up phenomenon. 
As for the case when $m=2$, we can observe from Figure~\ref{fig:ks2d,2} that the solution converges to a stable state denoted by a single bump, provided $C_{2d}=20$. 
Figure~\ref{fig:ks2d} illustrates the determinant value plots for $m=1$ and $m=2$. It can be noticed that the determinant values remain positive for the given  time.  However, since the particles are clustered at the center for $m=1$ and at the circumference for $m=2$, the particle trajectories will actually be distorted or exchanged as time increases, and the positive value of the determinant will not always be maintained.

\begin{figure}[!htb]
	\centering
	\subfigure[Density with $C_{2d}=1$ at $t=0.25$ ]{
		\includegraphics[width=0.45\textwidth]{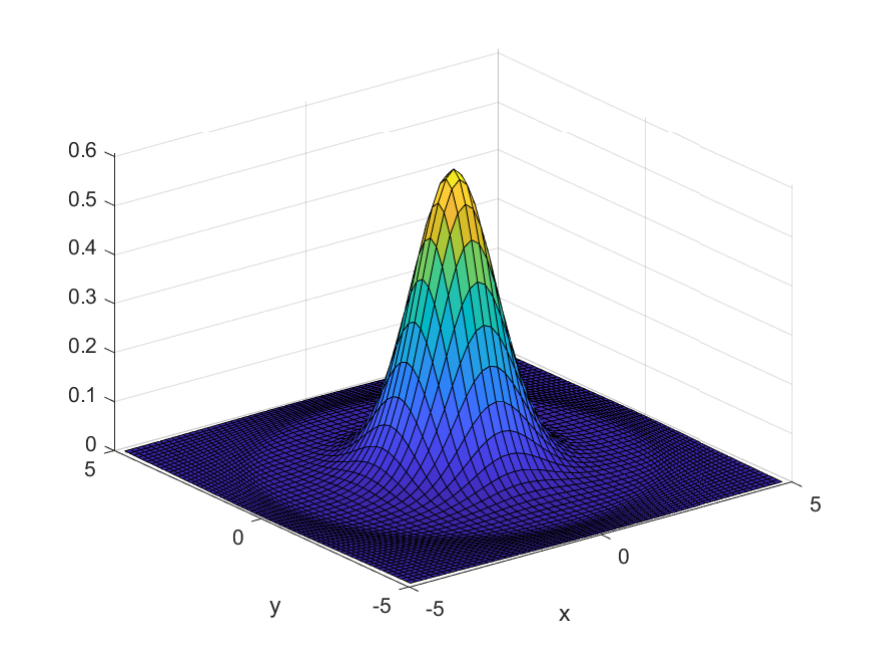}
	}
	\subfigure[Density with $C_{2d}=20$ at $t=0.25$  ]{
		\includegraphics[width=0.45\textwidth]{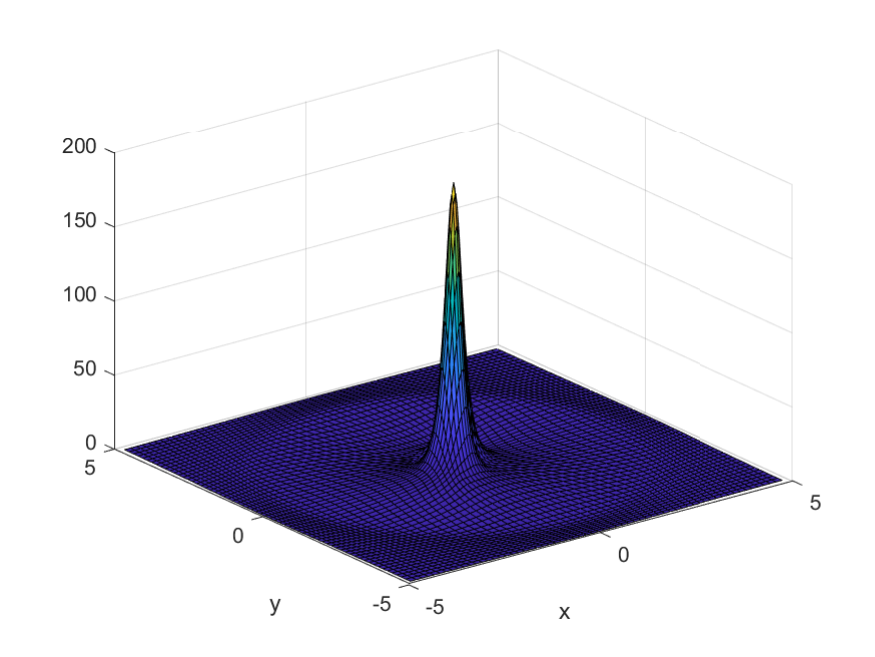}
	}
	\subfigure[Trajectory with $C_{2d}=1$ at $t=0.25$ ]{
	\includegraphics[width=0.45\textwidth]{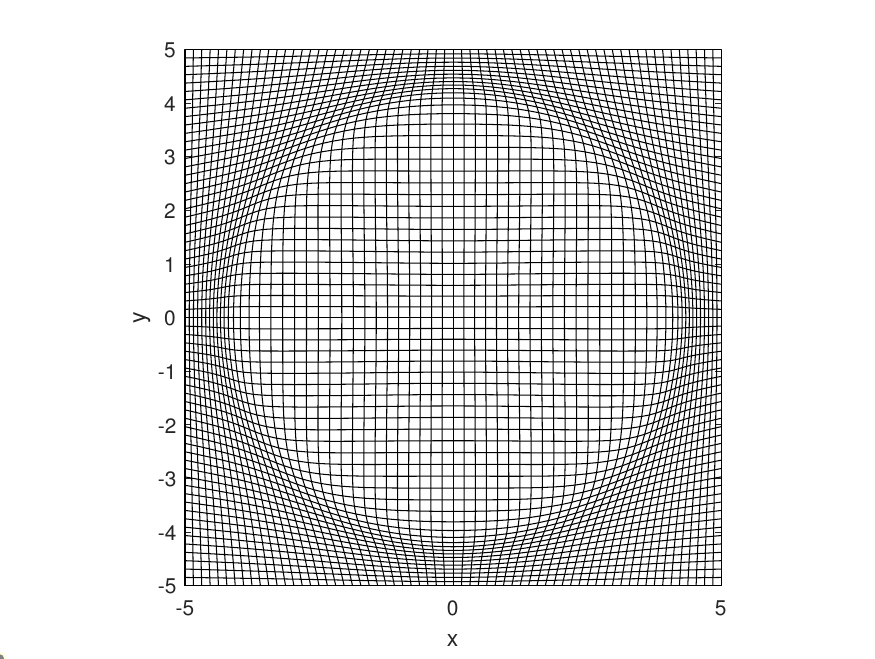}
}
\subfigure[Trajectory with $C_{2d}=20$ at $t=0.25$  ]{
	\includegraphics[width=0.45\textwidth]{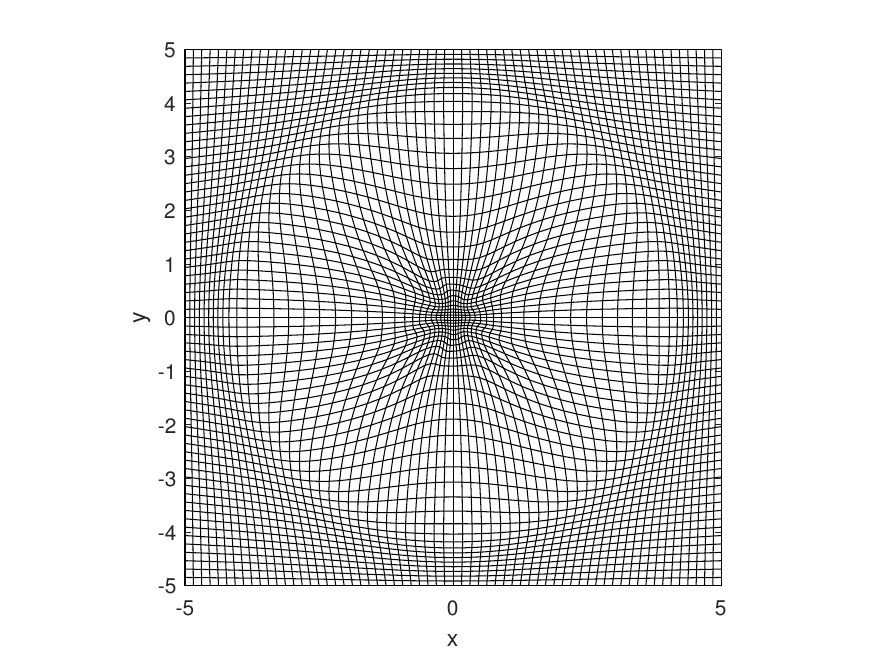}
}
	\caption{Keller-Segel model solved by \eqref{ex:1}-\eqref{ex:2} with the regularization term $\epsilon\Delta_{\bm X}{\bm x}^{k+1}$,  $\epsilon=0.1\delta t$, initial value \eqref{initial:ks2d}, $m=1$, $M_x=M_y=64$, $\delta t=0.001$.}\label{fig:ks2d,1}
\end{figure}
\begin{figure}[!htb]
	\centering
		\subfigure[Density with $C_{2d}=20$ at $t=0.12$ ]{
		\includegraphics[width=0.45\textwidth]{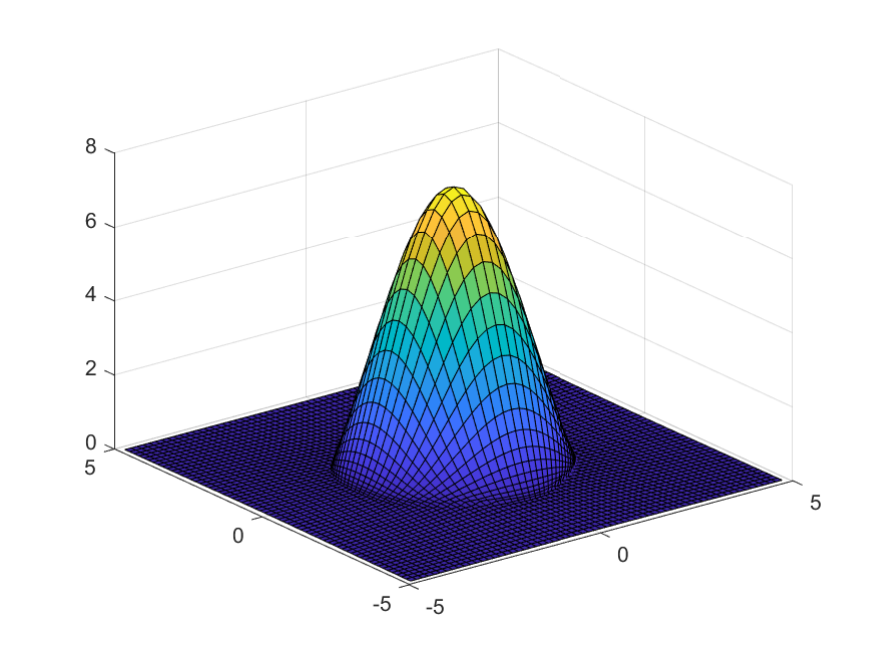}
	}
\subfigure[Trajectory with $C_{2d}=20$ at $t=0.12$ ]{
	\includegraphics[width=0.45\textwidth]{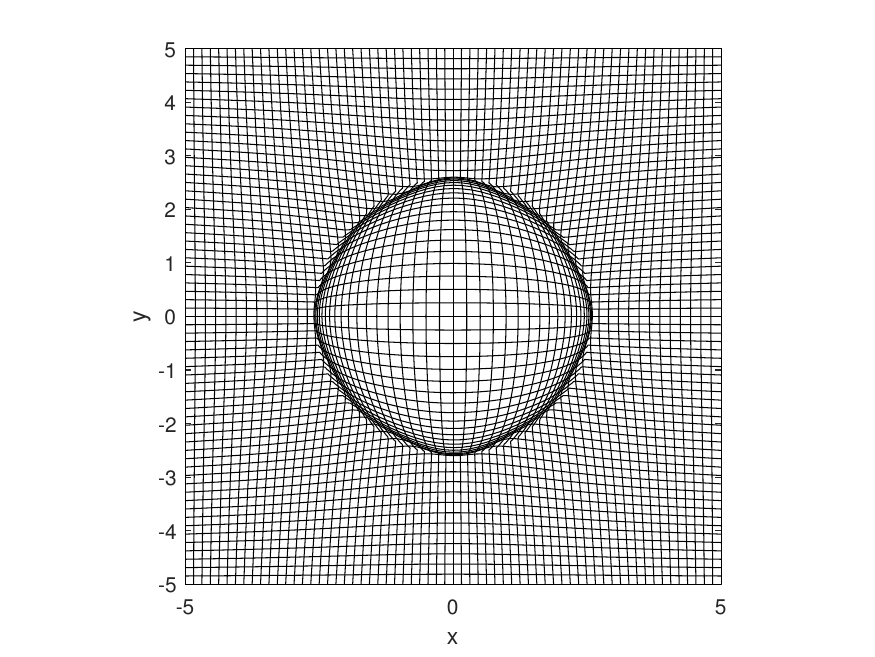}
}
	\caption{Keller-Segel model solved by \eqref{ex:1}-\eqref{ex:2} with the regularization term $\epsilon\Delta_{\bm X}{\bm x}^{k+1}$,  $\epsilon=0.1\delta t$, initial value \eqref{initial:ks2d} at $t=0.12$, $m=2$, $M_x=M_y=64$, $\delta t=0.001$.}\label{fig:ks2d,2}
\end{figure}

\begin{figure}[!htb]
	\centering
	\subfigure[$m=1$ with $C_{2d}=20$]{
		\includegraphics[width=0.45\textwidth]{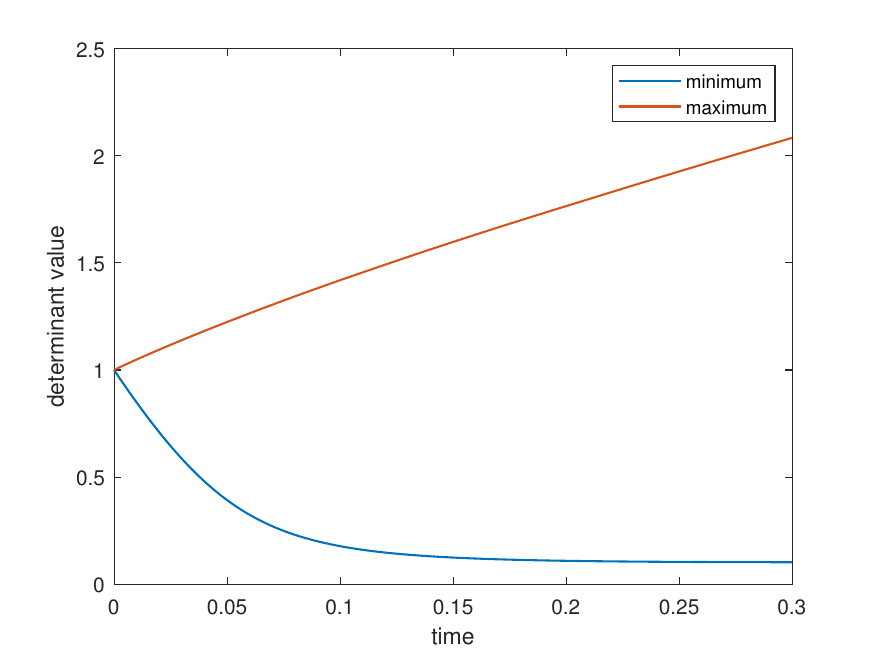}
	}
	\subfigure[$m=2$ with $C_{2d}=20$]{
		\includegraphics[width=0.45\textwidth]{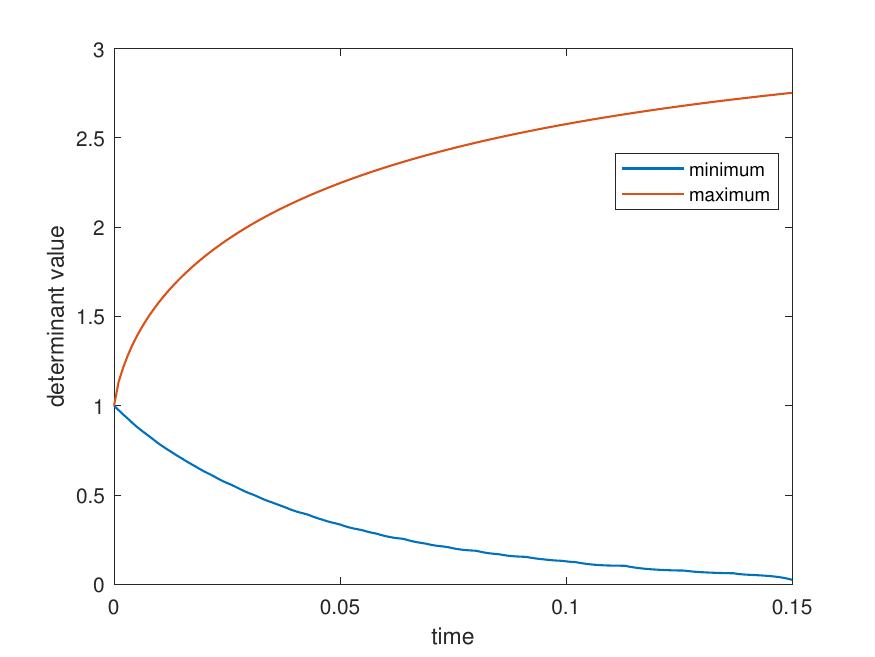}
	}
	\caption{Determinant value for the Keller-Segel model with $m=1,\ 2$,  solved by \eqref{ex:1}-\eqref{ex:2} with the regularization term $\epsilon\Delta_{\bm X}{\bm x}^{k+1}$, $\epsilon=0.1\delta t$, initial value \eqref{initial:ks2d}, $M_x=M_y=64$, $\delta t=0.001$.}\label{fig:ks2d}
\end{figure}

\section{Concluding remarks}

We constructed in this paper new numerical schemes for the Wasserstein gradient flows using a flow dynamic approach based on the Benamou-Bernier formula.  We showed that the new schemes preserve  essential structures of the Wasserstein gradient flows. More precisely, the fully discrete schemes are  shown to be positivity-preserving, mass conservative and energy dissipative. Moreover, it is shown that the schemes are uniquely solvable in the one dimensional case. 

We presented ample numerical experiments  to show that the proposed schemes are indeed positivity preserving, mass conservative and energy stable. Our numerical results also indicate  that the new schemes can capture   accurately the movement of the trajectory  and  the finite propagation speed for the Porous-Medium equation,  and can simulate blow-up phenomenon of the Keller-Segel equation.

\section*{Acknowledgements} 
Cheng is supported by NSFC 12301522, Liu is supported by NSFC 123B2015, Chen is supported by the NSFC 12071090 and NSFC 12241101, and Shen is supported by NSFC 12371409.

\appendix
\section{Appendix}
In this section, some details about the numerical experiments are presented.
\subsection{Porous-Medium and Fokker-Planck equations in 1D}
For the Porous-Medium problem with free boundaries, the following boundary condition can be obtained by using the fact that $\rho|_{\partial\Omega}=0$ as discussed in \cite{duan2019pme}:
\begin{align}
(\partial_Xx)^{m-1}\partial_tx=-\frac{m}{m-1}\frac{\partial_X(\rho(X,0))^{m-1}}{\partial_Xx}.
\end{align}
The numerical free boundary conditions are proposed as follows:
\begin{align}
&\left(\frac{x_1^k-x_0^k}{\delta X}\right)^{m-1}\frac{x_0^{k+1}-x_0^{k}}{\delta t}=-\frac{m}{m-1}\frac{\frac{(\rho(X_1,0))^{m-1}-(\rho(X_0,0))^{m-1}}{\delta X}}{\frac{x_1^{k+1}-x_0^{k+1}}{\delta X}},\label{appendix:pme_boundary1}\\
&	\left(\frac{x_{N}^k-x_{N-1}^k}{\delta X}\right)^{m-1}\frac{x_{N}^{k+1}-x_{N}^{k}}{\delta t}=-\frac{m}{m-1}\frac{\frac{(\rho(X_{N},0))^{m-1}-(\rho(X_{N-1},0))^{m-1}}{\delta X}}{\frac{x_{N}^{k+1}-x_{N-1}^{k+1}}{\delta X}}.\label{appendix:pme_boundary2}
\end{align}
The free boundaries problem for the Porous-Medium equation in the numerical experiments is solved with above free boundary conditions.

Similarly, for the Fokker-Planck equation with free boundaries, the following boundary condition can be obtained by using the fact that $\rho|_{\partial\Omega}=0$:
\begin{equation}\label{fp:free boundary}
\partial_tx=-\frac{m}{m-1}\frac{\partial_X(\rho(X,0))^{m-1}}{(\partial_Xx)^{m}}-V'(x),
\end{equation}
which can be solved by equations analogous to \eqref{appendix:pme_boundary1}-\eqref{appendix:pme_boundary2}.
\subsection{Aggregation equation in 1D}
In the subsection, we give details about the numerical experiments for the aggregation equation in 1D.

{\bf Implicit-explicit}. Let us set $x$ to be implicit and $y$ to be explicit in  scheme \eqref{schem:1}-\eqref{schem:1-2},  define the following discrete energy:
\begin{align}
&E_h({\bm x}^{k+1})=\delta X\sum_{i=0}^{N-1}\rho(X_{i+\frac{1}{2}},0)\sum_{j=0}^{N-1}\rho_{j+\frac{1}{2}}^k\int_{x_j^k}^{x_{j+1}^k}\left(\frac{|x_{i+\frac{1}{2}}^{k+1}-y|^2}{2}-\ln(|x_{i+\frac{1}{2}}^{k+1}-y|)\right)\mathrm{d}y,\nonumber\\
=&\delta X\sum_{i=0}^{N-1}\rho(X_{i+\frac{1}{2}},0)\sum_{j=0}^{N-1}\rho_{j+\frac{1}{2}}^k\left(-\frac{1}{6}(x_{i+\frac{1}{2}}^{k+1}-y)^3+(x_{i+\frac{1}{2}}^{k+1}-y)\ln(|x_{i+\frac{1}{2}}^{k+1}-y|)-(x_{i+\frac{1}{2}}^{k+1}-y)\right)\Big|_{x_j^k}^{x_{j+1}^k}\nonumber\\
=&\delta X\sum_{i=0}^{N-1}\rho(X_{i+\frac{1}{2}},0)\sum_{j=0}^{N-1}\rho_{j+\frac{1}{2}}^k\left(-\frac{1}{6}(x_{i+\frac{1}{2}}^{k+1}-x_{j+1}^k)^3+\frac{1}{6}(x_{i+\frac{1}{2}}^{k+1}-x_{j}^k)^3+x_{j+1}^k-x_j^k\right)\nonumber\\
&+\delta X\sum_{i=0}^{N-1}\rho(X_{i+\frac{1}{2}},0)\sum_{j=0}^{N-1}\rho_{j+\frac{1}{2}}^k\left((x_{i+\frac{1}{2}}^{k+1}-x_{j+1}^k)\ln(|x_{i+\frac{1}{2}}^{k+1}-x_{j+1}^k|)-(x_{i+\frac{1}{2}}^{k+1}-x_{j}^k)\ln(|x_{i+\frac{1}{2}}^{k+1}-x_{j}^k|)\right),\label{eq:energy}
\end{align}
and the following result can be obtained by simple calculations:
\begin{align}
\frac{\delta E_h^{k+1}}{\delta x_i^{k+1}}=&\delta X\rho(X_{i+\frac{1}{2}},0)\sum_{j=0}^{N-1}\rho_{j+\frac{1}{2}}^k\left(-\frac{1}{4}(x_{i+\frac{1}{2}}^{k+1}-x_{j+1}^k)^2+\frac{1}{4}(x_{i+\frac{1}{2}}^{k+1}-x_{j}^k)^2\right)\nonumber\\
&+\delta X\rho(X_{i-\frac{1}{2}},0)\sum_{j=0}^{N-1}\rho_{j+\frac{1}{2}}^k\left(-\frac{1}{4}(x_{i-\frac{1}{2}}^{k+1}-x_{j+1}^k)^2+\frac{1}{4}(x_{i-\frac{1}{2}}^{k+1}-x_{j}^k)^2\right)\nonumber\\
&+\delta X\rho(X_{i+\frac{1}{2}},0)\sum_{j=0}^{N-1}\rho_{j+\frac{1}{2}}^k\left(\frac{1}{2}\ln(|x_{i+\frac{1}{2}}^{k+1}-x_{j+1}^k|)-\frac{1}{2}\ln(|x_{i+\frac{1}{2}}^{k+1}-x_{j}^k|)\right)\nonumber\\
&+\delta X\rho(X_{i-\frac{1}{2}},0)\sum_{j=0}^{N-1}\rho_{j+\frac{1}{2}}^k\left(\frac{1}{2}\ln(|x_{i-\frac{1}{2}}^{k+1}-x_{j+1}^k|)-\frac{1}{2}\ln(|x_{i-\frac{1}{2}}^{k+1}-x_{j}^k|)\right).\label{appendix:aggre_xi_ye}
\end{align}
  One can also find that
	\begin{align}
	\frac{\delta^2 E_h^{k+1}}{\delta (x_i^{k+1})^2}=&\delta X(\rho(X_{i+\frac{1}{2}},0)+\rho(X_{i-\frac{1}{2}},0))\sum_{j=0}^{N-1}\frac{\rho_{j+\frac{1}{2}}^k}{4}(x_{j+1}^{k}-x_{j}^k)\nonumber\\
	&+\delta X\rho(X_{i+\frac{1}{2}},0)\sum_{j=0}^{N-1}\frac{\rho_{j+\frac{1}{2}}^k}{4}\left(\frac{1}{x_{i+\frac{1}{2}}^{k+1}-x_{j+1}^k}-\frac{1}{x_{i+\frac{1}{2}}^{k+1}-x_{j}^k}\right)\nonumber\\
	&+\delta X\rho(X_{i-\frac{1}{2}},0)\sum_{j=0}^{N-1}\frac{\rho_{j+\frac{1}{2}}^k}{4}\left(\frac{1}{x_{i-\frac{1}{2}}^{k+1}-x_{j+1}^k}-\frac{1}{x_{i-\frac{1}{2}}^{k+1}-x_{j}^k}\right)>0,\nonumber
	\end{align}
	where the fact that $\frac{1}{x-y}|_{a}^{b}=\int_{a}^{b}\frac{1}{(x-y)^2}\mathrm{d}y>0$ for $a<b$ has been utilized in the last inequality. Then the discrete energy is convex since the Hessian matrix $\nabla^2E_h$ is positive definite.

{\bf Explicit-implicit}. If we set $x$ to be explicit and $y$ to be implicit in scheme \eqref{schem:1}-\eqref{schem:1-2}, i.e.\,
\begin{align}
\frac{\delta E_h^{k+1}}{\delta x_i^{k+1}}=&\delta X\rho(X_{i+\frac{1}{2}},0)\sum_{j=0}^{N-1}\rho_{j+\frac{1}{2}}^k\left(-\frac{1}{4}(x_{i+\frac{1}{2}}^{k}-x_{j+1}^{k+1})^2+\frac{1}{4}(x_{i+\frac{1}{2}}^{k}-x_{j}^{k+1})^2\right)\nonumber\\
&+\delta X\rho(X_{i-\frac{1}{2}},0)\sum_{j=0}^{N-1}\rho_{j+\frac{1}{2}}^k\left(-\frac{1}{4}(x_{i-\frac{1}{2}}^{k}-x_{j+1}^{k+1})^2+\frac{1}{4}(x_{i-\frac{1}{2}}^{k}-x_{j}^{k+1})^2\right)\nonumber\\
&+\delta X\rho(X_{i+\frac{1}{2}},0)\sum_{j=0}^{N-1}\rho_{j+\frac{1}{2}}^k\left(\frac{1}{2}\ln(|x_{i+\frac{1}{2}}^{k}-x_{j+1}^{k+1}|)-\frac{1}{2}\ln(|x_{i+\frac{1}{2}}^{k}-x_{j}^{k+1}|)\right)\nonumber\\
&+\delta X\rho(X_{i-\frac{1}{2}},0)\sum_{j=0}^{N-1}\rho_{j+\frac{1}{2}}^k\left(\frac{1}{2}\ln(|x_{i-\frac{1}{2}}^{k}-x_{j+1}^{k+1}|)-\frac{1}{2}\ln(|x_{i-\frac{1}{2}}^{k}-x_{j}^{k+1}|)\right).\label{appendix:aggre_xe_yi}
\end{align}

{\bf Implicit-implicit}. If we take both $x$ and $y$ implicit in scheme \eqref{schem:1}-\eqref{schem:1-2}, set $\frac{\delta E_h^{k+1}}{\delta x_i^{k+1}}$ as follows: 
\begin{align}
\frac{\delta E_h^{k+1}}{\delta x_i^{k+1}}=&\delta X\rho(X_{i+\frac{1}{2}},0)\sum_{j=0}^{N-1}\rho_{j+\frac{1}{2}}^k\left(-\frac{1}{4}(x_{i+\frac{1}{2}}^{k+1}-x_{j+1}^{k+1})^2+\frac{1}{4}(x_{i+\frac{1}{2}}^{k+1}-x_{j}^{k+1})^2\right)\nonumber\\
&+\delta X\rho(X_{i-\frac{1}{2}},0)\sum_{j=0}^{N-1}\rho_{j+\frac{1}{2}}^k\left(-\frac{1}{4}(x_{i-\frac{1}{2}}^{k+1}-x_{j+1}^{k+1})^2+\frac{1}{4}(x_{i-\frac{1}{2}}^{k+1}-x_{j}^{k+1})^2\right)\nonumber\\
&+\delta X\rho(X_{i+\frac{1}{2}},0)\sum_{j=0}^{N-1}\rho_{j+\frac{1}{2}}^k\left(\frac{1}{2}\ln(|x_{i+\frac{1}{2}}^{k+1}-x_{j+1}^{k+1}|)-\frac{1}{2}\ln(|x_{i+\frac{1}{2}}^{k+1}-x_{j}^{k+1}|)\right)\nonumber\\
&+\delta X\rho(X_{i-\frac{1}{2}},0)\sum_{j=0}^{N-1}\rho_{j+\frac{1}{2}}^k\left(\frac{1}{2}\ln(|x_{i-\frac{1}{2}}^{k+1}-x_{j+1}^{k+1}|)-\frac{1}{2}\ln(|x_{i-\frac{1}{2}}^{k+1}-x_{j}^{k+1}|)\right).\label{appendix:aggre_xi_yi}
\end{align}

\subsection{Keller-Segel model in 1D}
The discrete energy for the Keller-Segel model in 1D is defined as follows:
\begin{align}
	E_h({\bm x}^{k+1})=&\sum_{i=0}^{N-1}\delta X\rho(X_{i+\frac{1}{2}},0)\log\left(\frac{\rho(X_{i+\frac{1}{2}},0)}{\frac{x_{i+1}^{k+1}-x_i^{k+1}}{\delta X}}\right)-\frac{\delta X}{2\pi}\sum_{i=0}^{N-1}\rho(X_{i+\frac{1}{2}},0)\nonumber\\
	\times\sum_{j=0}^{N-1}&\rho_{j+\frac{1}{2}}^k\left((x_{i+\frac{1}{2}}^{k+1}-x_{j+1}^k)\ln(|x_{i+\frac{1}{2}}^{k+1}-x_{j+1}^k|)-(x_{i+\frac{1}{2}}^{k+1}-x_{j}^k)\ln(|x_{i+\frac{1}{2}}^{k+1}-x_{j}^k|)\right),\nonumber
	\end{align}
it can be calculated that
\begin{align}
\frac{\delta E_h^{k+1}}{\delta x_i^{k+1}}=
&-\frac{\delta X}{2\pi}\rho(X_{i+\frac{1}{2}},0)\sum_{j=0}^{N-1}\rho_{j+\frac{1}{2}}^k\left(\frac{1}{2}\ln(|x_{i+\frac{1}{2}}^{k+1}-x_{j+1}^k|)-\frac{1}{2}\ln(|x_{i+\frac{1}{2}}^{k+1}-x_{j}^k|)\right)\nonumber\\
&-\frac{\delta X}{2\pi}\rho(X_{i-\frac{1}{2}},0)\sum_{j=0}^{N-1}\rho_{j+\frac{1}{2}}^k\left(\frac{1}{2}\ln(|x_{i-\frac{1}{2}}^{k+1}-x_{j+1}^k|)-\frac{1}{2}\ln(|x_{i-\frac{1}{2}}^{k+1}-x_{j}^k|)\right)\nonumber\\
&+\frac{\delta X\rho(X_{i+\frac{1}{2}},0)}{x_{i+1}^{k+1}-x_{i}^{k+1}}-\frac{\delta X\rho(X_{i-\frac{1}{2}},0)}{x_{i}^{k+1}-x_{i-1}^{k+1}},\label{appendix:ks}
\end{align}
 and we have
	\begin{align}
	\frac{\delta^2 E_h}{\delta (x_i^{k+1})^2}=
	&-\frac{\delta X}{2\pi}\rho(X_{i+\frac{1}{2}},0)\sum_{j=0}^{N-1}\frac{\rho_{j+\frac{1}{2}}^k}{4}\left(\frac{1}{x_{i+\frac{1}{2}}^{k+1}-x_{j+1}^k}-\frac{1}{x_{i+\frac{1}{2}}^{k+1}-x_{j}^k}\right)\nonumber\\
	&-\frac{\delta X}{2\pi}\rho(X_{i-\frac{1}{2}},0)\sum_{j=0}^{N-1}\frac{\rho_{j+\frac{1}{2}}^k}{4}\left(\frac{1}{x_{i-\frac{1}{2}}^{k+1}-x_{j+1}^k}-\frac{1}{x_{i-\frac{1}{2}}^{k+1}-x_{j}^k}\right)\nonumber\\
	&+\frac{\delta X\rho(X_{i+\frac{1}{2}},0)}{(x_{i+1}^{k+1}-x_{i}^{k+1})^2}+\frac{\delta X\rho(X_{i-\frac{1}{2}},0)}{(x_{i}^{k+1}-x_{i-1}^{k+1})^2}.
	\end{align}
	Notice that
	\begin{align}
	-\sum_{j=0}^{N-1}\rho_{j+\frac{1}{2}}^k\left(\frac{1}{x_{i+\frac{1}{2}}^{k+1}-x_{j+1}^k}-\frac{1}{x_{i+\frac{1}{2}}^{k+1}-x_{j}^k}\right)=\sum_{j=0}^{N-1}\rho_{j+\frac{1}{2}}^k\int_{x_j^k}^{x_{j+1}^{k}}\frac{-1}{(x_{i+\frac{1}{2}}^{k+1}-y)^2}\mathrm{d}y<0,\\
	-\sum_{j=0}^{N-1}\rho_{j+\frac{1}{2}}^k\left(\frac{1}{x_{i-\frac{1}{2}}^{k+1}-x_{j+1}^k}-\frac{1}{x_{i-\frac{1}{2}}^{k+1}-x_{j}^k}\right)=\sum_{j=0}^{N-1}\rho_{j+\frac{1}{2}}^k\int_{x_j^k}^{x_{j+1}^{k}}\frac{-1}{(x_{i-\frac{1}{2}}^{k+1}-y)^2}\mathrm{d}y<0.
	\end{align}
	However, the positivity or negativity of  $\frac{\delta^2 E_h}{\delta (x_i^{k+1})^2}$ is not determined.
	Whether the discrete energy is convex or not is uncertain,  since the positivity or negativity of the Hessian matrix $\nabla^2E_h$ is not clear.

	\bibliographystyle{plain}
	
\bibliography{ref.bib}

\begin{thebibliography}{10}

\bibitem{ambrosio2005gradient}
Luigi Ambrosio, Nicola Gigli, and Giuseppe Savar{\'e}.
\newblock {\em Gradient flows: in metric spaces and in the space of probability
  measures}.
\newblock Springer Science \& Business Media, 2005.

\bibitem{aronson1983initially}
DG~Aronson, LA~Caffarelli, and S~Kamin.
\newblock How an initially stationary interface begins to move in porous medium
  flow.
\newblock {\em SIAM Journal on Mathematical Analysis}, 14(4):639--658, 1983.

\bibitem{aronson2006porous}
Donald~G Aronson.
\newblock The porous medium equation.
\newblock {\em Nonlinear Diffusion Problems: Lectures given at the 2nd 1985
  Session of the Centro Internazionale Matermatico Estivo (CIME) held at
  Montecatini Terme, Italy June 10--June 18, 1985}, pages 1--46, 2006.

\bibitem{benamou2000computational}
Jean-David Benamou and Yann Brenier.
\newblock A computational fluid mechanics solution to the {M}onge-{K}antorovich
  mass transfer problem.
\newblock {\em Numerische Mathematik}, 84(3):375--393, 2000.

\bibitem{benamou2016augmented}
Jean-David Benamou, Guillaume Carlier, and Maxime Laborde.
\newblock An augmented {L}agrangian approach to {W}asserstein gradient flows
  and applications.
\newblock {\em ESAIM: Proceedings and Surveys}, 54:1--17, 2016.

\bibitem{bonet2022efficient}
Cl{\'e}ment Bonet, Nicolas Courty, Fran{\c{c}}ois Septier, and Lucas Drumetz.
\newblock Efficient gradient flows in sliced-{W}asserstein space.
\newblock {\em Transactions on Machine Learning Research}, 2022.

\bibitem{cances2020variational}
Cl{\'e}ment Cances, Thomas~O Gallou{\"e}t, and Gabriele Todeschi.
\newblock A variational finite volume scheme for {W}asserstein gradient flows.
\newblock {\em Numerische Mathematik}, 146:437--480, 2020.

\bibitem{carrillo2015finite}
Jos{\'e}~A Carrillo, Alina Chertock, and Yanghong Huang.
\newblock A finite-volume method for nonlinear nonlocal equations with a
  gradient flow structure.
\newblock {\em Communications in Computational Physics}, 17(1):233--258, 2015.

\bibitem{carrillo2022primal}
Jos{\'e}~A Carrillo, Katy Craig, Li~Wang, and Chaozhen Wei.
\newblock Primal dual methods for {W}asserstein gradient flows.
\newblock {\em Foundations of Computational Mathematics}, 22:1--55, 2022.

\bibitem{carrillo2018lagrangian}
Jos{\'e}~A Carrillo, Bertram D{\"u}ring, Daniel Matthes, and David~S McCormick.
\newblock A lagrangian scheme for the solution of nonlinear diffusion equations
  using moving simplex meshes.
\newblock {\em Journal of Scientific Computing}, 75:1463--1499, 2018.

\bibitem{carrillo2012mass}
Jos{\'e}~A Carrillo, Lucas~CF Ferreira, and Juliana~C Precioso.
\newblock A mass-transportation approach to a one dimensional fluid mechanics
  model with nonlocal velocity.
\newblock {\em Advances in Mathematics}, 231(1):306--327, 2012.

\bibitem{carrillo2001entropy}
Jos{\'e}~A Carrillo, Ansgar J{\"u}ngel, Peter~A Markowich, Giuseppe Toscani,
  and Andreas Unterreiter.
\newblock Entropy dissipation methods for degenerate parabolicproblems and
  generalized sobolev inequalities.
\newblock {\em Monatshefte f{\"u}r Mathematik}, 133:1--82, 2001.

\bibitem{carrillo2021lagrangian}
Jose~A Carrillo, Daniel Matthes, and Marie-Therese Wolfram.
\newblock Lagrangian schemes for {W}asserstein gradient flows.
\newblock {\em Handbook of Numerical Analysis}, 22:271--311, 2021.

\bibitem{carrillo2003kinetic}
Jos{\'e}~A Carrillo, Robert~J McCann, and C{\'e}dric Villani.
\newblock Kinetic equilibration rates for granular media and related equations:
  entropy dissipation and mass transportation estimates.
\newblock {\em Revista Matematica Iberoamericana}, 19(3):971--1018, 2003.

\bibitem{carrillo2000asymptotic}
Jos{\'e}~A Carrillo and Giuseppe Toscani.
\newblock Asymptotic {L}\textsuperscript{1}-decay of solutions of the porous
  medium equation to self-similarity.
\newblock {\em Indiana University Mathematics Journal}, 49:113--142, 2000.

\bibitem{chen2022error}
Wenbin Chen, Qianqian Liu, and Jie Shen.
\newblock Error estimates and blow-up analysis of a finite-element
  approximation for the parabolic-elliptic {K}eller-{S}egel system.
\newblock {\em International Journal of Numerical Analysis \& Modeling},
  19:275--298, 2022.

\bibitem{cheng2020new}
Qing Cheng, Chun Liu, and Jie Shen.
\newblock A new interface capturing method for {A}llen-{C}ahn type equations
  based on a flow dynamic approach in {L}agrangian coordinates, {I}.
  {O}ne-dimensional case.
\newblock {\em Journal of Computational Physics}, 419:109509, 2020.

\bibitem{duan2021structure}
Chenghua Duan, Wenbin Chen, Chun Liu, Xingye Yue, and Shenggao Zhou.
\newblock Structure-preserving numerical methods for nonlinear
  {F}okker--{P}lanck equations with nonlocal interactions by an energetic
  variational approach.
\newblock {\em SIAM Journal on Scientific Computing}, 43(1):B82--B107, 2021.

\bibitem{duan2019pme}
Chenghua Duan, Chun Liu, Cheng Wang, and Xingye Yue.
\newblock Numerical methods for porous medium equation by an energetic
  variational approach.
\newblock {\em Journal of Computational Physics}, 385:13--32, 2019.

\bibitem{eisenberg1996computing}
Bob Eisenberg.
\newblock Computing the field in proteins and channels.
\newblock {\em The Journal of Membrane Biology}, 150(1):1--25, 1996.

\bibitem{eisenberg2010energy}
Bob Eisenberg, Yunkyong Hyon, and Chun Liu.
\newblock Energy variational analysis of ions in water and channels: {F}ield
  theory for primitive models of complex ionic fluids.
\newblock {\em The Journal of Chemical Physics}, 133(10), 2010.

\bibitem{fetecau2011swarm}
Razvan~C Fetecau, Yanghong Huang, and Theodore Kolokolnikov.
\newblock Swarm dynamics and equilibria for a nonlocal aggregation model.
\newblock {\em Nonlinearity}, 24(10):2681, 2011.

\bibitem{filbet2006finite}
Francis Filbet.
\newblock A finite volume scheme for the {P}atlak--{K}eller--{S}egel chemotaxis
  model.
\newblock {\em Numerische Mathematik}, 104:457--488, 2006.

\bibitem{horstmann20031970}
Dirk Horstmann.
\newblock {From 1970 until present\,:\,the {K}eller-{S}egel model in chemotaxis
  and its consequences {I}.}
\newblock {\em Jahresbericht der Deutschen Mathematiker-Vereinigung},
  105(3):103--165, 2003.

\bibitem{jordan1998variational}
Richard Jordan, David Kinderlehrer, and Felix Otto.
\newblock The variational formulation of the {F}okker--{P}lanck equation.
\newblock {\em SIAM Journal on Mathematical Analysis}, 29(1):1--17, 1998.

\bibitem{keller1970initiation}
Evelyn~F Keller and Lee~A Segel.
\newblock Initiation of slime mold aggregation viewed as an instability.
\newblock {\em Journal of Theoretical Biology}, 26(3):399--415, 1970.

\bibitem{kinderlehrer2017wasserstein}
David Kinderlehrer, L{\'e}onard Monsaingeon, and Xiang Xu.
\newblock A wasserstein gradient flow approach to {P}oisson- {N}ernst- {P}lanck
  equations.
\newblock {\em ESAIM: Control, Optimisation and Calculus of Variations},
  23(1):137--164, 2017.

\bibitem{li2020fisher}
Wuchen Li, Jianfeng Lu, and Li~Wang.
\newblock Fisher information regularization schemes for {W}asserstein gradient
  flows.
\newblock {\em Journal of Computational Physics}, 416:109449, 2020.

\bibitem{liu2020lagrangian}
Chun Liu and Yiwei Wang.
\newblock On lagrangian schemes for porous medium type generalized diffusion
  equations: A discrete energetic variational approach.
\newblock {\em Journal of Computational Physics}, 417:109566, 2020.

\bibitem{liu2020variational}
Chun Liu and Yiwei Wang.
\newblock A variational {L}agrangian scheme for a phase-field model: {A}
  discrete energetic variational approach.
\newblock {\em SIAM Journal on Scientific Computing}, 42(6):B1541--B1569, 2020.

\bibitem{liu2019energetic}
Chun Liu and Hao Wu.
\newblock An energetic variational approach for the {C}ahn--{H}illiard equation
  with dynamic boundary condition: model derivation and mathematical analysis.
\newblock {\em Archive for Rational Mechanics and Analysis}, 233(1):167--247,
  2019.

\bibitem{liu2004numerical}
Fawang Liu, Vo~Anh, and Ian Turner.
\newblock Numerical solution of the space fractional {F}okker--{P}lanck
  equation.
\newblock {\em Journal of Computational and Applied Mathematics},
  166(1):209--219, 2004.

\bibitem{liu2023dynamic}
Hailiang Liu and Wumaier Maimaitiyiming.
\newblock A dynamic mass transport method for {P}oisson-{N}ernst-{P}lanck
  equations.
\newblock {\em Journal of Computational Physics}, 473:111699, 2023.

\bibitem{liu2016entropy}
Hailiang Liu and Zhongming Wang.
\newblock An entropy satisfying discontinuous galerkin method for nonlinear
  {F}okker--{P}lanck equations.
\newblock {\em Journal of Scientific Computing}, 68:1217--1240, 2016.

\bibitem{liu2018positivity}
Jianguo Liu, Li~Wang, and Zhennan Zhou.
\newblock Positivity-preserving and asymptotic preserving method for 2{D}
  {K}eller-{S}egal equations.
\newblock {\em Mathematics of Computation}, 87(311):1165--1189, 2018.

\bibitem{liu2023envara}
Qianqian Liu, Chenghua Duan, and Wenbin Chen.
\newblock {EnVarA-FEM} for the flux-limited porous medium equation.
\newblock {\em Journal of Computational Physics}, 493:112432, 2023.

\bibitem{ngo2017study}
Cuong Ngo and Weizhang Huang.
\newblock A study on moving mesh finite element solution of the porous medium
  equation.
\newblock {\em Journal of Computational Physics}, 331:357--380, 2017.

\bibitem{onsager1931reciprocal}
Lars Onsager.
\newblock Reciprocal relations in irreversible processes. {I}.
\newblock {\em Physical Review}, 37(4):405, 1931.

\bibitem{onsager1931reciprocal2}
Lars Onsager.
\newblock Reciprocal relations in irreversible processes. {II}.
\newblock {\em Physical Review}, 38(12):2265, 1931.

\bibitem{pareschi2018structure}
Lorenzo Pareschi and Mattia Zanella.
\newblock Structure preserving schemes for nonlinear {F}okker--{P}lanck
  equations and applications.
\newblock {\em Journal of Scientific Computing}, 74:1575--1600, 2018.

\bibitem{shen2020unconditionally}
Jie Shen and Jie Xu.
\newblock Unconditionally bound preserving and energy dissipative schemes for a
  class of {K}eller--{S}egel equations.
\newblock {\em SIAM Journal on Numerical Analysis}, 58(3):1674--1695, 2020.

\bibitem{rayleigh1873some}
John~William Strutt.
\newblock Some general theorems relating to vibrations.
\newblock {\em Proceedings of the London Mathematical Society}, 1(1):357--368,
  1871.

\bibitem{vazquez2007porous}
Juan~Luis V{\'a}zquez.
\newblock {\em The {P}orous {M}edium {E}quation: {M}athematical {T}heory}.
\newblock Oxford University Press on Demand, 2007.

\bibitem{wang2022fully}
Shufen Wang, Simin Zhou, Shuxun Shi, and Wenbin Chen.
\newblock Fully decoupled and energy stable {BDF} schemes for a class of
  {K}eller-{S}egel equations.
\newblock {\em Journal of Computational Physics}, 449:110799, 2022.

\bibitem{wang2022some}
Yiwei Wang and Chun Liu.
\newblock Some recent advances in energetic variational approaches.
\newblock {\em Entropy}, 24(5):721, 2022.

\bibitem{winkler2010aggregation}
Michael Winkler.
\newblock Aggregation vs. global diffusive behavior in the higher-dimensional
  {K}eller--{S}egel model.
\newblock {\em Journal of Differential Equations}, 248(12):2889--2905, 2010.

\bibitem{zhang2009numerical}
Qiang Zhang and Zilong Wu.
\newblock Numerical simulation for porous medium equation by local
  discontinuous {G}alerkin finite element method.
\newblock {\em Journal of Scientific Computing}, 38:127--148, 2009.

\end{thebibliography}

\end{document}